\documentclass[a4paper,twoside]{article}
\usepackage{latexsym,bm}
\usepackage{mathrsfs,amsmath,amssymb,amsthm,youngtab}
\usepackage[dvipdfm]{graphicx}

\makeatletter
\@addtoreset{equation}{section}
\makeatother
\usepackage{enumerate}

\def\wg{{\mathrm w}_g}

\def\Index{\mathrm {Index}}

\def\cI{{\mathcal{I}}}
\def\wdeg{{\mathrm{deg}_\mathrm{w}}}
\def\mwdeg{{\mathrm{deg}_{\mathrm{w}^{-1}}}}
\def\ldeg{{\mathrm{deg}_{\lambda}}}

\def\uab{ {w}}
\def\bs{{\bold{s}}}

\def\CC{{\mathbb C}}
\def\JJ{{\mathcal J}}
\def\cM{{\mathcal M}}
\def\cC{{\mathcal C}}
\def\ZZ{{\mathbb Z}}

\def\QQ{{\mathbb Q}}
\def\RR{{\mathbb R}}

\def\cO{{\mathcal{O}}}

\def\WW{{\mathcal{W}}}
\def\SSS{{\mathcal{S}}}
\def\ella{{{\ell_{\mathrm{a}}}}}

\def\nuI{{{\nu^{I}}}}
\def\hnuI{{\hat{\nu^{I}}}}
\def\nuII{{{\nu^{II}}}}

\def\eps{\epsilon}

\def\LA{\langle}
\def\RA{\rangle}

\newtheorem{definition}{Definition}[section]

\newtheorem{theorem}[definition]{Theorem}
\newtheorem{proposition}[definition]{Proposition}
\newtheorem{corollary}[definition]{Corollary}
\newtheorem{remark}[definition]{Remark}
\newtheorem{lemma}[definition]{Lemma}

\def\book#1{\rm{#1}, }
\def\paper#1{\textit{#1}, }
\def\jour#1{\rm{#1}, }
\def\yr#1{({\rm{#1}) }}
\def\vol#1{\textbf{#1}}
\def\pages#1{\rm{#1}}
\def\page#1{\rm{#1}}

\def\publ#1{\rm{#1}, }
\def\by#1{{\rm{#1}, }}

\begin{document}

\title{Jacobi inversion on strata of the Jacobian of the $C_{rs}$ curve 
 $y^r = f(x)$. II}

\author{Shigeki \textsc{Matsutani} and Emma \textsc{Previato}}



\maketitle

\abstract{Previous work by the authors
(this journal, \vol{60} (2008),
1009-1044) produced equations that hold on certain loci of the Jacobian of 
a cyclic $C_{rs}$ curve.  A curve of this type generalizes 
elliptic curves, and the equations in question are given in terms
of (Klein's) generalization of Weierstrass' $\sigma$-function.
The key tool is a matrix with entries that are
polynomial in the coordinates of the affine plane model of the curve, thus
can be expressed in terms of
$\sigma$ and its derivatives.
The key geometric loci on the Jacobian of the curve give a stratification
of Brill-Noether type. The results are of the type of
Riemann-Kempf singularity theorem, the methods are germane to those
used by J.D. Fay, who gave vanishing tables for Riemann's $\theta$-function
and its derivatives. 
The main objects we use were developed by several contemporary authors,
aside from the classical definitions:
 meromorphic differentials were
expressed in terms of the coordinates mainly by V.M. Buchstaber, 
J.C. Eilbeck, V.Z. Enolski, D.V. Leykin, and Taylor expansions 
for $\sigma$ in terms of Schur polynomials also contributed by A. Nakayashiki,
in terms of Sato's $\tau$-function.
Within this framework, following specific results for $\sigma$-derivatives
given by Y. \^Onishi, we arrive at our main results,
namely statements on the vanishing 
on given strata of the partial derivatives of $\sigma$ indexed 
by Young-diagrams subsets that can be worked out in terms
of the Weierstrass semigroup of the curve at its point at infinity.
The combinatorial statements
hold not only for Jacobians but 
for the stratification of Sato's infinite-dimensional
Grassmann manifold as well.}

{\ \ \ Key Words {sigma function, Schur functions,
$\Theta$-stratification of the Jacobian,
$C_{rs}$}

}

\newpage

\centerline{\textbf{Guide to Symbols}}

$$
\begin{array}{lr}
N(n) & \mbox{order of pole at $n$-th Weierstrass non-gap $(N(0)=0,
N(g) = 2g)$ p. \pageref{pg:N(n)}}\\
\phi_i & \mbox{monomial associated to the $n$-th Weierstrass non-gap  
p. \pageref{pg:phi_n}}\\
d_>(t^\ell), d_<(t^\ell),
d_\ge(t^\ell), d_\le(t^\ell)
& \mbox{truncated expansions in $t$  p. \pageref{pg:d_>}}\\
\nuI & \mbox{holomorphic one form  (\ref{eq:1stkind})}\\
\nuII & \mbox{meromorphic one form  p. \pageref{nutwo}}\\
w & \mbox{Abel map $(w: \SSS^k(X) \to \CC^g)$ p. \pageref{pg:Abelmap}}\\
\wdeg, \mwdeg, \ldeg & \mbox{degrees $(\wdeg(\phi_n)=N(n))$ 
p. \pageref{pg:wdeg}}\\
\SSS^k(X), \SSS^k_m(X) 
& \mbox{symmetric product of the curve and its singular strata
 p. \pageref{pg:S^k(X)}}\\
\JJ, \Pi, \kappa  & \mbox{Jacobian $\JJ  =  \CC^g/\Pi$, 
                 $\kappa:\CC^g \to \JJ$}\\
\WW^k, \WW^k_m & \mbox{strata in the Jacobian ($w( 
\SSS^k(X)) = \WW^k, w(\SSS^k_m(X)) =\WW^k_m)$
(\ref{eq:W_k1}) and p. \pageref{pg:W^n_m}}\\
\Lambda & \mbox{Young diagram p. \pageref{pg:Lambda}}\\
\Psi_n, \Psi_n^{(\ell)}, \psi_n, \psi_n^{(\ell)}
& \mbox{Frobenius-Stickelberger (FS) matrices and 
determinants p. \pageref{pg:FS-matrix}}\\
\mu_n & \mbox{ Definition \ref{def:mul}}\\
\omega', \omega'', \eta', \eta''
 & \mbox{complete integrals of the first kind and the second kind
   (\ref{eq2.5})}\\
\sigma  & \mbox{$\sigma$ function   (\ref{def_sigma})}\\
L(u, v), \chi(\ell)  & \mbox{quasi-periodicity of $\sigma$
 (\ref{eq:4.11})} \\
\Theta^k, \Theta^k_1 
 & \mbox{strata in the Jacobian 
(\ref{eq:Theta:g-1}), (\ref{eq:Theta:k}) and (\ref{eq:Theta:k1})} \\
h_n, h_n^{\LA \ell_1, \ell_2 \RA}  & \mbox{homogeneous symmetric functions
 p. \pageref{pg:h_n}}\\
T_n, T_n^{\LA \ell_1, \ell_2 \RA}  & \mbox{power symmetric functions,
 Proposition \ref{prop:SchurC}}\\
(a_1, \ldots, a_r; b_1, \ldots, b_r)
 & \mbox{characteristics of a partition  
p. \pageref{pg:charac.part.}}\\
s_{\Lambda}, S_{\Lambda}  & \mbox{ Schur function
 (\ref{eq:def_Schur}) }\\
u^\alpha, |\alpha|, d_>(u^\ell), 
 \wg(\alpha)
& \mbox{multi-index convention 
 (\ref{eq:multi-index}) p. \pageref{pg:multi-index},
 p. \pageref{pg:multi-index2}} \\ 
T_n^{(g;k)}, T_n^{(k)},
h_n^{(g;k)}  & \mbox{$T_n^{(k)}:=T_n^{\LA 1, k\RA}$,
$T_n^{(g;k)}:=T_n^{\LA k+1, g\RA}$, $h_n^{(g;k)}:=h_n^{\LA k+1, g\RA}$
p. \pageref{pg:T^k_j}}\\
\Lambda^{(k)}, \Lambda^{[k]} 
 & \mbox{
\ \ \ \ \ \ \ \ \ \ \ \ \ \ \ \ \ \ \ \ \ \ 
\ \ \ \ \ \ \ \ \ \ \ \ \ \ \ \ \ \ \ \ \ \ 
\ \ \ \ \ \ \ \ \ \ \ 
  truncated Young diagrams p. \pageref{pg:Lambdak}}\\
L^{[k]}(a_i, b_i) 
 & \mbox{  Proposition \ref{prop:L^k}}\\
H^i, h^i
 & \mbox{cohomology and its dimension p. \pageref{pg:cohomology}}\\
H_{\ell_1,\ell_2}, H_i
 & \mbox{ matrices, Lemma \ref{lemma:schur_Hgk}, proof of
Lemma \ref{lemma:schurLambda}}\\
\natural_{k}, \natural^{(i)}_k 
 & \mbox{sequences, Definition \ref{def:naturalk}}\\
u^{[k]}, u^{[g;k]},
 & u^{[k]}:= T_{\Lambda_i+g-i}^{(k)},
 u^{[g;k]}:= T_{\Lambda_i+g-i}^{(g;k)},
 \mbox{Lemmas \ref{lemma:schurLambda} and \ref{lmm:ggggg}}\\
u^{[k]}, u^{[g;k]}, v^{(i)},
 & v^{(i)}:=w(P_i), u^{[k]}= u^{[k-1]} +  v^{(k)}, 
u^{[g;k]}= u^{[g]}- u^{[k]}\\
\end{array}
$$

\section{Introduction}
Vanishing theorems for Riemann's theta function 
have been investigated classically as well as in the current 
literature, and connected with problems as diverse as 
the Schottky problem and integrable
nonlinear PDEs. The Riemann theta function has special 
algebraic properties when it comes from a Riemann surface;
in fact we focus on a special type of curve $X$ 
(reviewed in Section \ref{curves}) which is a 
cyclic cover of $\mathbb{P}^1$, 
as in our previous paper \cite{MP}.
We were able to express certain abelian functions in
terms of the polynomial defining the affine part of the curve.
We used the $\sigma$ function,
associated to the theta function, which Klein introduced \cite{K}
for genus-2 curves
to generalize the genus-1 Weierstrass $\sigma$; further work 
ensued in the 19th  century,
mainly for hyperelliptic curves 
(Klein, H. Burkhardt, O. Bolza); Klein gave a definition for general curves of
genus 3 \cite[\S 27]{klein1890} and Korotkin with Shramchenko,
inspired by Klein's constructions, recently produced both odd
and (for generic curves) even $\sigma$-functions for all (smooth) curves,
the latter invariant under modular transformations,
and the former, invariant up to a root of unity
\cite{KS}\footnote{We are very grateful
 to one Referee for this reference, which was not available at the time of our
 manuscript's submission.}\label{ftnt}.
 Baker brought the theory together in a monograph \cite{B3} and
extended the analysis (e.g., the aspect of
power-series expansion and partial differential equations). More recently, 
$\sigma$ was studied for all $C_{rs}$ curves
(cf., e.g., \cite{EEL}). By taking suitable limits,
we obtained the order of vanishing of $\sigma$ on
the stratification in the Jacobian given by the Abel image
of the symmetric products $\SSS^k(X)$ of the curve \cite[Remark 5.8]{MP}.
We continue that analysis, using Schur polynomials and representation theory,
to  obtain a 
vanishing pattern  using Young diagrams (Th. \ref{vanishingTh})
and  moreover we express ratios of coefficients of the multivariable
Taylor expansion of $\sigma$ by algebraic functions (Th. \ref{algebraic},
as well as Proposition \ref{prop:gggggk}); we note also
  recent 
work \cite{EHKJLP}$^{\mbox{\scriptsize{\ref{ftnt}}}}$ that 
solves the Jacobi inversion problem on  the singular
strata of the Jacobian by 
explicit algebraic expressions
for  derivatives of $\sigma$ in the hyperelliptic case,
based on the  embedding of the curve in the Jacobian
and expressed in terms of the Weierstrass semigroup at the basepoint
and attendant Schur-Weierstrasss polynomials.
Moreover, equations for the periods of holomorphic and meromorphic
differentials allow the authors of \cite{EHKJLP}
to perform numerical calculations and obtain qualitative
information on the geodesics of black hole space-time.  
Our key technique consists in enabling partial derivatives
on the Jacobian image of symmetric powers of the curve,
cf. Section 
\ref{CoordinateChange}, just as
 Klein and Baker used the derivative along the curve.
These techniques have produced new solutions of non-linear
wave equations \cite{MP09}.

In \cite{O1},
 \^Onishi gave a non-vanishing theorem for $\sigma$ over
a hyperelliptic curve $X$of genus $g$ given by affine equation
$y^2 = x^{2g+1} + \lambda_{2g} x^{2g} + \cdots
  + \lambda_{1} x^{} + \lambda_{0}$ and a point $\infty$, 
as a special case 
 of the Riemann Singularity Theorem  (Section \ref{vanishing}); specifically,
\begin{theorem}\label{Onishi}
 For $0<k\le g-1$,
let $D = P_1 + \cdots + P_k$ belong to 
$\SSS^k(X\backslash\infty ) 
\setminus (\SSS^k_1(X)\cap \SSS^k(X\backslash\infty ))$,
where 
$\SSS^k_1(X)$ are divisors in $\SSS^k(X)$, the $k$-th symmetric product
of the curve, whose linear series has
projective dimension at least 1; let
$$
	u^{[k]} := \sum_{i=1}^k \int^{P_i}_{\infty} \nuI,
$$
for a suitable basis  $\nuI$ of holomorphic differentials, and 
$$
 \natural_k:=\left\{
\begin{matrix} 
\{g, g - 2, \ldots, k + 2, k\} & \mbox{if } g-k \mbox{ is even}, \\
\{g - 1, g - 3, \ldots, k +3 , k +1 \} & \mbox{ otherwise};
\end{matrix} \right.
$$ 
call $n_k := \# \natural_k$ the cardinality of the set $\natural_k$.
The following holds:
\begin{enumerate}
\item
 For every multiple index  $(\alpha_1, \ldots, \alpha_m)$ 
with $\alpha_i\in \{ 1, \ldots, g\}$ (possibly repeated) 
and $m < n_k$,
$$
	\frac{\partial^m}
        {\partial u_{\alpha_1} \ldots \partial u_{\alpha_m}}
            \sigma(u^{[k]}) = 0.
$$
\item
For the multiple index $\natural_k$,
\begin{equation}
   \left(\prod_{\beta \in \natural_k}
	\frac{\partial}
         {\partial u_{\beta}}\right)
            \sigma(u^{[k]}) \neq 0.
\label{eq:Hvn-nvn}
\end{equation}
\end{enumerate}
\end{theorem}

We show some examples of $\natural_k$ in Table 1.1.
{\tiny{
\begin{gather*}
\centerline{
\vbox{
	\baselineskip =10pt
	\tabskip = 1em
	\halign{&\hfil#\hfil \cr
        \multispan7 \hfil Table 1.1 \hfil \cr
	\noalign{\smallskip}
	\noalign{\hrule height0.8pt}
	\noalign{\smallskip}
$(r,s)$ & \strut\vrule& $g$ & \strut\vrule & $\natural_1$ &$\natural_2$
 & $\natural_3$ &$ \natural_4$
& $\natural_5$ & $\natural_6$ & $\natural_7$ \cr
\noalign{\smallskip}
\noalign{\hrule height0.3pt}
\noalign{\smallskip}
$(2,3)$&\strut\vrule& $1$ & \strut\vrule \cr
$(2,5)$&\strut\vrule& $2$ & \strut\vrule & $\{2\}$\cr
$(2,7)$&\strut\vrule& $3$ & \strut\vrule & $\{2\}$&$\{3\}$\cr
$(2,9)$&\strut\vrule& $4$ & \strut\vrule & $\{2,4\}$&$\{3\}$&$\{4\}$ \cr
$(2,11)$&\strut\vrule& $5$ & \strut\vrule 
                    & $\{2,4\}$&$\{3,5\}$&$\{4\}$& $\{5\}$\cr
$(2,13)$&\strut\vrule& $6$ & \strut\vrule 
                    & $\{2,4,6\}$&$\{3,5\}$&$\{4,6\}$& $\{5\}$&$\{6\}$ \cr
$(2,15)$&\strut\vrule& $7$ &\strut\vrule 
                    & $\{2,4,6\}$&$\{3,5,7\}$&$\{4,6\}$& $\{5,7\}$&
                                           $\{6\}$&$\{7\}$\cr
$(2,17)$&\strut\vrule& $8$ & \strut\vrule 
                    & $\{2,4,6,8\}$&$\{3,5,7\}$&$\{4,6,8\}$& $\{5,7\}$&
                                 $\{6,8\}$&$\{7\}$& $\{8\}$ \cr
\noalign{\smallskip}
\noalign{\hrule height0.3pt}
\noalign{\smallskip}
\noalign{\hrule height0.8pt}
}
}
}
\end{gather*}
}}

\^Onishi generalized the theorem to cyclic
 $C_{3s}$ and $C_{5s}$ curves \cite{O2,MO}. For the 
analogous set $\natural_k$, similar (non-)vanishing results hold,
exemplified in Table 1.2.
{\tiny{
\begin{gather*}
\centerline{
\vbox{
	\baselineskip =10pt
	\tabskip = 1em
	\halign{&\hfil#\hfil \cr
        \multispan7 \hfil Table 1.2 \hfil \cr
	\noalign{\smallskip}
	\noalign{\hrule height0.8pt}
	\noalign{\smallskip}
$(r,s)$ & \strut\vrule& $g$ & \strut\vrule & $\natural_1$ &$\natural_2$
 & $\natural_3$ &$ \natural_4$
& $\natural_5$ & $\natural_6$ & $\natural_7$ & $\natural_8$ & $\natural_9$
 & $\natural_{10}$ & $\natural_{11}$ \cr
\noalign{\smallskip}
\noalign{\hrule height0.3pt}
\noalign{\smallskip}
$(3,4)$&\strut\vrule& $3$ & \strut\vrule & $\{2\}$&$\{3\}$\cr
$(3,5)$&\strut\vrule& $4$ & \strut\vrule & $\{2\}$&$\{3\}$&$\{4\}$ \cr
$(3,7)$&\strut\vrule& $6$ & \strut\vrule & $\{2,5\}$&$\{3,6\}$&$\{4\}$& 
                                           $\{5\}$&$\{6\}$  \cr
$(3,8)$&\strut\vrule& $7$ & \strut\vrule & $\{2,5\}$&$\{3,6\}$&$\{4,7\}$& 
                                           $\{5\}$&$\{6\}$&$\{7\}$ \cr
$(3,10)$&\strut\vrule&$9$ & \strut\vrule & $\{2,4,7\}$&$\{3,5,9\}$&$\{4,7\}$& 
                                           $\{5,8\}$&$\{6,9\}$& 
                                           $\{7\}$&$\{8\}$&$\{9\}$ \cr
$(5,6)$&\strut\vrule& $10$ & \strut\vrule&$\{2,5,8\}$&$\{3,7,9\}$&$\{4,8,10\}$
& 
                                          $\{5,9\}$&$\{6\}$& 
                                          $\{7\}$&$\{8\}$&$\{9\}$&$\{10\}$
\cr
$(5,7)$&\strut\vrule& $12$ & \strut\vrule & $\{2,5,8,12\}$&$\{3,7,9\}$
                                           &$\{4,8,11\}$&$\{5,9,12\}$& 
                                            $\{6,10\}$&$\{7,12\}$& 
                                            $\{8\}$&$\{9\}$& 
                                            $\{10\}$&$\{11\}$&$\{12\}$\cr
\noalign{\smallskip}
\noalign{\hrule height0.3pt}
\noalign{\smallskip}
\noalign{\hrule height0.8pt}
}
}
}
\end{gather*}
}}

In this article, we generalize these relations  to the
$C_{rs}$ curve of Section \ref{curves}, 
based on
  results of our previous paper \cite{MP} and
on the theory of Young diagrams, which governs the $\sigma$-function
as proved by Nakayashiki 
for general $C_{rs}$ curves \cite{N}; we note that
 such patterns for Schur-Weierstrass polynomials
and their derivatives were introduced in
\cite{BEL} and proven to satisfy 
 the Riemann vanishing theorem.
Nakayashiki's identification of the leading term in the expansion of $\sigma$
in terms of Schur-Weierstrass polynomials is used below in a crucial way
 particularly in Section \ref{vanishing}.

Theorem
\ref{vanishingTh} below may be of independent interest in the theory of Schur
functions and their derivatives on the stratification 
by partitions of Sato's Grassmannian (more general than our case, which
is concerned with  partitions related to the Weierstrass gaps of a $C_{rs}$
curve). 
Precise knowledge of the order of vanishing of $\sigma$ 
in terms of Weierstrass gaps is important 
not only for the intersection theory of the Jacobian, cf.
\cite{BV}, but also for the intersection theory of the
moduli space of Jacobians, in terms of Brill-Noether strata.

The paper can be read independently of the previous part
\cite{MP}: below, we briefly set up the notation, cite the statements 
we need and give 
precise references to proofs  in \cite{MP}. 
To illustrate the significance of patterns,
 we give two detailed examples, which go beyond the ones displayed above,
being pentagonal and
heptagonal. The Weierstrass-gap and Young-diagrams
numerology is provided in Section \ref{curves}, the $\Theta$
stratification and Schur function notation are found in Section \ref{the
sigma function};
this notation is used in Section \ref{vanishing}, 
which contains our main results.

\bigskip

One of the authors (S.M.) thanks to Yasushi Homma for his
private notebook on Young diagram; this work is based upon
it and Kenichi Tamano for private lectures for
two decades. The authors are grateful to
Victor Enolski and Yoshihiro \^Onishi for posing these
problems and suggesting the relation between sigma functions
and Young diagrams. 
They wish to acknowledge the hospitality of the
Hanse Institute for Advanced Study (Hanse-Wissenschaftskolleg)
as participants in the workshop
``Algebro-geometric methods in Gauge theory and General Relativity''
 (September 2011)
when  parts of this work were done.
They also thank the Referees for many 
 helpful comments, especially detecting some errors in 
earlier proofs of 
Corollary \ref{cor:Fay} and 
Lemma \ref{lemma:schurLambda} and providing references to relevant 
forthcoming papers.
E.P. wishes to acknowledge very valuable
partial support of her research under grant NSF-DMS-0808708.

\section{The $C_{rs}$ curve
and the pre-image of its Wirtinger varieties}\label{curves}

We recall from \cite{MP} notations and basic properties
for the Riemann surface
$$
	X:=\{(x,y) \ | \ y^r = f(x)\}\cup \infty
$$ 
whose finite part is given by an affine equation
\begin{equation}
y^r = f(x), \quad f(x) := x^s + \lambda_{s-1} x^{s-1} + \cdots +
\lambda_1 x + \lambda_0.
\end{equation}
The integers $r$ and $s$ are such
that $(r, s) = 1$ and $r<s$; the complex numbers $\lambda_0,\ldots ,
\lambda _{s-1}$ are such that the finite part of $X$ is smooth.
This is a particular $C_{rs}$ curve (a nomenclature introduced
in the 1990s to generalize elliptic curves in Weierstrass form).
It has genus $g=\frac{(r-1)(s-1)}{2}.$
In Section 5  we will  consider the degeneration of $X$ to
a singular curve $X_0$ for which all
$\lambda$'s vanish.

Let $R:=\CC[x,y]/(y^r - f(x))$,
$\cO_X$ be the sheaf of  holomorphic functions over $X$
 and $\JJ$ the Jacobian of $X$. \label{pg:Jacobian}
 We note that $R=\cO_X(*\infty) $ is the ring of meromorphic functions  
on $X$ regular outside $\infty$, where * stands for any order of pole.

For a non-negative integer $n$,
we denote by
 $\phi_n \in\CC[x,y] $ \label{pg:phi_n} 
 the (monic) monomial
  that at $\infty$
has a pole of order  $N(n)$ \label{pg:N(n)}, the $n$-th integer
 in  the (increasing)  sequence complementary to the 
Weierstrass gaps:
  $\phi_0 = 1$, $\phi_1 = x$, etc.;
by letting  $t_\infty$ be a local parameter at $\infty$,
the leading term of $\phi_n$ is proportional to $t_\infty^{-N(n)}$.
We routinely abuse notation slightly, to think of polynomials 
in $\CC[x,y]$ as functions in $R$.
A direct calculation gives
$N(0) = 0$,
$N(g-1) = 2g-2$,
$N(g) = 2g$ for a $C_{r,s}$ curve.
We define the w-degree, $\wdeg : R \to \ZZ$, \label{pg:wdeg}
which assigns to an element of $R$  
 its order of pole at $\infty$, 
$\wdeg(x) = r$, $\wdeg(y) = s$,
$\wdeg(\phi_n(P)) = N(n)$. 
 We also  consider the ring
$R_\lambda:=\QQ[x, y, \lambda_0, \ldots, \lambda_{s-1}]/(y^r - f(x))$
by regarding $\lambda$'s as indeterminates, and define a
 $\lambda$-degree,
$\ldeg:R_\lambda \to \ZZ$ as an extension of the w-degree by
assigning the degree $(s - i) r$ to  each $\lambda_i$.
This makes the polynomial defining the curve, $y^r-f(x)$, 
homogeneous of degree $rs$ with respect to the $\lambda$-degree. 

We denote a point
$P \in X\backslash\infty$ by its affine coordinates $(x, y)$;
we also loosely denote by a $k$-tuple
$(P_1,\ldots ,P_k )$, or by
a divisor $D=\sum_{i=1}^kP_i$,
 an element of $\SSS^k(X)$, the $k$-th symmetric
product of the curve. \label{pg:S^k(X)}
 For a given local parameter $t$ at some $P$ in $X$, by
$d_>(t^\ell)$ (resp. $d_<(t^\ell)$) we denote the  terms of 
a function on $X$ in its $t$-expansion
whose orders of zero at $P$
are greater (resp. less) than $\ell$;
$d_\ge(t^\ell)$ (resp. $d_\le(t^\ell)$) includes
terms of order $\ell$. \label{pg:d_>}
For   the local parameter $t_\infty$ at $\infty$, 
we have
$$
x = \frac{1}{t_\infty^r}, \quad
y = \frac{1}{t_\infty^s}\left(1 + d_>(t_\infty)\right), \quad
\phi_n(P) = \frac{1}{t_\infty^{N(n)}}\left(1 + d_>(t_\infty)\right).
$$

\medskip

A basis $\{\nuI_1, \ldots, \nuI_g\}$ of $H^0(X, K_X)$, where $K_X$ as 
customary denotes the canonical bundle (and with a slight abuse of
notation we do not distinguish between the 
 bundle, given divisor, and sheaf that correspond to each other),  
is given in terms of the $\phi_i$ following 
\cite[Ch. VI, \S91]{B1},
\begin{equation}
\nuI_i=
 \frac{\phi_{i-1}(P) d x}{r y^{r-1}}, \quad(i = 1, \ldots, g).
\label{eq:1stkind}
\end{equation}
We extend the w-degree to one-forms, by fixing the local parameter 
$t_\infty$ at $\infty$, so that 
for a one-form $\nu = (t_\infty^n + d_>(t_\infty^n)) d t_\infty$,
$\wdeg(\nu) = - n$.
Since we have 
\begin{equation}
 \nuI_i = t_\infty^{2g-N(i-1)-2}(1+d_{>0}(t_\infty)) d t_\infty,
\label{eq:dnu_tinf}
\end{equation}
the degree is given by
$$
 \mwdeg(\nuI_i) = 2g-N(i-1)-2,
$$
where $\mwdeg(f) = - \wdeg(f)$,
and to entire functions on $\mathbb{C}^g$, by pulling them back
to the curve via the Abel map defined next. 
Using  the analytic as opposed to the algebraic 
nature of the curve, we consider 
the Abel images of the $k$-fold symmetric products of the curve:
\begin{gather}
    \WW^k := \kappa\left(\left\{
             \sum_{i=1}^k \int^{(x_i,y_i)}_\infty
              \begin{pmatrix} \nuI_1 \\ \vdots \\ \nuI_g
               \end{pmatrix}\ \Bigr| (x_i,y_i) \in X\right\}\right)\subset\JJ
             , 
\label{eq:W_k1}
\end{gather}
where $\kappa$ is the projection $\CC^g \to \JJ= \CC^g/\Pi$, 
$\Pi$
is the period lattice of the basis
$\{\nuI_1, \ldots, \nuI_g\}$,
 and $\JJ$ is the Jacobian of $X$. We  denote by $w$ the 
Abel map \label{pg:Abelmap} from  $\SSS^k(X)$ of $X$ to 
$\kappa^{-1}\WW^k$ with base-point $\infty$, for any 
positive integer $k$. 
Note that there is a remaining 
$\Pi$-ambiguity due to the choice of path of integration:
our results below will be independent of such ambiguity,
but they require a $g$-tuple of complex numbers to be given
explicitly, $\uab:(P_1,\ldots ,P_k)\mapsto \uab(P_1,\ldots ,P_k)
=\sum_{i=1}^k\int_{\infty}^{P_i}\nuI\in\mathbb{C}^g$,
where we abbreviate by $\nu^\mathrm{I}$ the $g$-vector 
of holomorphic differentials $\nuI_i$.
When an analytic function, say, of $g$ complex variables
is evaluated on $u:=\uab(P_1,\ldots ,P_k)$,
we view it as function of the coordinates $(u_1,\ldots ,u_g)$
of the (column) vector $u$, as the convention goes.
We also introduce 
$$
\SSS^n_m(X) := \{D \in \SSS^n(X) \ | \
    \mathrm{dim} | D | \ge m\},
$$
where $|D|$ is the complete linear system
$\uab^{-1}(\uab(D))$ 
\cite[IV.1]{ACGH}.
If $n< g$, the singular locus of $\SSS^n(X)$ 
after moding out by linear equivalence, or
 projecting to the Picard group,
is $\SSS_1^n(X) $ 
 \cite[Ch. IV, Proposition 4.2, Corollary 4.5,
where our $\SSS^n(X)$ is $C^0_n$]{ACGH}. We let 
$\mathcal{W}^n_m:=w(\mathcal{S}^n_m (X))$.
\label{pg:W^n_m}

\medskip





The choice of basis $\{\nuI_1, \ldots, \nuI_g\}$
allows us to connect the expansion of the $\sigma$ function
(cf. Section \ref{the sigma function})
in the attendant Abelian coordinates $\{ u_1,...,u_g\}$
with the gap sequence 
at $\infty$.  To do so, we  introduce
a Young diagram (cf., e.g., \cite{Sa, BEL})
$\Lambda$ relative to the Weierstrass-gap sequence: \label{pg:Lambda}
from the top  down, $1\le i\le g$, the rows have length:
$$
	\Lambda_i = N(g) - N(i-1)  -g + i -1=g-N(i-1)+(i-1),
$$
$$
	|\Lambda| = \sum_{i=1}^g \Lambda_i = \frac{1}{24}(r^2 - 1) (s^2 -
        1)=g+\mathrm{w}(\infty ),
$$
where $\mathrm{w}(\infty )$ is the Weierstrass weight of the point $\infty$
(if we write $\Lambda_j$ for $(j > g)$ we set it equal to zero).

We give two examples, 
 the first in  the pentagonal
class, cf. \cite{MO}, and the second 
heptagonal (the fact that both $r$ and $s$ are prime
numbers is not essential but more convenient):
For the case  $(r, s) = (5, 7)$ (Table 2.1), we have
{\tiny
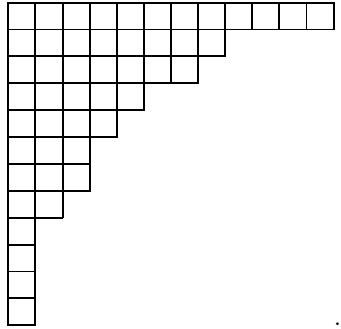
{
\begin{gather*}
\centerline{
\vbox{
	\baselineskip =10pt
	\tabskip = 1em
	\halign{&\hfil#\hfil \cr
        \multispan7 \hfil Table 2.1 \hfil \cr
	\noalign{\smallskip}
	\noalign{\hrule height0.8pt}
	\noalign{\smallskip}
$i$ & \strut\vrule& 0 &1 & 2 & 3 & 4 & 5 & 6 & 7 & 8 & 9 & 10 & 11 & 12 \cr
\noalign{\smallskip}
\noalign{\hrule height0.3pt}
\noalign{\smallskip}
$\phi(i)$ & \strut\vrule 
&1& $x$& $y$ & $x^2$&$xy $&$y^2$ & $x^3$ & $x^2y$ 
& $xy^2$ & $x^4$ & $y^3$ & $x^3y$ & $x^2y^2$ \cr 
$N(i)$ &
 \strut\vrule & 
 0&  5 & 7 & 10 & 12 & 14 & 15 & 17 & 19 & 20 & 21 & 22 & 24  \cr 
$\Lambda_i$ &
 \strut\vrule & 
-  & 12 & 8 & 7 & 5 & 4 & 3 & 3 & 2 & 1 & 1 & 1 & 1  \cr 
\noalign{\smallskip}
\noalign{\hrule height0.3pt}
\noalign{\smallskip}
	\noalign{\hrule height0.8pt}
}
}
}
\end{gather*}
}}

\begin{equation*}
 \yng(12,8,7,5,4,3,3,2,1,1,1,1).
\end{equation*}

For the case  $(r, s) = (7, 9)$ (Table 2.2), we have

{\tiny
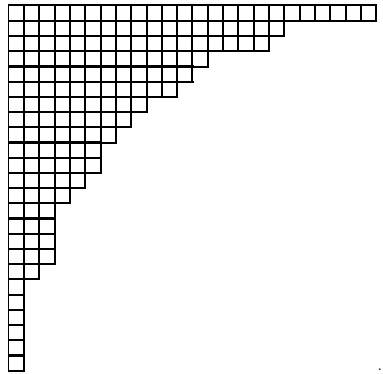
{
\begin{gather*}
\centerline{
\vbox{
	\baselineskip =10pt
	\tabskip = 1em
	\halign{&\hfil#\hfil \cr
        \multispan7 \hfil Table 2.2 \hfil \cr
	\noalign{\smallskip}
	\noalign{\hrule height0.8pt}
	\noalign{\smallskip}
$i$ & \strut\vrule& 
0 &1 & 2 & 3 & 4 & 5 & 6 & 7 & 8 & 9 & 10 & 11 & 12 \cr
\noalign{\smallskip}
\noalign{\hrule height0.3pt}
\noalign{\smallskip}
$\phi(i)$ & \strut\vrule 
&1& $x$& $y$ & $x^2$&$xy$&$y^2$ & $x^3$ & $x^2y$ 
&$xy^2$&$y^3$ & $x^4$ & $x^3 y$ & $x^2y^2$ \cr
$N(i)$ &
 \strut\vrule & 
 0&  7 & 9 & 14 & 16 & 18 & 21 & 23 & 25 & 27 & 28 & 30  & 32 \cr
$\Lambda_i$ &
 \strut\vrule & 
-  & 24 & 18 & 17 & 13 & 12 & 11 & 9 & 8 & 7 & 6 & 6 & 5 \cr
\noalign{\smallskip}
\noalign{\hrule height0.3pt}
\noalign{\smallskip}
	\noalign{\hrule height0.8pt}
$i$ & \strut\vrule& 
 13 & 14 & 15 & 16 & 17 & 18 & 19 & 20 & 21 & 22 & 23 & 24 \cr
\noalign{\smallskip}
\noalign{\hrule height0.3pt}
\noalign{\smallskip}
$\phi(i)$ & \strut\vrule 
& $xy^3$
&$x^5$ &$y^4$ & $x^4y$ & $x^3 y^2$ & $x^2y^3$
&$x^6$ & $xy^4$ & $x^5y$ & $y^5$ & $x^4y^2$ & $x^2y^4$ \cr
$N(i)$ &
 \strut\vrule & 
  34 & 35 & 36 & 37 & 39 & 41 & 42 & 43 & 44 & 45 & 46 & 48  \cr 
$\Lambda_i$ &
 \strut\vrule  
  & 4 & 3 & 3 & 3 & 3 & 2 & 1 & 1 & 1 & 1 & 1 & 1   \cr 
\noalign{\smallskip}
\noalign{\hrule height0.3pt}
\noalign{\smallskip}
	\noalign{\hrule height0.8pt}
}
}
}
\end{gather*}
}}
{\tiny{
\begin{equation*}
 \yng(24,18,17,13,12,11,9,8,7,6,6,5,4,3,3,3,3,2,1,1,1,1,1,1).
\end{equation*}
}}

\begin{lemma} \label{lm:degu}
For $v\in w(P)$, $P\in X$,
$$
	\mwdeg(v_i) = N(g) - N(i-1) - 1  
	 = 2g - N(i-1) - 1 = \Lambda_i + g - i,
$$
and $\mwdeg(v_g) = 1$.  
\end{lemma}
\begin{proof} 
>From (\ref{eq:dnu_tinf}),
 around $\infty$, $v_i$ is expressed by 
a local parameter $t_\infty$:
$$
 v_i = 
\frac{1}{2g-N(i-1)-1} t_\infty^{2g-N(i-1)-1}(1+d_{>}(t_\infty)) ,
$$
and thus we have
$
  \mwdeg(v_i) 
= 2g - N(i-1) - 1= N(g-1) - N(i-1) + 1.
$
\end{proof}

\begin{remark} \label{rmk:degu}
{\rm{
We extend the degree to a $g$-vector
$u = w(P_1,\ldots,P_k)$ in such a way  that $\mwdeg(u_i) = \mwdeg(v_i)$ if
$v_i$ of $v=w(P_\ell)$ does not vanish
for a certain $\ell \in \{1, 2, \ldots, k\}$.
Assuming
that a point $(P_1,\ldots,P_k)$ belongs to
$\SSS^{k}(X) \setminus \SSS^{k}_1(X)$
and each $P_j$  is near $\infty$,
by using a  local parameter $t_{\infty,j}$ for each $P_j$,
we read this degree off  the formal sum:
\begin{equation}
 u_i =
\frac{1}{2g-N(i-1)-1}
\left(t_{\infty,1}^{2g-N(i-1)-1}+
\cdots
+t_{\infty,k}^{2g-N(i-1)-1}\right)
\left(1+d_{>}(t_{\infty})\right).
\label{eq:v_t}
\end{equation}
(it is important to note that we do not perform an actual sum,
but just record formally the degree of each independent variable
in the $i$-component of the abelian integral;
for example, in the hyperelliptic case, if $k=2$ and the points
$P_1,\ P_2$ were hyperelliptic conjugates, the result
of the sum would be zero).
We introduce the multi-index convention:
for  non-negative integers $\alpha_i$ and
$\alpha := (\alpha_1, \ldots, \alpha_k)$, we write
\label{pg:multi-index}
\begin{equation}
        t^\alpha := t_1^{\alpha_1} t_2^{\alpha_2}
        \cdots t_k^{\alpha_k}, \quad
        \alpha = ( \alpha_1, \alpha_2, \ldots, \alpha_k),
        \quad |\alpha | := \sum_{i=1}^k  \alpha_i,
\label{eq:multi-index}
\end{equation}
and extend the definition $d_{>}(t_i^\ell)$
for variables, $t_1,\ldots,t_k$:
$d_{>}(t^\ell)\in \{\sum_{|\alpha|>\ell} a_\alpha t^\alpha\}$
and
$d_{\ge}(t^\ell)\in \{\sum_{|\alpha|\ge\ell} a_\alpha t^\alpha\}$.
For example,
$$
u_g = 
\left(t_{\infty,1}+ \cdots +t_{\infty,k}\right)
\left(1+d_{>}(t_{\infty})\right) \quad
\mbox{thus} \quad\mwdeg(u_g) = 1.
$$
}}
\end{remark}

\begin{remark}\label{serreduality}
{\rm By Serre duality, specifically $\dim H^0(X, D)=\deg D-g+1-
\dim H^0(X, (2g-2)\infty -D)$, we have a symmetry in the Young diagram.}
\end{remark}

\medskip
\begin{remark}
{\rm{Using Remark \ref{serreduality}, 
$ \mwdeg(u_i) = \Lambda_i + g - i$ is the hooklength  (cf. \cite[Ch. 3]{Sa})
of the node $(1,i)$ in the Young diagram $\Lambda$.
}}
\end{remark}

In \cite{MP},
we introduced meromorphic functions on the curve,
reviewed here in
Definition \ref{def:mul}, 
which  generalize the polynomial  
$U$ in Mumford's $(U,V,W)$ parametrization of
a hyperelliptic Jacobian (which he attributes to Jacobi)
\cite[Ch. IIIa]{Mu2}.

\medskip

To achieve an algebraic representation, e.g. in
 Section \ref{CoordinateChange}, of Abelian  vector fields,
 we introduce the Frobenius-Stickelberger (FS) matrix
and its determinant following \cite{MP}.
Let $n$ be a positive integer  and 
$P_1, \ldots, P_n$ be in $X\backslash\infty$.
We  define the \textit{$\ell$-reduced  
Frobenius-Stickelberger} (FS) \textit{matrix} by: \label{pg:FS-matrix}
$$
\Psi_{n}^{(\check\ell)}(P_1, P_2, \ldots, P_n) := 
\begin{pmatrix}
1 &\phi_1(P_1) & \phi_2(P_1)  &\cdots & \check{\phi_{\ell}}(P_1) 
& \cdots & \phi_{n}(P_1) \\
1 & \phi_1(P_2) & \phi_2(P_2) &\cdots & \check{\phi_{\ell}}(P_2) 
 & \cdots & \phi_{n}(P_2) \\
\vdots & \vdots & \vdots & \ddots& \vdots & \ddots& \vdots\\
1 & \phi_1(P_{n}) & \phi_2(P_{n}) &\cdots & \check{\phi_{\ell}}(P_n) 
 & \cdots&  \phi_{n}(P_{n})
\end{pmatrix},
$$
and $\psi_{n}^{(\check\ell)}(P_1, P_2, \ldots, P_n) := 
| \Psi_{n}^{(\check\ell)}(P_1, P_2, \ldots, P_n)|$
(a check on top of a letter signifies deletion).
It is also convenient to introduce the simpler notation:
\begin{gather}
\psi_{n}(P_1, \ldots, P_n)
 := |\Psi_{n}^{(\check{n})}(P_1, \ldots, P_n)|, \quad
\Psi_{n}(P_1, \ldots, P_n):= \Psi_{n}^{(\check{n})}(P_1, \ldots, P_n),
\label{eq:psin}
\end{gather}
for the un-bordered matrix.
We call this matrix {\it{Frobenius-Stickelberger (FS) matrix}}
and its determinant {\it{Frobenius-Stickelberger (FS) determinant}}.
These become undefined for some tuples in $(X \backslash \infty)^n$.

\medskip

Meromorphic functions, viewed as divisors
on the curve, allow us to express the addition structure
of $\mathrm{Pic}X$ in terms of  FS-matrices.
For $n$ points $(P_i)_{i=1, \ldots, n}$ $\in X\backslash\infty$,
we find an element of $R$ associated with 
 any point $P= (x,y)$ in $(X\backslash\infty )$,
$\alpha_n(P) :=\alpha_n(P; P_1, \ldots, P_n)
= \sum_{i=0}^{n} a_i \phi_i(P)$, $a_i \in \CC$ and $a_n = 1$,
which has a  zero at each point $P_i$ (with multiplicity, if
the $P_i$ are repeated)
and has smallest possible order of pole at $\infty$ with this property.
We obtain $\alpha_n(P)$ from the FS matrix as the $\mu_n$
defined herewith:

\begin{definition} \label{def:mul}
For $P, P_1, \ldots, P_n$ $\in (X\backslash\infty) \times
\SSS^n(X\backslash\infty)$, 
we define $\mu_n(P)$ by
$$
\mu_n(P): = 
\mu_n(P; P_1,  \ldots, P_n): = 
\lim_{P_i' \to P_i}\frac{1}{\psi_{n}(P_1' , \ldots, P_n' )}
\psi_{n+1}(P_1' , \ldots, P_n' , P),
$$
where the $P_i^\prime$ are generic and
the limit is taken (irrespective of the order) for each i;
we define $\mu_{n, k}(P_1, \ldots, P_n)$ by
$$
\mu_n(P)
 = \phi_n(P) + 
\sum_{k=0}^{n-1} (-1)^{n-k}\mu_{n, k}(P_1, \ldots, P_n) \phi_k(P),
$$
with the convention $\mu_{n, n}(P_1, \ldots, P_n) \equiv 1$.
\end{definition}

\begin{lemma}\label{prop:2theta2}
Let $n$ be a positive integer.
For $(P_i)_{i=1,\ldots, n}\in \SSS^n(X \backslash\infty) $,
the  function  $\alpha_n$
over $X$ induces the map (which we call by the same name): 
$$
\alpha_n: 
\SSS^n(X \backslash\infty)  \to \SSS^{N(n) - n}(X), 
$$
{{i.e.}}, to 
$(P_i)_{i=1,\ldots, n}\in \SSS^n(X \backslash\infty) $ there corresponds an 
 element $(Q_i)_{i=1,\ldots, N(n)-n}\in \SSS^{N(n)-n}(X)$,
such that
$$
\sum_{i=1}^{n} P_i - n \infty 
\sim - \sum_{i=1}^{N(n)-n} Q_i  + (N(n)-n) \infty .
$$
\end{lemma}

\bigskip
For an effective divisor 
of degree $n$, $D\in \SSS^n(X)$, 
let $D'$  be the maximal subdivisor  of $D$ which does
not contain $\infty$,
$D = D' + (n-m) \infty $
where $\deg D'=m (\le n)$ and
$D' \in \SSS^m(X\backslash\infty)$:
we extend the map  ${\alpha}_n$ to $\SSS^n(X)$ 
by defining $\overline{\alpha}_n(D)={\alpha}_m(D^\prime )+[N(n)-n-
(N(m)-m)]\infty.$

We see from the linear equivalence of Lemma \ref{prop:2theta2}:
\begin{proposition} \label{prop:addition}
For a positive integer, the
Abel map composed with $\alpha_n$ induces 
$$
\iota_n :  \WW^n \to  \WW^{N(n) - n}, \
\kappa\circ \uab \mapsto -\kappa\circ \uab.
$$
\end{proposition}

Let $\mathrm{image}(\iota_n)$ be denoted by $[-1]\WW^n$.

\begin{remark} \label{cor:add}
{\rm We recover the well-known result:}
The Serre involution on $\mathrm{Pic}^{g-1}$,
$\mathcal{L}\mapsto{K}_X\mathcal{L}^{-1}$,
 is given by $\iota_{g-1}$,
$$
\iota_{g-1}: \WW^{g-1} \to [-1] \WW^{g-1}.
$$
\end{remark}

\medskip
\section{The $\sigma$-function}\label{the sigma function}

In this section,
we summarize the results of \cite{MP} that are needed below.
As customary, we choose a basis
$   \alpha_i, \beta_j$  $ (1\leqq i, j\leqq g)$
of $H_1(X,\ZZ)$ with
intersection pairing
$\alpha_i\cdot\alpha_j=\beta_i\cdot\beta_j= 0$,
$\alpha_i\cdot\beta_j=\delta_{ij}$, and we denote
 the half-period matrices by
\begin{equation}
   \left[\,\omega'  \ \omega''  \right]= 
\frac{1}{2}\left[\int_{\alpha_i}\nuI_j \ \ \int_{\beta_i}\nuI_j
\right]_{i,j=1,2, \ldots, g},
   \,\,
   \left[\,\eta'  \ \eta''  \right]= 
\frac{1}{2}\left[\int_{\alpha_i}\nuII_j \ \ \int_{\beta_i}\nuII_j
\right]_{i,j=1,2, \ldots, g},
   \label{eq2.5}
  \end{equation} 
where $\nuII_j=\nuII_j(x,y)$ $(j=1, 2, \cdots, g)$ \label{nutwo}
are  differentials of the 
second kind, which we defined algebraically
in \cite{MP} after \cite{EEL}.

The following Proposition gives a 
 {\it generalized Legendre relation}
\cite{B1, BLE, EEL}.
\begin{proposition} \label{prop:gLegendreR}
The matrix
\begin{equation}
   M := \left[\begin{array}{cc}2\omega' & 2\omega'' \\ 2\eta' & 2\eta''
     \end{array}\right],
\end{equation} 
 satisfies 
\begin{equation}
   M\left[\begin{array}{cc} & -1 \\ 1 & \end{array}\right]{}^t {M}
   =2\pi\sqrt{-1}\left[\begin{array}{cc} & -1 \\ 1 &
     \end{array}\right].
   \label{eq2.7}
\end{equation} 
\end{proposition}

By the  Riemann relations \cite{F1}, it is known that
 $\text{Im}\,({\omega'}^{-1}\omega'') $ is positive definite.
Referring to  Theorem 1.1 in \cite{F1}, let
\begin{equation}
   \delta:=\left[\begin{array}{cc}\delta'\ \\
       \delta''\end{array}\right]\in \left(\tfrac12\ZZ\right)^{2g}
   \label{eq2.9} 
\end{equation} 
be the theta characteristic which gives the Riemann-constant vector
$\omega_\mathrm{R}=2\omega^{\prime\prime}\delta^{\prime}
+2\omega^\prime\delta^{\prime\prime}$ with
respect to the base point $\infty$ and the period matrix 
$[\,2\omega'\ 2\omega'']$. 

 We define an entire function of (a column-vector)
$u={}^t\negthinspace (u_1, u_2, \ldots, u_g)\in \mathbb{C}^g$,
associated 
(i.e.,
they differ by a multiplicative factor which is the exponential of
a quadratic form in the variables, cf. \cite[Chapter IV]{L})
to a translate of the Riemann $\theta$-function,
\begin{equation}
\begin{aligned}
   \sigma(u)&=\sigma(u;M)=\sigma(u_1, u_2, \ldots, u_g;M) \\
   &=c\,\text{exp}(-\tfrac{1}{2}{}\ ^t\negthinspace  u\eta'{\omega'}^{-1}u)
   \theta\negthinspace
   \left[\delta\right](\frac{1}{2}{\omega'}^{-1} u;\ 
{\omega'}^{-1}\omega'') \\
   &=c\,\text{exp}(-\tfrac{1}{2}\ ^t\negthinspace u\eta'{\omega'}^{-1}\  u) \\
   &\hskip 20pt\times
   \sum_{n \in \ZZ^g} \exp\big\{ \sqrt{-1}\big[\pi
    \ ^t\negthinspace (n+\delta'){\omega'}^{-1}\omega''(n+\delta')
   + \ ^t\negthinspace (n+\delta')({\omega'}^{-1} u+\delta'')\big]\big\}, 
\end{aligned}
   \label{def_sigma}
\end{equation}
where  $c$ is a constant that depends on the moduli of the curve.
Since in this paper we deal 
only with ratios of $\sigma$-functions, or with the vanishing order
of $\sigma$, 
we tacitly suppress the constant $c$.


For a given $u\in\CC^g$, we  introduce
$u'$ and $u''$ in $\RR^g$ so that
\begin{equation*}
   u=2\omega'u'+2\omega''u''.
\end{equation*}

\begin{proposition} \label{prop:pperiod} \cite[Proposition 4.3]{MP}
For $u$, $v\in\CC^g$, and $\ella$
$(=2\omega'\ella'+2\omega''\ella'')$ $\in\Pi$, we define
\begin{align*}
  L(u,v)    &:=2\ {}^t{u}(\eta'v'+\eta''v''),\nonumber \\
  \chi(\ella)&:=\exp[\pi\sqrt{-1}\big(2({}^t {\ella'}\delta''-{}^t
  {\ella''}\delta') +{}^t {\ella'}\ella''\big)] \ (\in \{1,\,-1\}).
\end{align*}
The following holds
\begin{equation}
\sigma(u + \ella) = \sigma(u) \exp(L(u+\frac{1}{2}\ella, \ella)) \chi(\ella).
        \label{eq:4.11}
\end{equation}
\end{proposition}

\begin{remark}\label{rmk:SigmaTheta}{\rm{
The above periodicity property of $\sigma$ is essentially the same 
that holds for the normalized theta function in Chapter VI of \cite{L}.
The normalized theta function is based upon the Hodge
structure of the Jacobian, whereas
$\sigma$ is related to the symplectic action appearing in the Legendre
relation, cf.
 Proposition \ref{prop:gLegendreR}. We note here the specific periodicity
because our vanishing results
allow us to extend it to
 certain derivatives of $\sigma$, cf. Corollary \ref{cor:pperiod}.
}}
\end{remark}

We  also note that, as in genus 1, 
$\sigma$ is  {\em modular invariant} (cf.  \cite{N} for any $(r,s)$ curve),
namely,
for every
 $\gamma\in\mathrm{Sp}(2g,{\mathbf Z})$,
we have 
$$
\sigma(u;\gamma M)=\sigma(u;M).
$$
In the case of  
general curves, the definition of odd $\sigma$ given in  \cite{KS} 
is modular invariant up to a given root of unity, and the even version
is modular invariant.

The vanishing locus of $\sigma$ is:
\begin{equation}
	\Theta^{g-1} =( \WW^{g-1} \cup [-1] \WW^{g-1}) =
\WW^{g-1}.
\label{eq:Theta:g-1}
\end{equation}
The last equality is due 
  to our choice of base point $\infty$ such that
$(2g-2)\infty=K_X$; indeed, $-w(D)=
w(-D)$ by definition for a divisor $D$,
and when $D$ has degree $g-1$,
by Serre duality $D$ is special if and only if $K_X-D$ is special.
The reason for
introducing $\WW^{g-1} \cup [-1] \WW^{g-1}$ is that the analogous
loci when $g-1$ is replaced by $k$ play an important role
and $\WW^{k}$ is not $[-1]$-invariant in general:
\begin{equation}
	\Theta^{k} := \WW^{k} \cup [-1] \WW^{k}.
\label{eq:Theta:k}
\end{equation}
Similarly, we define
\begin{equation}
	\Theta^{k}_1 := w(\SSS^{k}_1 (X)) \cup [-1] w(\SSS_1^{k}(X)).
\label{eq:Theta:k1}
\end{equation}

For $(r=2,s=2g+1)$ (hyperelliptic) curves and
$\infty$ a branch point,  $\Theta^k$ 
equals  $\WW^k$ for every positive integer $k$
but in general it does not.

The main result in \cite{MP} is  the following:
\begin{theorem}
{\rm{(}\bf{Jacobi inversion formulae over $\Theta^k$}\rm{)}}
\label{thm:JIF}
The following relations  hold
\begin{enumerate}

\item $\Theta^g$ case:
for
$(P_1, \ldots, P_g) \in \SSS^g(X) \setminus \SSS^g_1(X)$ and
$u = \pm w(P_1, \ldots, P_g)\in \kappa^{-1}(\Theta^g)$,
\begin{gather*}
\frac{\sigma_i(u) \sigma_g(u) - \sigma_{g i}(u) \sigma(u)}
               {\sigma^2(u)}
         =(-1)^{g-i+1} \mu_{g, i - 1}(P_1, \ldots, P_g),
         \quad \mbox{for } 0< i \le g.\\
\end{gather*}

\item $\Theta^{g-1}$ case:
for
$(P_1, \dots, P_{g-1}) \in \SSS^{g-1}(X) \setminus \SSS^{g-1}_1(X)$ and
$u = \pm w(P_1, \ldots, P_{g-1})\in \kappa^{-1}(\Theta^{g-1})$,
\begin{gather*}
\frac{\sigma_i(u)}
               {\sigma_g(u)}=
\left\{\begin{matrix}
         (-1)^{g-i} \mu_{g-1, i-1}(P_1, \ldots, P_{g-1}) 
                 & \mbox{for } 0 < i \le g,\\
          1 & \mbox{for } i = g.
       \end{matrix}\right.
\end{gather*}

\item $\Theta^{k}$ case:
for
$(P_1, \ldots, P_k) \in \SSS^k(X) \setminus \SSS^k_1(X)$ and
$u = \pm w(P_1, \ldots, P_k)\in \kappa^{-1}(\Theta^k)$,
\begin{gather*}
\frac{\sigma_i(u)}
               {\sigma_{k+1}(u)}=
\left\{\begin{matrix}
          (-1)^{k-i+1} \mu_{k, i-1} (P_1, \ldots, P_k) 
          & \mbox{for }  0 < i \le k,\\
          1 & \mbox{for } i = k+1,\\
         0 &\mbox{for } k+ 1 < i \le  g.
       \end{matrix}\right.
\end{gather*}
\end{enumerate}
\end{theorem}

\begin{remark} {\rm{
We could easily extend  parts 2. and 3.
of Theorem \ref{thm:JIF} to any
non-singular $(r,s)$ curve,
 with affine equation:
\begin{equation}
 f(x,y) = y^r + x^s + \sum_{i,j: rs > si + rj \ge 0} \lambda_{ij} y^i x^j = 0,
\label{eq:f(x,y)}
\end{equation}
where $\lambda_{ij}$ are complex numbers.
As stated in Section 2, we limit our study to the cyclic
type to use the work in \cite{MP}. For the
general $C_{rs}$ curve, the definitions of Section 2
are naturally modified
(cf. \cite{EEL}). 
 As mentioned in \cite[Remark 5.10]{MP},
Theorem \ref{thm:JIF} 2. and 3. could be proved by Fay's and 
Jorgenson's results  \cite{F1,Jo} which hold for every Riemann surface, 
including (\ref{eq:f(x,y)}).  
A direct proof can also be given  by the result of
Nakayashiki \cite[Theorem 1]{N}, where
  the sigma function is explicitly expressed 
in terms of the ``prime form'' and FS-matrices.
When computing the vanishing order
of both numerator and denominator of 
the left-hand side of 2. and 3.
in essentially the same way as in \cite[Section 5]{MP},
the prime forms cancel and the ratio is reduced to a ratio of    
FS-matrices.

It should also
be possible to prove 
 statement 1. of Theorem
\ref{thm:JIF} for the general curve (\ref{eq:f(x,y)}), 
by using the relation in \cite[Proposition 4.5]{MP}.
To compare poles and zeros
of  both sides, as functions of $u = w(x,y)$ and $(x,y)$,
respectively,
($(x,y)$ a point of (\ref{eq:f(x,y)})),
one would use two facts: $\sigma$ vanishes to order one
on $\Theta^{g-1}$, and the Weierstrass gap at $2g-1$ is 
adjacent to non-gaps. While this would give the expected number of
zeros, however, it is not easy to specify them because
the numerator in the left-hand side of 1. is a difference.
}}
\end{remark}

In \cite{MP}, we had arrived at these results
by comparing  abelian functions with meromorphic
functions on the curve.  In the present work, we
 give precise results on the order of vanishing of
$\sigma$ itself on the stratification $\Theta^k$.  To do so, we first
give the expansion of the $\sigma$-function. 
We introduce the Schur function $\bold{s}_\Lambda (t)$, 
\begin{equation}
	\bold{s}_{\Lambda}(t) :=\frac{
          |t_j^{\Lambda_i + g - i}|_{1\le i, j \le g}}
          {|t_j^{i-1}|_{1\le i, j \le g}}.
\label{eq:def_Schur}
\end{equation}

The complete homogeneous symmetric function 
$h_n^{\LA\ell_1,\ell_2\RA}=h_n(t_{\ell_1},\ldots,t_{\ell_2})$
for positive integers $\ell_1$ and $\ell_2$
($\ell_1 <\ell_2$) is given by
\label{pg:h_n}
$$
\prod_{i=\ell_1}^{\ell_2}\frac{1}{(1-z t_i)}
=\sum_{n\ge 0} h_n^{\LA\ell_1,\ell_2\RA} z^n,\quad
h_n^{\LA\ell_1,\ell_2\RA} =0 \ \mbox{for } n < 0.
$$

\begin{proposition} \label{prop:SchurC} 
 {\rm{\cite[Theorem 4.5.1]{Sa}}}
Using the complete homogeneous symmetric functions $h_n:=h_n^{\LA1,g\RA}$,
we can express $\bold{s}_\Lambda$ by a ($g\times g$)
\textrm{Jacobi-Trudi Determinant}, $|a_{ij}|_{1\le i,j\le g}$ 
with $a_{ij}=h_{\Lambda_i+j-i}$:
$$
	\bold{s}_{\Lambda}(t) :=|h_{\Lambda_i+j-i}|,\quad
        h_n=\frac{1}{n!}
\left|\begin{matrix}
   T_1 & -1 & 0 & \cdots & \\
   2T_2 & T_1 & -2 & \cdots & \\
   \vdots & \vdots &\vdots&\ddots&\vdots\\
   (n-1)T_{n-1} & (n-2)T_{n-2} & (n-3)T_{n-3} & \cdots & 1-n\\
   nT_{n} & (n-1)T_{n-1} & (n-2)T_{n-2} & \cdots & T_1\\
\end{matrix} \right|,
$$
where $h_0 =1$, $h_{i<0} =0$ and $T_k:=T_k^{\LA 1,g\RA}$,
$$
          T_k^{\LA\ell_1,\ell_2\RA} 
:=\frac{1}{k}\sum_{j=\ell_1}^{\ell_2} t_j^k.
$$
\end{proposition} 

When regarded  as a function of $T$, we rename $\bs$,
$S_{\Lambda}(T) :=\bs_{\Lambda}(t)$.\label{pg:S_Lambda}

We now give an earlier version of the Jacobi-Trudi formula
of Proposition \ref{prop:SchurC}; 
by connecting the two, we provide a proof of Proposition
\ref{prop:SchurC},
as well as a modified formula which will be used in Section
\ref{vanishing}.

\begin{lemma} \label{lemma:schur_hgk}
$$
	\bold{s}_{\Lambda}(t) :=|h^{\LA j,g\RA}_{\Lambda_i+j-i}|_{1\le
          i,j\le g}.
$$
\end{lemma}

\begin{proof}
 Using a  Vandermonde determinant, (\ref{eq:def_Schur}) is reduced to
\begin{equation*}
\frac{
         \left|
 \begin{matrix}
t_1^{\Lambda_1+ g-1} \ & t_2^{\Lambda_1+ g-1} \ & 
\cdots  \ & t_{g-1}^{\Lambda_1+ g-1} \ & t_g^{\Lambda_1+ g-1} \\ 
\vdots \ & \vdots \ & 
\ddots \ & \vdots \ & \vdots \\ 
t_1^{\Lambda_{g-1}+1} \ & t_2^{\Lambda_{g-1}+1} \ &
\cdots  \ & t_{g-1}^{\Lambda_{g-1}+ 1} \ & t_g^{\Lambda_{g-1}+ 1} \\ 
t_1^{\Lambda_{g}} \ & t_2^{\Lambda_{g}} \ &
\cdots  \ & t_{g-1}^{\Lambda_{g}} \ & t_g^{\Lambda_{g}} \\ 
 \end{matrix}
        \right|
}{ \prod_{i<j} (t_i- t_j)}.
\end{equation*}
This equals
\begin{equation*}
\frac{
         \left|
 \begin{matrix}
t_1^{\Lambda_1+ g-1} - t_g^{\Lambda_1+ g-1} 
  \ & t_2^{\Lambda_1+ g-1} - t_g^{\Lambda_1+ g-1} 
\ & \cdots  \ & t_{g-1}^{\Lambda_1+ g-1} - t_{g}^{\Lambda_1+ g-1}
 \ & t_g^{\Lambda_1+ g-1} \\ 
\vdots \ & \vdots \ & 
\ddots \ & \vdots \ & \vdots \\ 
t_1^{\Lambda_{g-1}+1} -t_g^{\Lambda_{g-1}+1} 
\ & t_2^{\Lambda_{g-1}+1} - t_g^{\Lambda_{g-1}+1}
 \ &
\cdots  \ & t_{g-1}^{\Lambda_{g-1}+ 1} - t_{g}^{\Lambda_{g-1}+ 1} 
\ & t_g^{\Lambda_{g-1}+ 1} \\ 
t_1^{\Lambda_{g}} -t_g^{\Lambda_{g}} 
\ & t_2^{\Lambda_{g}} -t_g^{\Lambda_{g}} 
 \ &
\cdots  \ & t_{g-1}^{\Lambda_{g}} -t_g^{\Lambda_{g}} 
\ & t_g^{\Lambda_{g}} \\ 
 \end{matrix}
        \right|
}{ \prod_{i<j} (t_i- t_j)}
\end{equation*}
and noting $t^\ell - t^{\prime\ell} 
= (t-t')h_\ell(t, t')$, this becomes 
{\small{
\begin{equation*}
\frac{
         \left|
 \begin{matrix}
h_{\Lambda_1+ g-2}(t_1, t_g) 
  \ & h_{\Lambda_1+ g-2}(t_2, t_g) 
\ & \cdots  \ & 
   h_{\Lambda_1+ g-2}(t_{g-1}, t_g) 
 \ & t_g^{\Lambda_1+ g-1} \\ 
\vdots \ & \vdots \ & 
\ddots \ & \vdots \ & \vdots \\ 
h_{\Lambda_{g-1}}(t_1, t_g)
\ & h_{\Lambda_{g-1}}(t_2, t_g)
 \ & \cdots  \ & h_{\Lambda_{g-1}}(t_{g-1}, t_g)
\ & t_g^{\Lambda_{g-1}+1} \\ 
h_{\Lambda_{g}-1}(t_1, t_g)
\ & h_{\Lambda_{g}-1}(t_2, t_g)
 \ & \cdots  \ & h_{\Lambda_{g}-1}(t_{g-1}, t_g)
\ & t_g^{\Lambda_{g}} 
 \end{matrix}
        \right|
}{ \prod_{i<j<g} (t_i- t_j)}.
\end{equation*}
}}
We go on similarly, to co-factorize by $\prod_i^{g-2} (t_{g-1}-t_i)$,
{\tiny{
\begin{equation*}
\frac{
         \left|
 \begin{matrix}
h_{\Lambda_1+ g-3}(t_1, t_{g-1}, t_g) 
  \ & h_{\Lambda_1+ g-3}(t_2, t_{g-1}, t_g) 
\ & \cdots  \ & 
   h_{\Lambda_1+ g-3}(t_{g-2},  t_{g-1}, t_g) 
\ &  h_{\Lambda_1+ g-2}(t_{g-1},  t_g) 
 \ & t_g^{\Lambda_1+ g-1} \\ 
\vdots \ & \vdots \ & 
\ddots \ & \vdots \ & 
\vdots \ & \vdots \\ 
h_{\Lambda_{g-1}-1}(t_1, t_{g-1}, t_g)
\ & h_{\Lambda_{g-1}-1}(t_2, t_{g-1}, t_g)
 \ & \cdots  
\ & h_{\Lambda_{g-1}-1}(t_{g-2}, t_{g-1}, t_g)
\ & h_{\Lambda_{g-1}}(t_{g-1}, t_g)
\ & t_g^{\Lambda_{g-1}+1} \\ 
h_{\Lambda_{g}-2}(t_1, t_{g-1}, t_g)
\ & h_{\Lambda_{g}-2}(t_2, t_{g-1}, t_g)
 \ & \cdots  
\ & h_{\Lambda_{g}-2}(t_{g-2}, t_{g-1}, t_g)
\ & h_{\Lambda_{g}-1}(t_{g-1}, t_g)
\ & t_g^{\Lambda_{g}} 
 \end{matrix}
        \right|
}{ \prod_{i<j<g-1} (t_i- t_j)}
\end{equation*}
}}
and derive the statement.
\end{proof}
 Lemma \ref{lemma:schur_hgk}
 will yield the Jacobi-Trudi determinant formula
in Proposition \ref{prop:SchurC}. In order to  relate the two expressions,
we prove:
\begin{lemma}
\begin{enumerate}
\item 
$h_n^{\LA\ell_1,\ell_2\RA}
=h_n^{\LA\ell_1 + 1,\ell_2\RA}
+h_{n-1}^{\LA\ell_1,\ell_2\RA}t_{\ell_1}
=h_n^{\LA\ell_1,\ell_2 - 1\RA}
+h_{n-1}^{\LA\ell_1,\ell_2\RA}t_{\ell_2}
$.

\item  For positive integers $\ell_1$, $\ell_2$, $n$ and $m$
satisfying $\ell_2>\ell_1$ and $m < \ell_2 - \ell_1$,
$h_n^{\LA\ell_1,\ell_2\RA}
=h_n^{\LA\ell_1 + m,\ell_2\RA}
+h_{n-1}^{\LA\ell_1 + m - 1,\ell_2\RA}
h_1^{\LA\ell_1, \ell_1 +m - 1\RA}
+h_{n-2}^{\LA\ell_1 +m-2, \ell_2\RA}
h_2^{\LA\ell_1,\ell_1 +m-2\RA}
+\cdots
+h_{n-m}^{\LA\ell_1,\ell_2\RA}
h_{m}^{\LA\ell_1,\ell_1\RA}
$.

\item  For positive integers $\ell_1$, $\ell_2$ and $n$
satisfying $\ell_2>\ell_1$,
$h_n^{\LA\ell_1,\ell_2\RA}
=h_n^{\LA\ell_2,\ell_2\RA}
+h_{n-1}^{\LA\ell_2 - 1,\ell_2\RA}
h_1^{\LA\ell_1, \ell_2 - 1\RA}
+h_{n-2}^{\LA\ell_2-2, \ell_2\RA}
h_2^{\LA\ell_1,\ell_2-2\RA}
+\cdots
+h_{n-\ell_2+\ell_1}^{\LA\ell_1,\ell_2\RA}
h_{\ell_2-\ell_1}^{\LA\ell_1,\ell_1\RA}
$.
\end{enumerate}
\end{lemma}

\begin{proof}
The product $
\prod_{i=\ell_1}^{\ell_2}\frac{1}{(1-z t_i)}
=\frac{1}{(1-z t_{\ell_1})}
\prod_{i=\ell_1+1}^{\ell_2}\frac{1}{(1-z t_i)}
$
is expanded using
$$
(1 - t_{\ell_1} z)
(1+ h_1^{\LA\ell_1,\ell_2\RA} z+
h_2^{\LA\ell_1,\ell_2\RA} z^2 + \cdots) 
=
(1+ h_1^{\LA\ell_1+1,\ell_2\RA} z +
h_2^{\LA\ell_1+1,\ell_2\RA} z^2 + \cdots),
$$
which shows the first equality in the first statement;
the second equality is similarly obtained.
Statement 2. for the $m= 1$ case is reduced to the first statement
for every $n$. We give a proof by induction:
 assume that 2. holds for $m\ge m_0$,
\begin{equation*}
\begin{split}
h_n^{\LA\ell_1,\ell_2\RA}
&=h_n^{\LA\ell_1 + m,\ell_2\RA}
+h_{n-1}^{\LA\ell_1 + m - 1,\ell_2\RA}
h_1^{\LA\ell_1 + 1, \ell_1 +m-1\RA}
+h_{n-2}^{\LA\ell_1 +m-2, \ell_2\RA}
h_2^{\LA\ell_1,\ell_1 +m-2\RA}\\
&+\cdots+
h_{n-m}^{\LA\ell_1,\ell_2\RA}
h_{m}^{\LA\ell_1,\ell_1\RA}.
\end{split}
\end{equation*}
We consider the $m+1$ case:
part 1. implies that  the individual terms in the right-hand side satisfy,
\begin{equation*}
\begin{split}
h_n^{\LA\ell_1,\ell_2\RA}
&=(h_n^{\LA\ell_1 + m +1,\ell_2\RA}
+h_{n-1}^{\LA\ell_1 + m,\ell_2\RA}
t_{\ell_1 + m})\\
&\ +(h_{n-1}^{\LA\ell_1 + m ,\ell_2\RA}
+h_{n-2}^{\LA\ell_1 + m - 1,\ell_2\RA} t_{\ell_2 - m - 1})
h_1^{\LA\ell_1 + 1, \ell_1 +m+1\RA} \\
&\ +(h_{n-2}^{\LA\ell_1 + m - 1,\ell_2\RA}
+h_{n-3}^{\LA\ell_1 + m - 2,\ell_2\RA} t_{\ell_2 - m - 2})
h_2^{\LA\ell_1 + 2, \ell_1 +m-1\RA} \\
&\ +\cdots
+(h_{n-m}^{\LA\ell_1+1,\ell_2\RA}
+h_{n-m-1}^{\LA\ell_1,\ell_2\RA}t_{\ell_1})
h_{m}^{\LA\ell_1,\ell_1\RA}\\
&=h_n^{\LA\ell_1 + m +1,\ell_2\RA}
+h_{n-1}^{\LA\ell_1 + m,\ell_2\RA}
(t_{\ell_1 + m} + h_1^{\LA\ell_1 + 1, \ell_1 +m+1\RA}) \\
&\ +h_{n-2}^{\LA\ell_1 + m - 1,\ell_2\RA} 
(t_{\ell_1 - m - 1} h_1^{\LA\ell_1 + 1, \ell_1 +m+1\RA} 
+h_2^{\LA\ell_1 + 2, \ell_1 +m-1\RA}) \\
&\ +\cdots
+h_{n-m-1}^{\LA\ell_1,\ell_2\RA} t_{\ell_1}
h_{m}^{\LA\ell_1,\ell_1\RA},
\end{split}
\end{equation*}
which gives the $m+1$ case.
Statement 3. is obtained by setting $m= \ell_2 - \ell_1 -1$. 
\end{proof}

A modified version of the Jacobi-Trudi relation follows:
\begin{lemma} \label{lemma:schur_Hgk}
For any fixed $k$, $1\le k\le g$, juxtaposing two matrices
to obtain a $g\times g$ matrix,
$$
	\bold{s}_{\Lambda}(t_1,\ldots,t_g) =|
H_{g,k}(t_1,\ldots,t_g), 
H_{g,g-k}(t_{k+1},\ldots,t_g)|, 
$$
where $H_{a,b}(t_\ell,\ldots,t_g)$'s are $a\times b$ matrices defined by
\begin{equation*}
\begin{split}
H_{g,k}(t_1,\ldots,t_g) 
               &:=(h_{\Lambda_i+j-i}(t_1,\ldots,t_g))_{1\le i
\le g,1\le j \le k}, \\
H_{g,g-k}(t_{k+1},\ldots,t_g) 
               &:=(h_{\Lambda_i+j-i}(t_{k+1},\ldots,t_g))_{1\le i
\le g,k+1\le j \le g}. \\
\end{split}
\end{equation*}
\end{lemma}

For a given Young diagram $\Lambda=(\Lambda_1, \Lambda_2, \ldots, 
\Lambda_\ell)$, 
the length $r$ of the diagonal is
called the rank of the partition \cite[\S 4.1, p. 51]{FH}.
Let $a_i$ and $b_i$ be the number of boxes below and to the 
right of the $i$-th box of the diagonal, reading from lower right
to upper left.
Frobenius named
$(a_1, \ldots, a_r; b_1, \ldots, b_r)$
 the characteristics of
the partition \cite[\S 4.1 p. 51]{FH}. \label{pg:charac.part.}
Here $a_i < a_j$ and $b_i < b_j$ for $i < j$.

We use the multi-index convention (cf. Guide to Symbols)
and define a map
for $\beta:= (\beta_1, \ldots, \beta_g)$,
 \label{pg:multi-index2}
$$
\wg(\beta):= 
((2g - N(0) - 1)\beta_1,
(2g - N(1) - 1)\beta_2,
 \ldots,
(2g - N(g-2) - 1)\beta_{g-1},
\beta_g) \in \ZZ^g
$$
The following
relation between
Schur-Weierstrass polynomials and the $\sigma$-function 
was proved by Nakayashiki \cite{N}.

\begin{proposition} \label{prop:Nakayashiki}
The expansion of $\sigma(u)$ at the origin takes the form
$$
   \sigma(u) = S_{\Lambda}(T)|_{T_{\Lambda_i + g - i} = u_i} 
            + \sum_{|\wg(\alpha)|>|\Lambda| } c_\alpha u^\alpha
$$
where $c_\alpha\in \QQ[\lambda_j]$ and $S_\Lambda (T)$ is the lowest-order
term in the w-degree of the $u_i$; 
$\sigma(u)$ is homogeneous of degree  $|\Lambda|$ 
 with respect to the $\lambda$-degrees. 
\end{proposition}

We note that $S_\Lambda$ is a function of 
$\{T_{\Lambda_i + g - i}\}_{i=1, \ldots, g}$, even though
\textit{a priori} it depends on $\{T_{i}\}_{i=1, \ldots, 2g-1}$.

\begin{proof}
The result follows from \cite[Theorem 3]{N}.
\end{proof}

\begin{remark}{\rm{
The symmetric functions have degree coming from the natural
order of the multi-variables $(t_1, \ldots, t_g)$.
In Proposition \ref{prop:Nakayashiki}, the degree corresponds
to the w-degree on $W^{k-1} \subset W^k$ and 
 $S^{k-1} X \cup \infty\subset S^k X$, according to the convention of 
Remark \ref{rmk:degu}.
}}
\end{remark}

\section{Algebraic expression of the Jacobian
of a coordinate change}\label{CoordinateChange}
In order to detect the vanishing of the multiderivatives of $\sigma$,
following Weierstrass, Klein, Baker and others
\cite{B1, B2, B3, K, W},
we  consider certain vector fields on the symmetric products of the curve.

Let $k$ be a positive integer $\le g$.
We fix a subset $K_k$ of $\{1, 2, \ldots, g\}$ with 
 $k$ elements. We relabel the indices  by the map 
$\iota: \{1, 2, \ldots, k\} \to K_k$ such that
$\iota(i) < \iota(i+1)$ for $i = 1, \ldots, k-1$; we 
set $\iota'(i) := \iota(i)-1$.

By inverting the Jacobian determinant for coordinate change
where the truncated map (defined only locally around the points,
lest the paths of integration differ by homotopy) is smooth,
\begin{equation}
 \mathrm{proj}_{K_k} \circ \uab:
	\SSS^k(X) \to \mathbb{C}^g\rightarrow
	\mathbb{C}^k,\ (P_1,\ldots ,P_k)\mapsto
(u_j=\sum_{i=1}^k\int_\infty^{P_i}\nu_j^{\mathrm I})_{j\in K_k},
\label{partialabel}
\end{equation}
we give an algebraic expression for  vector fields that
correspond to the `partial' differentials.
Assuming (\ref{partialabel}) invertible over
an open set ${\mathcal{U}}\subset \SSS^k(X)$, 
and denoting, loosely,
by 
$$
\partial_{u_j}, \ \mathrm{or, \ if \ typographically\ preferable,} \ 
\frac{\partial}{ \partial u_j},\ \            
{j \in K_k},
$$
the coordinate vector field for projected coordinates
$\mathbb{C}^g\rightarrow\mathbb{C}^k, \ (u_1,\ldots ,u_g)\mapsto
(u_{\iota(1)},\ldots ,u_{\iota(k)}),$  
  which acts by holding constant the $u_i, i \in K_k, i\neq j$, 
we compute:
\begin{equation}
\left(\begin{matrix}
\partial_{u_{\iota(1)}}\\ 
\partial_{u_{\iota(2)}}\\
 \vdots\\ 
\partial_{u_{\iota(k)}}
\end{matrix}\right)
=r
\left(\begin{matrix}
\phi_{\iota'(1)}(P_1) & \phi_{\iota'(2)}(P_1) &
 \cdots &\phi_{{\iota'(k)}}(P_1) \\
\phi_{\iota'(1)}(P_2) & \phi_{\iota'(2)}(P_2) &
 \cdots &\phi_{{\iota'(k)}}(P_2) \\
 \vdots & \vdots & \ddots& \vdots\\
 \phi_{\iota'(1)}(P_{k}) & \phi_{\iota'(2)}(P_{k}) & \cdots 
&\phi_{\iota'(k)}(P_{k})
\end{matrix}\right)^{-1}
\left(\begin{matrix}
y^{r-1}_1 \partial_{x_1}\\ 
y^{r-1}_2 \partial_{x_2}\\
 \vdots\\ 
y^{r-1}_k \partial_{x_k}
\end{matrix}\right).
\label{eq:pu_px}
\end{equation}

By letting $\displaystyle{
\partial_{v_g^{(i)}}
:=\frac{\partial }{\partial v_g^{(i)}}
= \frac{r y_i^{r-1}}{\phi_{g-1}(x_i,y_i)} 
\frac{\partial }{\partial x_i}}$, this relation can be expressed by
\begin{equation}
\left(\begin{matrix}
\partial_{u_{\iota(1)}}\\ 
\partial_{u_{\iota(2)}}\\
 \vdots\\ 
\partial_{u_{\iota(k)}}
\end{matrix}\right)
=
\left(\begin{matrix}
\phi_{\iota'(1)}(P_1)/ \phi_{g-1}(P_1) &
 \cdots &\phi_{{\iota'(k)}}(P_1) / \phi_{g-1}(P_1)  \\
\phi_{\iota'(1)}(P_2) / \phi_{g-1}(P_2) &
 \cdots &\phi_{{\iota'(k)}}(P_2) / \phi_{g-1}(P_2) \\
 \vdots & \ddots & \vdots \\
 \phi_{\iota'(1)}(P_{k}) / \phi_{g-1}(P_{k}) & \cdots 
&\phi_{\iota'(k)}(P_{k}) / \phi_{g-1}(P_{k})  
\end{matrix}\right)^{-1}
\left(\begin{matrix}
 \partial_{v_g^{(1)}}\\ 
\partial_{v_g^{(2)}}\\
 \vdots\\ 
 \partial_{v_g^{(k)}}
\end{matrix}\right).
\label{eq:pu_pv}
\end{equation}

As a consequence:
\begin{lemma} \label{lemma:d_u} \cite[Proposition 2.11]{MP}
On the open set ${\mathcal{U}}\subset \SSS^k(X)$ and
for $(P_1, \cdots, P_k) \in {\mathcal{U}}$,
$\displaystyle{
r\sum_{i=1}^k  \epsilon_i 
\frac{\partial}{\partial u_{\iota(i)}}
}$
is expressed by
$$
\left|\begin{matrix}
\phi_{\iota'(1)}(P_1) & \phi_{\iota'(2)}(P_1) & \cdots 
   &\phi_{\iota'(k)}(P_1) \\
\phi_{\iota'(1)}(P_2) & \phi_{\iota'(2)}(P_2) & \cdots
    &\phi_{\iota'(k)}(P_2) \\
\vdots & \vdots & \ddots& \vdots \\
\phi_{\iota'(1)}(P_{k}) & \phi_{\iota'(2)}(P_{k}) & \cdots
    &\phi_{\iota'(k)}(P_{k}) 
\end{matrix}\right|^{-1}
\left|\begin{matrix}
\phi_{\iota'(1)}(P_1) & \phi_{\iota'(2)}(P_1) & \cdots 
   &\phi_{\iota'(k)}(P_1) 
&  y_1^{r-1} \partial_{x_1} \\
\phi_{\iota'(1)}(P_2) & \phi_{\iota'(2)}(P_2) & \cdots
    &\phi_{\iota'(k)}(P_2) 
&  y_2^{r-1} \partial_{x_2}\\
\vdots & \vdots & \ddots& \vdots &\vdots\\
 \phi_{\iota'(1)}(P_{k}) & \phi_{\iota'(2)}(P_{k}) & \cdots
    &\phi_{\iota'(k)}(P_{k}) 
&  y_k^{r-1} \partial_{x_k}\\
\epsilon_1 & \epsilon_2  & \cdots & \epsilon_k & 0
\end{matrix}\right|,
$$
where $(\epsilon_1,\ldots ,\epsilon_k)$ is any $k$-tuple
of numbers. 
\end{lemma}

\begin{lemma} \label{lm:4.2}
For the open set ${\mathcal{U}}\subset \SSS^k(X)$ and
$(P_1, \ldots, P_k) \in {\mathcal{U}}$,
let $v^{(i)}:=w(P_i)$. If $(P_1, \ldots, P_{i-1}, $ 
$\infty, $ $P_{i+1},  \ldots, P_k)\in{\mathcal{U}}, $
we regard $ v^{(i)}_j$ as a function of  $v^{(i)}_g$
regardless of whether $g\in K_k$,  
$ v^{(i)}_j= v^{(i)}_j(v^{(i)}_g) $, and the following holds:
$$
\frac{\partial}{\partial v_j^{(i)}}
 =
\left((v_g^{(i)})^{-N(g) + N(j-1) + 2}(1 + d_{>}( v_g^{(i)}) \right)
 \frac{\partial}{\partial v_g^{(i)}}.
$$
\end{lemma}

\begin{proof}
Using Lemma \ref{lm:degu}
and the chain rule, 
$$
\frac{\partial v_j^{(i)}}{\partial v_g^{(i)}}
    =\left(v_g^{(i)}\right)^{N(g) - N(j-1) - 2}
    \left(1 + d_{>}( v_g^{(i)})\right) .
$$
\end{proof}

In (\ref{partialabel}) and (\ref{eq:pu_pv}),
we have seen differential operators with respect to
  the variables $u = \sum_{i=1}^k v^{(i)}$,
the label $\{\iota(1), \ldots, \iota(k)\}$,
and the notation
$u^{[k]} = \sum_{i=1}^k v^{(i)}$ and
$(\partial/\partial u^{[k]}_{\iota(j)})_{j=1,\ldots,k}$.
Assume now that
$\{\iota(1),\ldots, \iota(\ell)\}$ and
$\{g+\ell-k+1, \ldots, g-1, g\}$ are disjoint.
Then we also consider the differential operators
$(\partial/\partial u^{[\ell]}_{\iota(j)})_{j=1,\ldots,\ell}$
 with respect to $u^{[\ell]}=
\sum_{i=1}^\ell v^{(i)}$ for $\ell \le k$, and
the differential operators
$(\partial/\partial u^{[\ell;g]}_{j})_{j= g+\ell-k+1, \ldots, g-1, g}$
with respect to $u^{[\ell;g]}= \sum_{i=\ell+1}^k v^{(i)}$.
Now we give a transformation formula among them for a suitable open set
${\mathcal{U}}\subset \SSS^{k}(X)$, which is essentially the same
as (\ref{eq:pug=pugk}) in the proof of
 Lemma \ref{lemma:schurLambda} and implies
(\ref{eq:diff_remainder}) in the proof of
Proposition \ref{prop:gggggk}.

\begin{proposition} \label{prop:d_u}
For the open set ${\mathcal{U}}\subset \SSS^{k}(X)$ and
$(P_1, \ldots, P_{k}) \in {\mathcal{U}}$,
let $v^{(i)}:=w(P_i)$,
  $u^{[k]} := \sum_{i=1}^k v^{(i)}$,
  $u^{[\ell]} := \sum_{i=1}^\ell v^{(i)}$,
  $u^{[g;\ell]} := \sum_{i=\ell+1}^k v^{(i)}$,
and $\iota(k-j)=g-j$ $(j=0, \ldots, k-\ell -1)$. 
The change of basis  
$\displaystyle{{}^t\left(
\frac{\partial}{\partial u^{[\ell]}_{\iota(1)}},\ldots,
\frac{\partial}{\partial u^{[\ell]}_{\iota(\ell)}},
\frac{\partial}{\partial u^{[g;\ell]}_{g+\ell-k+1}}, \ldots,
\frac{\partial}{\partial u^{[g;\ell]}_{g}} \right)}$ 
to $\displaystyle{{}^t\left(
\frac{\partial}{\partial u^{[k]}_{\iota(1)}},\ldots,
\frac{\partial}{\partial u^{[k]}_{\iota(\ell)}},
\frac{\partial}{\partial u^{[k]}_{\iota(\ell+1)}},\right.}$ 
$\displaystyle{\left.\ldots,
\frac{\partial}{\partial u^{[k]}_{\iota(k)}}\right)}$
$=\displaystyle{{}^t\left(
\frac{\partial}{\partial u^{[k]}_{\iota(1)}},\ldots,
\frac{\partial}{\partial u^{[k]}_{\iota(\ell)}},
\frac{\partial}{\partial u^{[k]}_{g+\ell-k+1}},\ldots,
\frac{\partial}{\partial u^{[k]}_{g}}\right)}$
is given by a matrix 
$\cM_{u^{[k]},u^{[g;\ell]}}$,
$$
\left(\begin{matrix}
\phi_{\iota'(1)}(P_1) & 
 \cdots &\phi_{{\iota'(k)}}(P_1) \\
 \vdots & \ddots& \vdots\\
 \phi_{\iota'(1)}(P_{\ell}) &  \cdots &\phi_{\iota'(k)}(P_{\ell})\\
 \phi_{\iota'(1)}(P_{\ell+1}) &  \cdots &\phi_{\iota'(k)}(P_{\ell+1})\\
 \vdots & \ddots& \vdots\\
 \phi_{\iota'(1)}(P_{k}) &  \cdots &\phi_{\iota'(k)}(P_{k})
\end{matrix}\right)^{-1}
\left(\begin{matrix}
\phi_{\iota'(1)}(P_1) & \cdots &\phi_{\iota'(\ell)}(P_1) &  & & \\
 \vdots & \ddots& \vdots & \\
\phi_{\iota'(1)}(P_{\ell}) & \cdots &\phi_{\iota'(\ell)}(P_{\ell}) & & &\\
  &  & &\phi_{g+\ell-k}(P_{\ell+1}) & \cdots &\phi_{g-1}(P_{\ell+1})  \\
  &  &  &\vdots  & \ddots& \vdots \\
  &  & &\phi_{g+\ell-k}(P_{k}) & \cdots &\phi_{g-1}(P_{k})  
\end{matrix}\right) .
$$
\end{proposition}

\begin{proof} 
By considering the intermediate basis
$\displaystyle{{}^{t}\left(
r y_1^{r-1}\frac{\partial}{\partial x_1},\ldots,
r y_k^{r-1}\frac{\partial}{\partial x_k}\right)}$,
(\ref{eq:pu_px}) gives the result.
\end{proof} 

>From the expansion
 (\ref{eq:v_t}), we deduce the following result (we label it Lemma
because it is used in the proof of Lemma \ref{lemma:schurLambda}):
\begin{lemma} \label{lemma:d_uinfty}
The transition matrix $\cM_{u^{[k]},u^{[g;\ell]}}$ behaves like
$$
\cM_{u^{[k]},u^{[g;\ell]}}=
\begin{pmatrix}
1_{\ell} & \cC_{k,\ell}\\
    & 1_{k-\ell}
\end{pmatrix}
+ \cM_{>0}(v^{(\ell+1)}_{g}, \ldots, v^{(k)}_{g}). 
$$
Here $1_{\ell}$ is the $\ell \times \ell$ identity matrix,
$\cC_{k,\ell}$ is an $\ell \times (k-\ell)$ matrix which
depends only on
$u^{[\ell]}_{\iota(1)},\ldots, u^{[\ell]}_{\iota(\ell)}$,
and $\cM_{>0}(v^{(\ell+1)}_{g}, \ldots, v^{(k)}_{g})$
is a $k \times k$ matrix whose
entries are given by
$$
\sum_{i_\ell, \ldots, i_k,
i_\ell+\cdots+ i_k >0}
c_{i_\ell,\ldots,i_k}
{v^{(\ell+1)}_{g}}^{i_\ell}\cdots{v^{(k)}_{g}}^{i_k},
$$
where
$i_j$ ($j=\ell, \ldots, k$) is a non-negative integer, and
$c_{i_\ell,\ldots,i_k}$ is a function of
$u^{[\ell]}_{\iota(1)},\ldots, u^{[\ell]}_{\iota(\ell)}$.
\end{lemma}

\begin{proof}
By letting
$$
\Psi:=\left(\begin{matrix}
\phi_{\iota'(1)}(P_1) & \cdots &\phi_{{\iota'(k)}}(P_1)\\
   \vdots  & \ddots& \vdots \\
\phi_{\iota'(1)}(P_k) & \cdots &\phi_{{\iota'(k)}}(P_k)\\ 
\end{matrix}\right),
$$
$$
\Psi_a:=\left(\Psi_{i,j}\right)_{i=1,\ldots,\ell,j=1,\ldots,\ell}, \quad
\Psi_b:=\left(\Psi_{i,j}\right)_{i=1,\ldots,\ell,j=\ell+ 1,\ldots,k},
$$
$$
\Psi_c:=\left(\Psi_{i,j}\right)_{i=\ell+1,\ldots,k,j=1,\ldots,\ell}, \quad
\Psi_d:=\left(\Psi_{i,j}\right)_{i=\ell+1,\ldots,k,j=\ell+ 1,\ldots,k},
$$
$\cM_{u^{[k]},u^{[g;\ell]} }^{-1}$ is given by
$$
\left(\begin{matrix} \Psi_{a}^{-1} &  \\
 & \Psi_{d}^{-1} \end{matrix}\right)
\left(\begin{matrix} \Psi_{a} & \Psi_b \\
\Psi_c & \Psi_{d} \end{matrix}\right).
$$
Since the w-degree of every entry of
$\Psi_{d}$ is greater than that of every entry of
 $\Psi_{c}$,
$\Psi_{d}^{-1}\Psi_{c}$ vanishes if each $P_{j}$ 
($j=\ell+1, \ldots, k$) equals $\infty$.
Indeed, if $P_j = \infty$ ($j = \ell +1, \ldots, k$)
$(\cM_{u^{[k]},u^{[g;\ell]}}^{-1})_{i j} =\delta_{i j}$ 
for $(i, j= 1, \ldots,\ell)$ and
$(i = \ell + 1, \ldots,k, j = 1, \ldots, k)$  and thus
$(\cM_{u^{[k]},u^{[g;\ell]}})_{i j} =\delta_{i j}$
for $(i,j =1, \ldots, \ell)$
and $(i = \ell + 1, \ldots,k, j = 1, \ldots, k)$.
Further we have an expression,
$$
\cM_{u^{[k]},u^{[g;\ell]}} = 
\left(\begin{matrix} \Psi_{a} & \Psi_b \\
\Psi_c & \Psi_{d} \end{matrix}\right) ^{-1}
\left(\begin{matrix} \Psi_{a} &  \\
 & \Psi_{d} \end{matrix}\right).
$$
Let $\Psi^{(i,j)}$ be the minor of $\Psi$.
By a natural extension of the w-degree with respect to $P_i$ to
$(P_{\ell+1}, \ldots, P_{k})$ at $(\infty,\ldots, \infty)$ 
as a multi-index (\ref{eq:multi-index}),
$|\Psi_a||\Psi_d|$ is the largest-degree term in the expansion 
of the determinant, $|\Psi| = |\Psi_a||\Psi_d| + \cdots$. 

For the case $i=1, \ldots, \ell,j=\ell+1, \ldots, k$, we argue as follows.
$$
(\cM_{u^{[k]},u^{[g;\ell]}})_{i j} =
\sum_{i'=\ell+1}^{k}(-1)^{i+i'}
\phi_{g-(k-j)-1}(P_{i'}) \Psi^{(i',i)}/|\Psi|.
$$
Since 
$\sum_{i'=1}^{k}\Psi_{ii'}^{-1}\Psi_{i'j}
=\sum_{i'=1}^{k}(-1)^{i+i'}
\phi_{g-(k-j)-1}(P_{i'}) \Psi^{(i',i)}/|\Psi| = 0$,
we have
$$
(\cM_{u^{[k]},u^{[g;\ell]}})_{i j} =
-\sum_{i'=1}^{\ell}(-1)^{i+i'}
\phi_{g-(k-j)-1}(P_{i'}) \Psi^{(i',i)}/|\Psi|.
$$
$\Psi^{(i',i)}$ is expanded as 
$\Psi^{(i',i)} = \Psi_a^{(i',i)} |\Psi_d| + \cdots $ and
the first term has the  largest degree with respect to
the
expansion at $(P_{\ell+1}, \ldots, P_{k})$ at $(\infty,\ldots, \infty)$
with the multi-index convention.
When $P_{j'}$ approaches  $\infty$ for every $j' = \ell+1, \ldots, k$,
$(\cM_{u^{[k]},u^{[g;\ell]}})_{i j} $ becomes
$-\sum_{i'=1}^{\ell}(-1)^{i+i'}
\phi_{g-(k-j)-1}(P_{i'}) \Psi_a^{(i',i)}/|\Psi_a|$,
 which is a function only of $(P_1, \ldots, P_{\ell})$.
\end{proof}

\section{Vanishing of $\sigma$ on $\Theta^k$ $(0< k < g)$}\label{vanishing}

We finally give the vanishing order of $\sigma$, which we
obtain directly from Proposition \ref{prop:Nakayashiki}.
Indeed, it is determined by the vanishing of the
Schur function $\bs_\Lambda$, which is the limit of the $\sigma$-function
when $X$ approaches the singular curve $X_0$.
However, we can work with the $\theta$-function
of the nonsingular curve $X$, 
for which Theorems \ref{thm:JIF}, \ref{thm:RST}, \ref{thm:Fay} hold,
since we  use properties of the Schur function only 
to obtain certain coefficients in the multivariable
Taylor expansion of $\sigma$ and check that they are 
different from  zero.

We state  Riemann's singularity theorem (cf. \cite[VI.1]{ACGH}),
with the usual notation of $h^i$ for the dimension of the 
cohomology space $H^i$: \label{pg:cohomology}
\begin{theorem} \label{thm:RST} 
If $D_k$ belongs to $\SSS^k(X\backslash\infty) \setminus 
(\SSS^k_1(X)\cap \SSS^k(X\backslash\infty)) $, and we let
$$
	u := \int^{D_k}_{k\infty} \nuI,
$$

$$
	n_k := h^0(X, D_k + (g-k-1)\infty) 
  \equiv\#\{ \ell \ | \ 0 \le \ell, N(\ell)  \le g-k-1 \},
$$
then:
\begin{enumerate}
\item
 For every multi-index  $(\alpha_1, \ldots, \alpha_m)$ 
with $\alpha_i\in \{ 1, \ldots, g\}$ and $m < n_k$,
$$
	\frac{\partial^m}
        {\partial u_{\alpha_1} \ldots \partial u_{\alpha_m}}
            \sigma(u) = 0.
$$
\item
There exists a multi-index 
$I_\beta:=(\beta_1, \ldots, \beta_{n_k})$, which in general depends on
$D_k$,  
such that
\begin{equation}
	\frac{\partial^{n_k}}
         {{\partial u}_{\beta_1} \ldots \partial u_{\beta_{n_k}}}
            \sigma(u) \neq 0.
\label{eq:Rvn-nvn}
\end{equation}
\end{enumerate}
\end{theorem}

\begin{remark} {\rm Since the $\sigma$-function is either
even or odd, in Theorem \ref{thm:RST} we can replace
the assumption for $D_k$ with
$D_k\in w(\SSS^k(X\backslash\infty) \setminus 
(\SSS^k_1(X)\cap \SSS^k(X\backslash\infty))
\cup [-1] 
w(\SSS^k(X\backslash\infty) \setminus 
(\SSS^k_1(X)\cap \SSS^k(X\backslash\infty))$, and we also 
extend via the operator $[-1]$
the defining set of $u$ in the Theorem,
$u\in
\kappa^{-1}(\mathcal{W}^k) \subset \kappa^{-1}(\Theta^k),$
with the appropriate (extended) excluded subsets as in the Theorem.
}
\end{remark}

\bigskip

For $D_k $ and $n_k$ as in  Theorem \ref{thm:RST},
 Fay \cite[Theorem 1.2]{F2} proved the following 
 (cf. also  \cite{BV, SW}):
\begin{theorem} \label{thm:Fay}
Let $\nu^+_i$ $(0 \le \nu^+_1 < \nu^+_2 < \cdots < \nu^+_{n_k} \le g-1)$ 
be such that
$$
	h^0(X, D_k + (g - k - \ell - 1)\infty) = {n_k} - i +1 
          \quad \mbox{for} \ \ell = \nu^+_i,
$$
$$
	h^0(X, D_k + ( g -  k - \ell - 1)\infty) \le {n_k} - i
           \quad \mbox{for} \ \ell > \nu^+_i,
$$
 with $\nu^-_i$  defined for $[-1]D_k$ as $\nu^+_i$ is for  $D_k$ and 
$$
	N_k := {n_k} + \sum_{i=1}^{n_k} (\nu^+_i + \nu^-_i).
$$
Let $\hnuI$ be the normalized basis of holomorphic one-forms
$$
	\hnuI := {\omega'}^{-1} \nuI.
$$
For $P$ and $Q$ in $X$ and $e -\omega_\mathrm{R}\in 
\Theta^k\setminus
\left(\Theta^k_1\cup \Theta^{k-1}\right)$ and for every $t \in \CC$,
\begin{equation}
	\theta\left( t \int^P_Q \hnuI+ e\right) = 
       t^{n_k} 
       \prod_{i=1}^{n_k} 
      \left(
       \prod_{k=1}^{\nu^+_i}(t - k) 
       \prod_{\ell=1}^{\nu^-_i}(t-\ell)
      \right) E(P, Q)^{N_k} \Phi(P, Q, t),
\label{eq:Faytheta}
\end{equation}
where 
$\theta(z)$ is the  Riemann $\theta$-function, 
 $E(P, Q)$ is the prime form and $\Phi(P, Q, t)$ is an entire function
of $t$ for all $P\in X$ near $Q$, with
$$
       \Phi(P, P, t) =
     \frac{1}{N_k!} \sum_{i_1, \ldots, i_{N_k}=1}^g
      \frac{\partial^{N_k}\theta(e)}
      {\partial z_{i_1} \ldots \partial z_{i_{N_k}} }
      dz_{i_1}(P) \ldots dz_{i_{N_k}}(P) \neq 0
$$
where $e = (z_1, \ldots, z_g)$.
\end{theorem}

Note that by $[-1]D_k$ here we mean simply a divisor linearly
equivalent
to $K-D_k$, so it has an order $k'$ for which the overriding
assumption of this section,  
$0<k^{\prime}<g-1$, is not satisfied; however,
 the statement is ultimately concerned
with the images under the Abel map.  Note also that, unlike $n_k$, the 
$\nu$'s here depend \textit{a priori} on the specific divisor $D_k$,
but see Corollary \ref{Fay-computed}.
The following corollary follows from  intersection theory
as in  \cite{BV}. In this article, we  prove it
as a corollary of the above theorem.

\begin{corollary} \label{cor:Fay}
For all $1\le k\le g-1$
(implicit in \ref{thm:RST}), $u^{[k]} \in \Theta^k\setminus
\left(\Theta^k_1\cup \Theta^{k-1}\right)$,
$u^{[g]} \in \CC^g$,
$v\in \WW^{1}$, and $t \in \mathbf{C}$ $(0<|t|<1)$, we have
\begin{enumerate}

\item
$\displaystyle{
	\left.\frac{\partial^{\ell}}
         {\partial {v_g}^{\ell}} \sigma(t v + u^{[k]})\right|_{v = 0} =
         0, \  \ell < N_k; \quad
	\left. \frac{\partial^{N_k}}
         {\partial {v_g}^{N_k}} 
          \sigma(t v + u^{[k]})\right|_{v = 0} \neq 0, 
}$ and

\item
$\displaystyle{
	\left.\frac{\partial^{\ell}}
         {\partial {u^{[g]}_g}^{\ell}} \sigma(u^{[g]})
         \right|_{u^{[g]} = u^{[k]}} = 0, \  \ell < N_k; \quad
	\left. \frac{\partial^{N_k}}
         {\partial {u^{[g]}_g}^{N_k}} 
          \sigma(u^{[g]})\right|_{u^{[g]} = u^{[k]}} \neq 0. 
}$

\end{enumerate}
\end{corollary}

\begin{proof}  We let $Q = \infty$,
$e = {\omega'}^{-1} u^{[k]} - \omega_\mathrm{R}$
in Theorem \ref{thm:Fay}, and we let
$v = {\omega'}\int^P_\infty \hnuI$.
We note that for 
a generic complex number $t$
there exist $P_1, \ldots, P_g \in X$
such that $tv = w(P_1, \ldots, P_g)$, and thus
$\theta(t {\omega'}^{-1}v + e)$ and
$\theta({\omega'}^{-1} u^{[g]} - \omega_\mathrm{R})$ do not
vanish in general where $u^{[g]} := tv + u^{[k]}$.
By differentiating the left-hand side of
(\ref{eq:Faytheta}), using the chain rule
and evaluating at $P=\infty$ where $v$ is the zero vector,
there exist $b_{\ell, i} \in \CC$ satisfying
\begin{equation}
\frac{\partial^{\ell}}
      {\partial {v_g}^{\ell}} 
\theta(t {\omega'}^{-1}v + e) \Bigr|_{P=\infty}
= \sum_{i=0}^\ell b_{\ell,i} t^i,
\label{eq:thetav_g}
\end{equation}
where
$$
b_{\ell, \ell}
:=\frac{\partial^{\ell}}
{\partial {(tv_g)}^{\ell}} \theta(t{\omega'}^{-1}v + e) \Bigr|_{P=\infty}
=\frac{\partial^{\ell}}
{\partial {u^{[g]}_g}^{\ell}} 
\theta({\omega'}^{-1}u^{[g]}- \omega_\mathrm{R})
 \Bigr|_{u^{[g]}=u^{[k]}}.
$$
This relation is obtained as follows: Since
$\theta(t{\omega'}^{-1}v + e)$ is an entire function of the vector $tv$
($\omega^\prime$ is an invertible coordinate change),
 it has an expansion
$\theta(t{\omega'}^{-1}v + e) = \sum_{\beta\ge0} a_\beta t^{|\beta|}v^\beta$,
where $\beta$ is a non-negative $g$-tuple, 
$\beta = (\beta_1, \ldots, \beta_g)$
with the conventions $|\beta| = \beta_1+\ldots+\beta_g$,
$v^\beta = v_1^{\beta_1}\ldots v_g^{\beta_g}$, and
$\beta \ge 0$ for every $\beta_i \ge 0$, 
 and each $a_\beta$ is a complex
number depending upon $e$, $\omega'$ and $\omega''$.
Now we compute up to lower-order terms in $v_g$.
Recall that the orders of zero at $\infty$ of the chosen 
basis are decreasing (Section \ref{curves}):
$\mathrm{deg}_{\mathrm{w}^{-1}}(\nu_i^I)=2g-N(i-1)-2$ and
$v_i =  \frac{1}{2g - N(i-1)-1} 
v_g^{2g - N(i-1)-1}(1+ d_>(v_g))$ around $v=0$.
There exist coefficients
$\hat a_\beta(e)$ and
$\tilde a_{\beta,i}$ satisfying
\begin{equation*}
\begin{split}
\frac{\partial^{\ell}}
      {\partial {v_g}^{\ell}} 
\theta(t {\omega'}^{-1}v + e) \Bigr|_{P=\infty}
     & = 
\frac{\partial^{\ell}}
      {\partial {v_g}^{\ell}} 
\sum_{\beta\ge0} a_\beta t^{|\beta|} v^\beta \Bigr|_{v=0} \\
&= \frac{\partial^{\ell}}
      {\partial {v_g}^{\ell}} 
\sum_{\beta\ge0} \hat{a}_\beta t^{|\beta|} v_g^{|\wg(\beta)|}
\left(1 + \sum_{i=1} \tilde a_{\beta,i} v_g^i\right) \Bigr|_{v_g=0} \\
\end{split}
\end{equation*}
which is equal to the right-hand side of (\ref{eq:thetav_g}),
and 
$\frac{\partial^{\ell}}
      {\partial {(tv_g)}^{\ell}} 
\theta(t{\omega'}^{-1}v + e) \Bigr|_{v=0} =
b_{\ell, \ell} = \ell! \hat a_{(0,\ldots,0,\ell)}$,
since the conditions
$|\wg(\beta)|=\ell$ and $|\beta|=\ell$ mean
$\beta=(0, \ldots,0, \ell)$. 

We now consider the derivative
of the right-hand side of (\ref{eq:Faytheta}) with respect to $v_g$,
noting that  
$E(P, \infty) = t_\infty(1 + d_{>}(t_\infty))$
and $t_\infty = v_g(1 + d_{>}(v_g))$.
Then for $\ell < N_k$,
the derivative of the right-hand side of (\ref{eq:Faytheta}) 
vanishes
when $P \to \infty$ or $v_g$ vanishes. 
In  the case that $\ell$ agrees with $N_k$,
it contains a term consisting of
 $(\partial_{v_g}E(P, \infty))^{N_k} = (1 + d_{>}(t_\infty))$ times a
non-zero number (use Theorem \ref{thm:Fay}).
Hence for 
$\ell < N_k$, we have
$$
\left.\frac{\partial^{\ell}}
{\partial {v_g}^{\ell}} \theta(\omega^{\prime -1}(tv + u^{[k]})
-\omega_\mathrm{R})
      \right|_{v = 0} = 0, \quad
	\left. \frac{\partial^{N_k}}
         {\partial {v_g}^{N_k}} 
          \theta(\omega^{\prime -1}(tv + u^{[k]})-\omega_\mathrm{R})\right|_{v = 0} \neq 0.
$$
Since $\sigma$ is associated to $\theta$ through
multiplication by an exponential quadratic in the variable, they
have the same vanishing order (the derivatives up to the
order differ through multiplication by an invertible matrix),
and the first statement of the Corollary is proved.
Statement 2. is proved  
 by comparing the coefficients $t^\ell$ in 
the derivative of both sides of (\ref{eq:Faytheta});
$b_{\ell,\ell}$
vanishes for $\ell<N_k$ and does not vanish for $\ell=N_k$.
\end{proof}

\begin{remark}
{\rm When computing a given number of derivatives of $\sigma$,
as opposed to the order of a singular point of the theta
divisor as in \cite[\textit{loc. cit.}]{ACGH},
we need to stay away from lower strata (recall that the
derivatives of $\sigma$ depend on deforming a given point along the curve).
Thus, the exclusion of the point at $\infty$ from the
divisors in  $\mathcal{S}^k(X)$, and of the set $\Theta^{k-1}$.
For example, in genus 2 the  locus $\mathcal{S}_1^2(X)$ is empty,
whereas $\mathcal{S}^{g-1=1}(X)$ is the curve. In \cite{MP}, 
as of Section 2 we extended the functions to be $\mathbb{P}^1$-valued,
but beginning 
with Proposition 4.4, we excluded the special divisors
$\mathcal{S}_1^k(X)$ (for which the right-hand side of
\cite[Proposition 4.4]{MP} could be infinite, if say 
$w(P)-w(P_1^\prime,...,P_g^\prime)$ is a zero of $\sigma$,
a condition on the speciality of $\sum_{i=1}^gP_i^\prime -P$, but
$w(P)-w(P_1,...,P_g)$ is not,
yet the points $P,Q$ are distinct from the given poles $P_i,P_j^\prime$
so that the left-hand side is a finite number); here, we need
(cf. Section \ref{CoordinateChange}, e.g.)
a domain capable for a given 
 number of points to move along the curve, so we may need positive divisors of
 degree less than $g$, thus special.
}
\end{remark}

We can rephrase Theorem \ref{thm:Fay} as follows:
\begin{corollary}\label{Fay-computed}
Let 
\begin{equation*}
\begin{split}
M_k &:= \{ g - N(\ell) -k -1 | 
g - N(\ell) - k - 1 \ge 0, \ \ell = 0, 1, \ldots \ \}, \\
\overline{M}_k &:=
 \{ g  - N(\ell+k) + k - 1 |
  g - N(\ell+k) + k - 1  \ge 0, \ \ell  =0, 1, \ldots \ \}.
\end{split}
\end{equation*}
Then the quantities in  Theorem \ref{thm:Fay} are given by
$$
\{\nu^+_i \}_{1\le k\le n_k}=  M_k, \quad \{\nu^-_i \}_{1\le k\le n_k}
 =  \overline{M}_k,
$$
 and 
\begin{equation}
N_k = \sum_{\ell=0}^{n_k - 1} (2g - N(\ell) - N(k + \ell) - 1).
\label{eq:N_k1}
\end{equation}
\end{corollary}

\begin{proof}
$M_k$ is obtained  in a straightforward way. 
When we consider $\overline{M}_k$, we need to know the dimension
of $H^0(X, [-1]D_k + (g - N(k) + k - 1)\infty )$.
Let $D_k = P_1 + \ldots + P_{k}$ where each $P_i \in (X\backslash\infty)$.
Using the notation in Definition \ref{def:mul}, we have
$$
	H^0(X, [-1]D_k + (g - N(k) + k - 1)\infty ) 
\ni \frac{\mu_{\ell+k}(P; P_1, \ldots, P_{\ell+k})}
	                    {\mu_{k}(P; P_1, \ldots, P_{k})},
\quad \ell = 0, \ldots, n_k - 1,
$$
which gives $\overline{M}_k$ and $N_k$ explicitly.
\end{proof}
\bigskip

For the expression (\ref{eq:N_k1}), cf. also \cite{BV, SW}.

\begin{proposition}
The following numerical identity holds:
\begin{equation}
	{N_k} = \sum_{i=k}^{g-1} (N(g) - N(i) + i - g) =
	\sum_{i=k}^{g-1} (g - N(i) + i) =
	\sum_{i=k+1}^{g} \Lambda_i .
\label{eq:N_k2}
\end{equation}
\end{proposition}

\begin{proof}
For a given $k > 1$, we have two cases:
$n_{k-1} = n_k$, or $n_{k-1} + 1= n_k$. The former case means
that $g-k$ is  the $(g- k - n_k + 1)$-th gap whereas 
the latter case implies that
$g-k$ is in the $n_k$-th semigroup interval.
As the dimensions of  the linear systems
$\mathcal L$ and ${K}_X \mathcal L^{-1}$ differ by the 
degree of $\mathcal{L}$ for any
line bundle $\mathcal{L}$ 
(cf. Remark \ref{serreduality}), in the latter case
$2g-2 -(g-k)$ is in the $(g-1)-(g- k - n_k + 1)$-th semigroup interval.
 For each case we then have the relation:
\begin{gather*}
\begin{array}{lc}
n_{k-1} = n_k: & N(k + n_k -1) = g + k - 1, \\
n_{k-1} = n_k + 1: & N(n_k) = g - k. \\
\end{array}
\end{gather*}
When $k=g-1$, $n_{g-1} = 1$  and
$N(0) = 1$, so (\ref{eq:N_k1}) and (\ref{eq:N_k2})
agree.
Next, assume that for a given $k$  the right-hand sides
of the two expressions are equal.
If $n_{k-1} = n_k$
the right-hand side of (\ref{eq:N_k1}) for the $k - 1$ case is
given by
$$
N_{k-1} =
N_k - N(k-1) + N(k+ n_k - 1),
$$
which is written using the above relation,
$$
N_{k-1} =
N_k - N(k-1) +  g + k - 1.
$$ 
When $n_{k-1} = n_k + 1$,
the right-hand side of (\ref{eq:N_k1}) for the $k - 1$ case is
$$
N_{k-1} =
 N_k - N(k-1) - N(n_k) + 2g - 1 
$$
which is also written using the above relation,
$$
N_{k-1} =
N_k - N(k-1) +  g + k - 1.
$$
We conclude that (\ref{eq:N_k1}) and (\ref{eq:N_k2})
are equal.
\end{proof}

\begin{corollary} \label{cor:6.7} 
\begin{enumerate}
\item
The number $n_k$ is the first component of the node $(n_k, m_k)$,
$m_k:=n_k+k$, 
encountered  on the rim hook \cite[Definition 4.10.1]{Sa} of
the diagram, in a right-to-left, up-to-down path, starting
with $n_{g-1}$ in row 1. For example, 
 in the diagram of the (5,7) curve (Table 2.1)
the nodes that correspond to $(n_k, m_k)$ are: (1,12), (1,11), 
(1,10), (1,9), (1,8), (2,8), (2,7), (3,7), (3,6), (3,5), (4,5), (4,4).

\item
$N_k$ is the number of cells in the rows of
the diagram  from row $k+1$ to $g$.

\item
For $k=g-r, \ldots, g-1$,
$N_{k} = \mwdeg(u_{k+1})$ and $n_{k}=1$ 
\end{enumerate}
\end{corollary}

\begin{remark}
{\rm{
Note that for the hyperelliptic  case
$$
	{N_k} = (g -k) (g-k+1)/2. 
$$
One of the authors (S.M.) learned this relation 
from Victor Enolskii who discovered it by numerical computations
in 2005. This turns out to be 
 a corollary of Theorem \ref{thm:Fay} \cite{BV,F2} but
the present study originated with Enolskii's communication of his discovery.
Birkenhake and Vanhaecke \cite{BV} showed that this number
is a sum of
hyperosculation degrees for 
embeddings  of the curve into 
 Grassmannians,
defined by linear subseries of the complete
linear series that defines the Weierstrass gaps at a point.  
}}
\end{remark}

We introduce  `truncated Young diagrams' 
$\Lambda^{(k)}:= (\Lambda_1, \ldots, \Lambda_{k})$ and
$\Lambda^{[k]}:= (\Lambda_{k+1}, \ldots, \Lambda_{g})$.
\label{pg:Lambdak}
We note  that 
 this truncated Young diagram was considered in
\cite{BEL} for the $k=g-1$ case.

Corollary \ref{cor:6.7} gives:
\begin{corollary} \label{cor:Schur}
\begin{enumerate} 
\item $N_k = |\Lambda^{[k]}|$.

\item $n_k$ is the rank of the partition of $\Lambda^{[k]}$. Thus,
we can read $n_k$ in the diagram dually to corollary
\ref{cor:6.7} 3., by numbering $k$ the boxes on the rim,
starting with $k=g-1$ in row $g$; $n_k$ is the number of boxes at the
left of, and including, the one containing the number $k$. 

\item For the characteristics of the partition of $\Lambda^{[k]}$,
$(a_1, \ldots, a_{n_k}; b_1, \ldots, b_{n_k})$, 
  $N_k = \sum_{i=1}^{n_k}(a_i+b_i+1)$.

\end{enumerate} 
\end{corollary} 

Bukhshtaber,  Le\v{\i}kin and \`Enol'ski\v{\i}
\cite{BEL} showed that the $\sigma$ function over the singular
curve $X_0$ given by $y^r = x^s$ is identified with a Schur function,
cf. Proposition \ref{prop:Nakayashiki}.
Although we make use of formulas that hold for the $\sigma$-function,
we implicitly go through  the singular curve $X_0$ 
to pick up combinatorial results from the theory of  Schur functions 
using Taylor expansions.

Using the Schur polynomials,
$$
\bold{s}_{\Lambda^{(k)}}(t) := \frac{|t_j^{\Lambda_i+k-i}|_{1\le i,j\le k}}
                            {|t_j^{i-1}|_{1\le i,j\le k}},\quad
\tilde \bs_{\Lambda^{[k]}}(t) := 
\frac{|t_{k+j}^{\Lambda_{k+i}+g-k-i}|_{1\le i,j\le g-k}}
     {|t_{k+j}^{i-1}|_{1\le i,j\le g-k}},\quad
$$
and letting $T^{(k)}_j := \frac{1}{j}(t_1^j + \cdots +t_k^j)\equiv
  T^{\LA 1,k\RA}_j$, and $T^{(g;k)}_j :=\frac{1}{j}
 (t_{k+1}^j + \cdots +t_g^j)\equiv T^{\LA k+1,g\RA}_j$,
\label{pg:T^k_j}
we define the following quantities,
$$
S_{\Lambda^{(k)}}(T^{(k)}) := \bold{s}_{\Lambda^{(k)}}(t), \quad
\tilde S_{\Lambda^{[k]}}(T^{(g;k)}) := \tilde \bs_{\Lambda^{[k]}}(t).
$$
The notation above for the left-hand sides 
$S_{\Lambda^{(k)}}(T^{(k)})$ and $\tilde S_{\Lambda^{[k]}}(T^{(k)})$ 
is consistent with Proposition \ref{prop:SchurC},
because they are functions of the $T^{(k)}$.
For a given $(t_\ell)_{\ell = 1, 2, \ldots, k} \in \CC^k$,
letting $u^{[k]} := (u^{[k]}_i)_{i=1, \ldots, g} \in \CC^g$ be
defined by $u_i^{[k]} := T^{(k)}_{\Lambda_i+g-i}$,
we can also define consistently,
in view of Lemma \ref{lemma:schurLambda} proven below:
$$
\bs_{\Lambda^{(k)}}(u^{[k]}) = 
S_{\Lambda^{(k)}}(T^{(k)})|_{T^{(k)}_{\Lambda_i+g-i} = u_i^{[k]}}.
$$
We also introduce the complete symmetric polynomial 
$h^{(g;k)}_n\equiv h^{\LA k+1,g\RA}$ 
such that
$h^{(g;k)}_{n\ge1} = (-1)^{n-1} (n-1)! T^{(g;k)}_n + \cdots $,
$h^{(g;k)}_{0} = 1$ and $h^{(g;k)}_{n<0} = 0$
(in particular, $h^{(g;k)}_n$ depends only on $t_{k+1},...,t_g$).

\begin{lemma} \label{lemma:schurLambda}
For a given $(t_\ell)_{\ell = 1, 2, \ldots, k} \in \CC^k$,
we let  $u^{[k]} := (u^{[k]}_i)_{i=1, \ldots, g} \in \CC^g$
be defined by $u_i^{[k]} := T^{(k)}_{\Lambda_i+g-i}$,
and
 $u^{[g;k]} := (u^{[g;k]}_i)_{i=1, \ldots, g} \in \CC^g$
be defined by $u_i^{[g;k]} := T^{(g;k)}_{\Lambda_i+g-i}$.
For brevity, we denote by $I$ a sequence 
of indices (among $\{ k+1,...,g\}$) which may be repeated,
and the notation $i\in I$ means that $i$ runs through
the sequence with given repetitions, if any; the order of the
indices in the sequence is irrelevant.
One such sequence for which $I$ has the smallest number of elements
$n_k$ and the sum of the degrees in $t$ (each variable $t_\ell$ having
degree 1),
  $\sum_{i\in I}{deg}
(u_i)=\sum_{i\in I}(\Lambda_i+g-i)=N_k$ is also minimum,
is given in the proof. 
We use the notation:  $\varepsilon_{\Lambda,I}:= \varepsilon'
\left(\prod_{i\in I} (\Lambda_{i} + g - i)!\right)^{-1}$,
where $\varepsilon^\prime$ is a plus or minus sign depending on
$\Lambda^{(g)}$ 
and $k$.
\begin{enumerate}
\item
There exists some (possibly non-unique)  finite sequence
$I$ (whose length  may not be unique 
but is at least  $n_k$)
such that,
by using the decomposition $u^{[g]} = u^{[k]} + u^{[g;k]}\in \CC^g$,
$$
\bs_{\Lambda^{(k)}}(u^{[k]}) = 
\varepsilon_{\Lambda,I}
\left(\prod_{i\in I}\frac{\partial}{\partial u_i^{[g;k]}}\right)
\bs_{\Lambda^{(g)}}(u^{[g]})
\Bigr|_{u^{[g]}=u^{[k]}}.
$$ 
In performing this partial derivative, 
we restrict the independent variables to a subspace $\CC^{g-k}\subset \CC^g$ spanned by
$\displaystyle{\left\{
\frac{\partial}
{\partial u_i^{[g;k]}}
\right\}_{i=k+1,\ldots,g}}$.

\item
There exists some (possibly non-unique)  finite sequence
$I$ (whose length  may not be unique 
but is at least  $n_k$)
such that,
$$
\bs_{\Lambda^{(k)}}(u^{[k]}) = 
\varepsilon_{\Lambda,I}
\left(\prod_{i\in I}\frac{\partial}{\partial u_i^{[g]}}\right)
\bs_{\Lambda^{(g)}}(u^{[g]})
\Bigr|_{u^{[g]}=u^{[k]}}.
$$
\end{enumerate}
\end{lemma}

\begin{proof} 
We treat both statements 1. and 2. together.

Using the modified Jacobi-Trudi determinant formula
in Lemma \ref{lemma:schur_Hgk},
$$
	\bs_{\Lambda}(t_1,\ldots,t_g) :=
\left|\begin{matrix}
	H_{k, k} & H_{k, g-k} \\ 
	H_{g-k, k} & H_{g-k, g-k} \\ 
\end{matrix} \right|
$$
where       
\begin{equation*}
\begin{split}
H_{k,k}(t_1,\ldots,t_g) 
               &:=(h_{\Lambda_i+j-i}(t_1,\ldots,t_g))_{1\le i, j \le k}, \\
H_{k,g-k}(t_1,\ldots,t_g) 
               &:=(h_{\Lambda_i+j-i}(t_1,\ldots,t_g))_{1\le i\le k, 
                                            k+1\le j \le g}, \\
H_{g-k,k}(t_{k+1},\ldots,t_g) 
	&:=(h_{\Lambda_i+j-i}(t_{k+1},\ldots,t_g))_{k+1\le i\le g, 
                                            1\le j \le k}, \\
H_{g-k,g-k}(t_{k+1},\ldots,t_g) 
	&:=(h_{\Lambda_i+j-i}(t_{k+1},\ldots,t_g))_{k+1\le i\le g, 
                                            k+1\le j \le g},
\end{split}
\end{equation*}
the right-hand side of 1., without the constant 
$\varepsilon_{\Lambda,I}$, is given by
\begin{equation}
\left(
\prod_{i\in I} \frac{\partial}{\partial{u^{[g;k]}_i}}
\right)
\left|\begin{matrix}
	H_{k, k} & H_{k, g-k} \\ 
	H_{g-k, k} & H_{g-k, g-k} \\ 
\end{matrix} \right|
\Bigr|_{t_j=0: j =k+1, \ldots, g}.
\label{eq:puH}
\end{equation}
Since we have $\bs_{\Lambda^{(k)}}(u^{[k]})=|H_{k, k}(t_1,\ldots,t_k)|$,
 we are to consider the derivative of 
$$
 \tilde \bs_{\Lambda^{[k]}}(t)  
	=| H_{g-k, g-k}(t_{k+1}, \ldots, t_g) |;
$$
we will settle one case, in which $I$ has the smallest number of
elements and least degree in $t$ as given in the statement.

Indeed, by producing  a sequence $I$ such
that $I$  satisfies
\begin{equation}
\left(
\prod_{i\in I}
 \frac{\partial}
 {\partial u^{[g;k]}_i}
\right)
\left|\begin{matrix}
H_{g-k, g-k} 
\end{matrix} \right|
\Bigr|_{t_j=0: j =k+1, \ldots, g}=\varepsilon
\prod_{i\in I} (\Lambda_{i} + g - i)!
\label{eq:5-11a}
\end{equation}
and for any proper subsequence $J\subset I$,
\begin{equation}
\left(
\prod_{i\in J}
 \frac{\partial}
 {\partial u^{[g;k]}_i}
\right)
\left|\begin{matrix}
H_{g-k, g-k} 
\end{matrix} \right|
\Bigr|_{t_j=0: j =k+1, \ldots, g}=0,
\label{eq:5-11b}
\end{equation}
then
\begin{equation}
\left(
\prod_{i\in I}
 \frac{\partial}
 {\partial u^{[g;k]}_i}\right) 
\bs_{\Lambda}(t_1,\ldots,t_g)\Bigr|_{t_j=0: j =k+1, \ldots, g}=
\varepsilon
\prod_{i\in I} (\Lambda_{i} + g - i)!\bs_{\Lambda}^{(k)}(u^{[k]})
 + \mbox{lower-order}.
\label{eq:5-11c}
\end{equation}
 The lower-order (in $t$) term vanishes for the following reasons:
firstly,  $H_{g-k,k}(t_{k+1},\ldots,t_g)$ becomes the zero matrix 
and $H_{g-k, g-k}$ becomes a matrix whose every entry is zero except 
$h_0(\equiv 1)$   when 
$t_j=0: j =k+1, \ldots, g$. The derivative lowers the order, 
so we have (\ref{eq:5-11a}) and (\ref{eq:5-11b}).
The vanishing order of the entries in each column in
$H_{g-k,k}$ is  larger than that of the entries in
$H_{g-k, g-k}$, thus the lowest-degree property of $I$
means that any term in the lower-order part of (\ref{eq:5-11c})
 which comes from  the derivative
of an entry in $H_{g-k, k}$ in (\ref{eq:5-11c}) vanishes.
Secondly, due to (\ref{eq:5-11b}), any term which comes from the derivative
of an entry in $H_{k, k}$ or $H_{g-k, k}$ gives no 
contribution to the right-hand side in (\ref{eq:5-11c}).

To find $I$ as in (\ref{eq:5-11a}) and (\ref{eq:5-11b}),
we consider the pattern:
 \begin{gather*}
\begin{split}
s_{\Lambda^{[k]}}(t) &=
\left|\begin{matrix}
h^{(g;k)}_{\Lambda_{k+1}} & h^{(g;k)}_{\Lambda_{k+1}+1} & \cdots &
h^{(g;k)}_{\Lambda_{k+1}+g-k-2} & h^{(g;k)}_{\Lambda_{k+1}+g-k-1} \\ 
h^{(g;k)}_{\Lambda_{k+2}-1} & h^{(g;k)}_{\Lambda_{k+2}} & \cdots &
h^{(g;k)}_{\Lambda_{k+2}+g-k-3} & h^{(g;k)}_{\Lambda_{k+2}+g-k-2} \\ 
\vdots &\vdots&\ddots&\vdots&\vdots\\
h^{(g;k)}_{\Lambda_{g-1}-g+k+2} & h^{(g;k)}_{\Lambda_{g-1}-g+k+3} & \cdots &
h^{(g;k)}_{\Lambda_{g-1}} & h^{(g;k)}_{\Lambda_{g-1}+1} \\ 
h^{(g;k)}_{\Lambda_{g}-g+k+1} & h^{(g;k)}_{\Lambda_{g}-g+k+2} & \cdots &
h^{(g;k)}_{\Lambda_{g}-1} & h^{(g;k)}_{\Lambda_{g}} \\ 
\end{matrix}
\right|\\
& =
\left|\begin{matrix}
\cdots & &  \cdots& & &  & & & h^{(g;k)}_{\Lambda_{k+1}+g-k-1} \\ 
\cdots & &  \cdots& &h^{(g;k)}_{\Lambda_{k+2}+g-k-\ell'}  &  &\\ 
\cdots & &  \cdots& & &  & &\cdots &  \\ 
\cdots & h^{(g;k)}_{0} &  & &  & &&  &\\ 
 \vdots &\vdots&\ddots& \vdots& \vdots&\vdots\\
  &  & \cdots& h^{(g;k)}_{0} & h^{(g;k)}_1 & & &  \\ 
&   &  & \cdots&  & h^{(g;k)}_{0} & & &  \\ 
&   &  & \cdots& &   &\ddots & &  \\ 
&  &  & \cdots& &   & &  h^{(g;k)}_{0}    & h^{(g;k)}_1\\ 
\end{matrix}
\right| .
\end{split}
\end{gather*}

\noindent Since the elements in the lower-left  part vanish,
we can reduce the determinant to a combination of ones of smaller size.
The pattern of $h^{(g;k)}_{0}$'s has the property that  there is one
situated $\Lambda_i$ entries to the
left of the diagonal elements in each 
 $i$-th row for $ i > n_k$ because of  the configuration
of the lower boundary of the diagram $\Lambda^{[k]}$.
In the determinant calculation, a series of $h_0^{(g;k)}$ 
in the diagonal direction occurs in a unique position and
 contributes to the determinant  a factor of 1. In order to 
omit that factor,
we observe the following
properties about the pattern of the matrix.

We define a sequence $(i_\ell ,d_\ell )_{\ell=1, \ldots , n_k' \le n_k}$
by the following conditions:
 $i_1 = g,\ d_1=1$,   $i_\ell >i_{\ell +1}$, $d_\ell >0$, and
$i_\ell$ is
the largest number  satisfying 
$\Lambda_{i_\ell +\sum_{\ell^\prime <\ell}d_{\ell^\prime}} = 
\Lambda_{i_\ell +\sum_{\ell^\prime <\ell}d_{\ell^\prime}+1} +d_\ell$.
By using the  sequence $(i_\ell ,d_\ell )_{\ell=1, \ldots ,n_k' \le n_k}$, 
and   
denoting the elements of the matrix by $a_{ij}$, 
we define a hierarchy of submatrices starting from the above matrix 
whose determinant is  $\bs_{\Lambda^{[k]}}(t)$,
$$
  H_1:=(a_{ij})_{1 \le i,j \le g-k}\equiv H_{g-k, g-k}(t_{k+1}, \ldots, t_g),
$$
\begin{gather*}
\begin{split}
  H_{\sum_{\ell^\prime <\ell}d_{\ell^\prime}+1}&:=(a_{ij})_{
 \sum_{\ell^\prime <\ell}d_{\ell^\prime} +1 \le i \le 
\sum_{\ell^\prime <\ell}d_{\ell^\prime}+i_\ell -k, \ 
1 \le j \le  i_\ell-k},\\
  H_{\sum_{\ell^\prime <\ell}d_{\ell^\prime}+2}&:=(a_{ij})_{
 \sum_{\ell^\prime <\ell}d_{\ell^\prime}+2 \le i \le 
\sum_{\ell^\prime <\ell}d_{\ell^\prime}+i_\ell -k, \ 
1 \le j \le  i_\ell-k-1},\\
\ldots& \quad \ldots,\quad \\
  H_{\sum_{\ell^\prime \le\ell}d_{\ell^\prime} }&:=(a_{ij})_{
 \sum_{\ell^\prime \le\ell}d_{\ell^\prime}  \le i \le 
\sum_{\ell^\prime \le\ell}d_{\ell^\prime}+i_\ell -k, \ 
1 \le j \le  i_\ell-k-d_\ell +1}.\\
\end{split}
\end{gather*}
The above sequence of matrices stops
when the number of entries 
becomes negative.
As  will be explained  in the proof of Lemma \ref{lmm:ggggg},
these matrices are directly related to the Schur functions of 
the Young diagrams associated with $\Lambda^{[k]}$.

For the example of the $(5,7)$ curve, when $k=4$, we have
\begin{gather*}
\begin{split}
H_1 &= \begin{pmatrix}
h_4^{(12;4)} & 
h_5^{(12;4)} & 
h_6^{(12;4)} & 
h_7^{(12;4)} & 
h_8^{(12;4)} & 
h_9^{(12;4)} & 
h_{10}^{(12;4)} & 
h_{11}^{(12;4)} \\
h_2^{(12;4)} & 
h_3^{(12;4)} & 
h_4^{(12;4)} & 
h_5^{(12;4)} & 
h_6^{(12;4)} & 
h_7^{(12;4)} & 
h_8^{(12;4)} & 
h_9^{(12;4)} \\ 
h_1^{(12;4)} & 
h_2^{(12;4)} & 
h_3^{(12;4)} & 
h_4^{(12;4)} & 
h_5^{(12;4)} & 
h_6^{(12;4)} & 
h_7^{(12;4)} & 
h_8^{(12;4)} \\ 
 &
h_0^{(12;4)} & 
h_1^{(12;4)} & 
h_2^{(12;4)} & 
h_3^{(12;4)} & 
h_4^{(12;4)} & 
h_5^{(12;4)} & 
h_6^{(12;4)} \\ 
 & & &
h_0^{(12;4)} & 
h_1^{(12;4)} & 
h_2^{(12;4)} & 
h_3^{(12;4)} & 
h_4^{(12;4)} \\ 
 & & & &
h_0^{(12;4)} & 
h_1^{(12;4)} & 
h_2^{(12;4)} & 
h_3^{(12;4)} \\
 & & & & &
h_0^{(12;4)} & 
h_1^{(12;4)} & 
h_2^{(12;4)} \\ 
 & & & & & &
h_0^{(12;4)} & 
h_1^{(12;4)} \end{pmatrix}, \\
H_2 &= \begin{pmatrix}
h_2^{(12;4)} & 
h_3^{(12;4)} & 
h_4^{(12;4)} \\ 
h_1^{(12;4)} & 
h_2^{(12;4)} & 
h_3^{(12;4)} \\ 
 & 
h_0^{(12;4)} & 
h_1^{(12;4)} \end{pmatrix}, \quad 
H_3 = \begin{pmatrix}
h_1^{(12;4)} \end{pmatrix}, \quad 
H_4 = \emptyset.
\end{split} 
\end{gather*} 
For the example of the $(7,9)$ curve, when $g=24$, $k=13$, we have
\begin{eqnarray*}
H_1=\left(
\begin{array}{ccccccccccc}
h_3^{(g;k)} & 
h_4^{(g;k)} & 
h_5^{(g;k)} & 
h_6^{(g;k)} & 
h_7^{(g;k)} & 
h_8^{(g;k)} & 
h_9^{(g;k)} & 
h_{10}^{(g;k)} & 
h_{11}^{(g;k)} & 
h_{12}^{(g;k)} & 
h_{13}^{(g;k)} \\
h_2^{(g;k)} & 
h_3^{(g;k)} & 
h_4^{(g;k)} & 
h_5^{(g;k)} & 
h_6^{(g;k)} & 
h_7^{(g;k)} & 
h_8^{(g;k)} & 
h_9^{(g;k)} & 
h_{10}^{(g;k)} & 
h_{11}^{(g;k)} & 
h_{12}^{(g;k)} \\ 
h_1^{(g;k)} & 
h_2^{(g;k)} & 
h_3^{(g;k)} & 
h_4^{(g;k)} & 
h_5^{(g;k)} & 
h_6^{(g;k)} & 
h_7^{(g;k)} & 
h_8^{(g;k)} & 
h_9^{(g;k)} & 
h_{10}^{(g;k)} & 
h_{11}^{(g;k)} \\ 
h_0^{(g;k)} & 
h_1^{(g;k)} & 
h_2^{(g;k)} & 
h_3^{(g;k)} & 
h_4^{(g;k)} & 
h_5^{(g;k)} & 
h_6^{(g;k)} & 
h_7^{(g;k)} & 
h_8^{(g;k)} & 
h_9^{(g;k)} & 
h_{10}^{(g;k)} \\ 
 &  &
h_0^{(g;k)} & 
h_1^{(g;k)} & 
h_2^{(g;k)} & 
h_3^{(g;k)} & 
h_4^{(g;k)} & 
h_5^{(g;k)} & 
h_6^{(g;k)} & 
h_7^{(g;k)} & 
h_8^{(g;k)} \\
 &  &  &  &
h_0^{(g;k)} & 
h_1^{(g;k)} & 
h_2^{(g;k)} & 
h_3^{(g;k)} & 
h_4^{(g;k)} & 
h_5^{(g;k)} & 
h_6^{(g;k)} \\
 &  &  &  & &
h_0^{(g;k)} & 
h_1^{(g;k)} & 
h_2^{(g;k)} & 
h_3^{(g;k)} & 
h_4^{(g;k)} & 
h_5^{(g;k)}  \\
 &  &  &  & & &
h_0^{(g;k)} & 
h_1^{(g;k)} & 
h_2^{(g;k)} & 
h_3^{(g;k)} & 
h_4^{(g;k)}  \\
 &  &  &  & & & &
h_0^{(g;k)} & 
h_1^{(g;k)} & 
h_2^{(g;k)} & 
h_3^{(g;k)}  \\
 &  &  &  & & & & &
h_0^{(g;k)} & 
h_1^{(g;k)} & 
h_2^{(g;k)}   \\
 &  &  &  & & & & & &
h_0^{(g;k)} & 
h_1^{(g;k)} \end{array}\right),
\end{eqnarray*}
\begin{equation*}
H_2 = \begin{pmatrix}
h_2^{(g;k)} & 
h_3^{(g;k)} & 
h_4^{(g;k)} & 
h_5^{(g;k)} \\ 
h_1^{(g;k)} & 
h_2^{(g;k)} & 
h_3^{(g;k)} & 
h_4^{(g;k)} \\  
h_0^{(g;k)} & 
h_1^{(g;k)} & 
h_2^{(g;k)} & 
h_3^{(g;k)}\\
& & 
h_0^{(g;k)} & 
h_1^{(g;k)} 
\end{pmatrix}, \quad 
H_3 = \begin{pmatrix}
h_1^{(g;k)} & 
h_2^{(g;k)} \\
h_0^{(g;k)} & 
h_1^{(g;k)}  \end{pmatrix}, \quad 
H_4 = \emptyset.
\end{equation*} 
It is clear that the number of $H_i$'s is $n_k$ and 
$H_i$ contains $H_j$ as a submatrix if $i<j$,
with the convention  $H_{n_k+1}=\emptyset$.
Every element $b_{d , e}$  $(d > c+1, e <  c)$ 
of $H_i:=(b_{d,e}) $ vanishes unless it is in $H_{i+1}$.
Every $(c+1,c)$ entry of $H_i$ is equal to $h_0^{(g;k)}=1$ if it is not
in $H_{i+1}$ and the number of $h^{(g;k)}_0$'s in $H_1$ is 
$g-k-n_k$. These facts imply that the term containing
$(h_0^{(g;k)})^{g-k-n_k}$ in the expansion of the determinant
given by products of all the entries with indices permuted, is
given by the determinant of an $(n_k\times n_k)$ matrix, as shown:
$$
|H_1| = \varepsilon^{\prime\prime}(h_0^{(g;k)})^{g-k-n_k}
\left|\begin{matrix}
  h^{(g;k)}_{\Lambda_{k+1}}  
& \cdots & h^{(g;k)}_{\Lambda_{k+1}+g-k-1} \\ 
  \vdots &\cdots &  \vdots \\ 
  h^{(g;k)}_{\Lambda_{k+n_k}-n_k +1} &
\cdots & h^{(g;k)}_{\Lambda_{k+n_k}+g -k - n_k} \\ 
\end{matrix}
\right| +\cdots ,
$$
where $\varepsilon^{\prime\prime}$ is a plus or minus sign.
Noting Corollary \ref{cor:Schur},
let $(a_1, \ldots, a_{n_k}; b_1, \ldots, b_{n_k})$ be 
the partition characteristics of $\Lambda^{[k]}$ \cite[\S 4.1, p. 51]{FH}.
We analyze the subscript $j$ of $h^{(g;k)}_{j}$.
The subscript of the upper-right corner  $H_i$ is given by $a_{n_k - i + 1} 
+ b_{n_k - i + 1} + 1$.
The subscripts of the elements
 on the straight line from the upper-right 
corner to the lower-left corner,
$\Lambda_{k+1}+g-k-1$, $\ldots$, $\Lambda_{k+n_k}-n_k + 1$,
are given by 
$a_{n_k} + b_{n_k} + 1, \ldots, a_3+b_3 +1, a_2+b_2 +1, a_1+b_1 +1$.
The determinant is given by
\begin{gather}
    \varepsilon'\prod_{i=1}^{n_k} h^{(g;k)}_{a_{i} + b_{i} + 1} + \ldots
\label{eq:hab1}
\end{gather}
where $\varepsilon'$ is a plus or minus sign.
We note that each term in the determinant has the same multidegree
in the $t$ variables,
with  each $t_i$ weighing one,
$\sum_{i=1}^{n_k} (a_{i} + b_{i} + 1) = N_k$.
Moreover, $h^{(g;k)}_{a_{i} + b_{i} + 1} =
T^{(g;k)}_{a_{i} + b_{i} + 1} + \cdots$.
 The subscript of the upper-right 
corner of $H_1$ is characterized by 
$\Lambda_k + g - k = a_{n_k} + b_{n_k} + 1$ and is
 the largest among the subscripts of the elements in $H_1$. 
In other words the first term in (\ref{eq:hab1}) is the only term
that contains $T^{(k,g)}_{\Lambda_k + g - k}$ and  cannot
be cancelled. 
Hence we have
$$
\left(\prod_{i=1}^{n_k}\frac{\partial}{ \partial
T^{(g;k)}_{a_{i} + b_{i} + 1}}\right)
\tilde s_{\Lambda^{[k]}}(t_{k+1}, \ldots, t_g) = 
\varepsilon'
\prod_{i=1}^{n_k}{(a_{i} + b_{i} + 1)!}.
$$
Moreover, for every proper subsequence $J$ of $\{1,2,\ldots,n_k\}$, we have
$$
\left(\prod_{i\in J}
\frac{\partial}{ \partial
T^{(g;k)}_{a_{i} + b_{i} + 1}}\right)
\tilde \bs_{\Lambda^{[k]}}(t_{k+1}, \ldots, t_g) 
\Bigr|_{t_j=0: j=k+1,\ldots ,g}
= 0.
$$
Since the element $h_{\Lambda_i + j -i}^{(g;k)}\ 
(k+1\le i\le g, 1\le j \le k$) in 
$H_{k, g-k}(t_{k+1},$ $\ldots,t_g)$ has sufficiently
large degree in $t$,
the contribution from
$
\left(\prod_{i=1}^{n_k}\frac{\partial}{ \partial
T^{(g;k)}_{a_{i} + b_{i} + 1}}\right)
H_{k, g-k}(t_{k+1},\ldots,t_g)  
$
vanishes when we compute (\ref{eq:puH}).
In other words, we have
\begin{gather*}
\begin{split}
&\left(
\prod_{i=1}^{n_k}\frac{\partial}{ \partial
T^{(g;k)}_{a_{i} + b_{i} + 1}}\right)
\bs_{\Lambda^{(g)}}(t_1, \ldots, t_g) 
\Bigr|_{t_j=0: j=k+1,\ldots ,g}\\
&=\varepsilon' \left(\prod_{i=1}^{n_k}{(a_{i} + b_{i} + 1)!}\right)
\bs_{\Lambda^{(k)}}(t_1, \ldots, t_k).
\end{split}
\end{gather*}

However, from Proposition \ref{prop:Nakayashiki},
$\bs_{\Lambda}(t) \equiv \bs_{\Lambda^{(g)}}(t)$ is a function only of
$T_{\Lambda_j + g - j}\equiv T^{(g)}_{\Lambda_j + g - j}$
$=T^{(k)}_{\Lambda_j + g - j}+ T^{(g;k)}_{\Lambda_j + g - j}$
for $j=1, \ldots, g$.
There exists an integer $\ell_i$ such that 
$\Lambda_{\ell_i} + g - \ell_i = a_{i} + b_{i} + 1$ for every $i$.
By naming $I$ the sequence
$(\ell_1,\ell_2,...,\ell_{n_k})$,
we obtain this way the smallest degree in $t$ and the
least number of derivatives, because,
given the configuration of the $H_i$'s, in the determinant
(\ref{eq:hab1}) we have the largest number of $h_0$'s (which
equal 1 hence have degree zero).
Since $\Lambda_{k+1}+g-k-1 = a_{n_k} + b_{n_k} +1$
and 
$a_i + b_i + 1 \le a_{n_k} + b_{n_k} +1$, each element
$j \in I$ belongs to $\{k +1, \ldots, g\}$.

For this $I$,
we obtain 
$$
\left(\prod_{i\in I}
\frac{\partial} {\partial u^{[g;k]}_i} 
\right)
\bs_{\Lambda^{(g)}}(t_1, \ldots, t_g) 
\Bigr|_{t_j=0: j=k+1, \ldots ,g}
= \varepsilon'\left(\prod_{i\in I}{(\Lambda_{i} + g - i)!} \right)
\bs_{\Lambda^{(k)}}(t_1, \ldots, t_k).
$$
Similarly for every proper subsequence $J$ of $I$,
$$
\left(\prod_{i\in J}
\frac{\partial}{ \partial u^{[g;k]}_i}
 \right)
\bs_{\Lambda^{(g)}}(t_1, \ldots, t_g) 
\Bigr|_{t_j=0: j=k+1\ldots g}
= 0.
$$
The proof of 1. is complete.

For statement 2., we use the same sequence $I$.

Noting that $j \in I$ belongs to $\{k +1, \ldots, g\}$,
we have
\begin{equation*}
\begin{pmatrix} 
d u_{1}^{[g]}\\ d u_{2}^{[g]}\\ \vdots\\ d u_{g}^{[g]}\\
\end{pmatrix} 
= M_T^{(g)}
\begin{pmatrix} d t_1 \\ d t_2 \\ \vdots \\ d t_{g} \\ \end{pmatrix},
\quad
\begin{pmatrix} 
d u_{k+1}^{[g;k]}\\ d u_{k+2}^{[g;k]}\\ \vdots\\ d u_{g}^{[g;k]}\\
\end{pmatrix}  
= M_T^{(g;k)}
\begin{pmatrix} d t_{k+1} \\ d t_{k+2} \\ \vdots \\ d t_{g} \\ \end{pmatrix},
\end{equation*}
where
$$
M_T^{(g)} :=
\begin{pmatrix} 
t_1^{\Lambda_{1}+g-2}  & t_2^{\Lambda_{1}+g-2}  &
\cdots& t_g^{\Lambda_{1}+g-2}  \\
t_1^{\Lambda_{2}+g-3}  & t_2^{\Lambda_{2}+g-3}  &
\cdots& t_g^{\Lambda_{2}+g-3}  \\
\vdots & \vdots & \ddots & \vdots \\
t_1^{\Lambda_{g-1}}  & t_2^{\Lambda_{g-1}}  &
\cdots& t_g^{\Lambda_{g-1}}  \\
t_1^{\Lambda_{g}-1}  & t_2^{\Lambda_{g}-1}  &
\cdots& t_g^{\Lambda_{g}-1}  \\
\end{pmatrix}, \quad
M_T^{(g;k)} :=
\begin{pmatrix} 
t_{k+1}^{\Lambda_{k+1}+g-2}  & t_{k+2}^{\Lambda_{k+1}+g-2}  &
\cdots& t_g^{\Lambda_{k+1}+g-2}  \\
t_{k+1}^{\Lambda_{k+1}+g-3}  & t_{k+2}^{\Lambda_{k+1}+g-3}  &
\cdots& t_g^{\Lambda_{2}+g-3}  \\
\vdots & \vdots & \ddots & \vdots \\
t_{k+1}^{\Lambda_{g-1}}  & t_{k+2}^{\Lambda_{g-1}}  &
\cdots& t_g^{\Lambda_{g-1}}  \\
t_{k+1}^{\Lambda_{g}-1}  & t_{k+2}^{\Lambda_{g}-1}  &
\cdots& t_g^{\Lambda_{g}-1}  \\
\end{pmatrix} .
$$
Since the one-forms are given by
\begin{equation*}
 d t_j 
= \sum_{i=1}^g 
\left[ \frac{\partial t_j}{\partial u_{i}^{[g]}} \right]
d u_{i}^{[g]} 
= \sum_{i=1}^g  ((M_T^{(g)})^{-1})_{ji} d u_{i}^{[g]},
\quad
 d t_j 
= \sum_{i=k+1}^g  ((M_T^{(g;k)})^{-1})_{ji} d u_{i}^{[g]},
\end{equation*}
we have
\begin{equation*}
\frac{\partial}{\partial u_{i}^{[g]}}
= \sum_{j=1}^g 
\left[ \frac{\partial t_j}{\partial u_{i}^{[g]}} \right]
\frac{\partial}{\partial t_j}
= \sum_{j=1}^g (M_T^{(g)})^{-1})_{ji} \frac{\partial}{\partial t_j},
 \quad
\frac{\partial}{\partial u_{i}^{[g;k]}}
= \sum_{j=k+1}^g (M_T^{(g;k)})^{-1})_{ji} \frac{\partial}{\partial t_j}. 
\end{equation*}
We note that the $\frac{\partial}{\partial u_i^{[g;k]}}$ span
a $\CC^{g-k}$.
$((M_T^{(g)})^{-1})_{ji}$ is given by
$$
((M_T^{(g)})^{-1})_{ji}=(-1)^{i+j}\frac{
\left|\begin{matrix} 
t_1^{\Lambda_{1}+g-2}  & \cdots & \check t_j^{\Lambda_{1}+g-2}  &
\cdots& t_g^{\Lambda_{1}+g-2}  \\
\vdots & \ddots & \vdots & \ddots & \vdots \\
t_1^{\Lambda_{i-1}+g-i}  & \cdots & \check t_j^{\Lambda_{i-1}+g-i}  &
\cdots& t_g^{\Lambda_{i-1}+g-i}  \\
t_1^{\Lambda_{i+1}+g-i-2}  & \cdots & \check t_j^{\Lambda_{i+1}+g-i-2}  &
\cdots& t_g^{\Lambda_{i+1}+g-i-2}  \\
\vdots & \ddots & \vdots & \ddots & \vdots \\
t_1^{\Lambda_{g}-1}  & \cdots &\check  t_j^{\Lambda_{g}-1}  &
\cdots& t_g^{\Lambda_{g}-1}  \\
\end{matrix}\right|}{
\left|\begin{matrix} 
t_1^{\Lambda_{1}+g-2}  & t_2^{\Lambda_{1}+g-2}  &
\cdots& t_g^{\Lambda_{1}+g-2}  \\
t_1^{\Lambda_{2}+g-3}  & t_2^{\Lambda_{2}+g-3}  &
\cdots& t_g^{\Lambda_{2}+g-3}  \\
\vdots & \vdots & \ddots & \vdots \\
t_1^{\Lambda_{g-1}}  & t_2^{\Lambda_{g-1}}  &
\cdots& t_g^{\Lambda_{g-1}}  \\
t_1^{\Lambda_{g}-1}  & t_2^{\Lambda_{g}-1}  &
\cdots& t_g^{\Lambda_{g}-1}  \\
\end{matrix}\right|}.
$$
We claim that for a symmetric function $h(u^{[g]})$,
and subsequence $J \subset I$,
we have
\begin{equation}
\left(\prod_{i \in J}
\frac{\partial}{\partial u_{i}^{[g]}}
\right) h(u^{[g]}) \Bigr|_{t_{k+1} =0, \ldots, t_g =0}=
\left(\prod_{i \in J}
\frac{\partial}{\partial u_{i}^{[g;k]}}
\right) h(u^{[g]}) \Bigr|_{t_{k+1} =0, \ldots, t_g =0}.
\label{eq:pug=pugk}
\end{equation}
This is essentially the same as Lemma \ref{lemma:d_uinfty} but
we prove (\ref{eq:pug=pugk}) directly as follows:
Let $t_{k+1}, \ldots, t_g$ have the same order $\epsilon$, 
written as $\epsilon_i
=t_i$, $i=k+1,...,g$.

We use the property of the chosen sequence $I$,
namely that every $i \in I$ satisfies $k<i\le g$.
Then  for $1\le j\le k$,
let 
$\Xi_{k,g,i}:= \sum_{\ell=k,\ell\neq i}^g (\Lambda_{\ell} + g-\ell-1) $
and
$\Xi_{k+1,g}:= \sum_{\ell=k+1}^g (\Lambda_{\ell} + g-\ell-1)$.
For $1\le j\le k$,  noting that 
$\Xi_{k,g,i}>\Xi_{k+1,g}$, we have
$$
((M_T^{(g)})^{-1})_{ji}=(-1)^{i+j}\frac{
\left|\begin{matrix} 
t_1^{\Lambda_{1}+g-2}  & \cdots & \check t_j^{\Lambda_{1}+g-2}  &
\cdots& t_k^{\Lambda_{1}+g-2}  \\
\vdots & \ddots & \vdots & \ddots & \vdots \\
t_1^{\Lambda_{k-1}+g-k}  & \cdots & \check t_j^{\Lambda_{k-1}+g-k}  &
\cdots& t_g^{\Lambda_{k-1}+g-k}  \\
\end{matrix}\right| 
d_\ge(\epsilon^{\Xi_{k,i,  g}})
+d_>(\epsilon^{\Xi_{k,i,  g}})
}{
\left|\begin{matrix} 
t_1^{\Lambda_{1}+g-2} & 
\cdots & t_k^{\Lambda_{1}+g-2}  \\
 \vdots & \ddots & \vdots \\
t_1^{\Lambda_{k}+g - k -1}  & \cdots& t_k^{\Lambda_{k}+g - k -1}  \\
\end{matrix}\right| 
d_\ge(\epsilon^{\Xi_{k+1,  g}})
+d_>(\epsilon^{\Xi_{k+1,  g}})
}
$$
which vanishes when $\epsilon$ vanishes.
Here $d_{>}(z^\ell)\in \{\sum_{|\alpha|>\ell} a_\alpha z^\alpha\}$
$d_{\ge}(z^\ell)\in \{\sum_{|\alpha|\ge\ell} a_\alpha z^\alpha\}$ for
$\alpha=(\alpha_1,\ldots,\alpha_m)$,
$z^\alpha=z_1^{\alpha_1}\cdots z_g^{\alpha_m}$, and
$|\alpha|=\alpha_1+\cdots+\alpha_m$.
Then  for $k+1\le j\le g$, 
let 
$\Xi_{k+1,g,i}:= \sum_{\ell=k+1,\ell\neq i}^g (\Lambda_{\ell} + g-\ell-1) $.
For $k+1\le j\le g$, noting
$\Xi_{k+1,g,i}<\Xi_{k+1,g}$, we have
$$
((M_T^{(g)})^{-1})_{ji}=(-1)^{i+j}\frac{
\left|\begin{matrix} t_1^{\Lambda_{1}+g-2} & \cdots& t_k^{\Lambda_{1}+g-2}  \\
 \vdots & \ddots & \vdots \\
t_1^{\Lambda_{k}+g - k -1}  & \cdots& t_k^{\Lambda_{k}+g - k -1}  \\
\end{matrix}\right| 
d_\ge(\epsilon^{\Xi_{k+1, i, g}})
+d_>(\epsilon^{\Xi_{k+1, i, g}})
}{
\left|\begin{matrix} t_1^{\Lambda_{1}+g-2}&  \cdots& t_k^{\Lambda_{1}+g-2}  \\
 \vdots & \ddots & \vdots \\
t_1^{\Lambda_{k}+g - k -1}  & \cdots& t_k^{\Lambda_{k}+g - k -1}  \\
\end{matrix}\right| 
d_\ge(\epsilon^{\Xi_{k+1,  g}})
+d_>(\epsilon^{\Xi_{k+1,  g}})
}
$$
which is singular for small $\epsilon$.  Its leading term becomes
$((M_T^{(g;k)})^{-1})_{ji}$.
This means that we have (\ref{eq:pug=pugk}) and  we proved the
second statement.

In our last example, to give a visual description of the general
pattern, rather than the (5,7) or (7,9) cases given above which
would occupy several pages, we treat the trigonal (3,7) case for $k=2$:
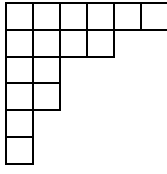
\begin{gather*}
\begin{split}
 \yng(6,4,2,2,1,1), \quad
H_1 &= \begin{pmatrix}
h_2^{(6,2)} & 
h_3^{(6,2)} & 
h_4^{(6,2)} & 
h_5^{(6,2)} \\ 
h_1^{(6,2)} & 
h_2^{(6,2)} & 
h_3^{(6,2)} & 
h_4^{(6,2)} \\
&
h_0^{(6,2)} & 
h_1^{(6,2)} & 
h_2^{(6,2)} \\
& & 
h_0^{(6,2)} & 
h_1^{(6,2)} \\ 
\end{pmatrix}, \quad 
H_2 = \begin{pmatrix}
h_1^{(6,2)}  
\end{pmatrix}, \quad 
H_3 = \emptyset.
\end{split}
\end{gather*}
Then the transition matrices expand as follows:
$$
M_T^{(6)} =
\begin{pmatrix} 
t_{1}^{10}  & t_{2}^{10}  &
\eps_{3}^{10}  & \eps_{4}^{10}  &
\eps_{5}^{10}  & \eps_{6}^{10}  \\
t_{1}^{7}  & t_{2}^{7}  &
\eps_{3}^{7}  & \eps_{4}^{7}  &
\eps_{5}^{7}  & \eps_{6}^{7}  \\
t_{1}^{4}  & t_{2}^{4}  &
\eps_{3}^{4}  & \eps_{4}^{4}  &
\eps_{5}^{4}  & \eps_{6}^{4}  \\
t_{1}^{3}  & t_{2}^{3}  &
\eps_{3}^{3}  & \eps_{4}^{3}  &
\eps_{5}^{3}  & \eps_{6}^{3}  \\
t_{1}^{}  & t_{2}^{}  &
\eps_{3}^{}  & \eps_{4}^{}  &
\eps_{5}^{}  & \eps_{6}^{}  \\
1 & 1 & 1 & 1 & 1 & 1
\end{pmatrix} 
$$
{\small{
\begin{equation*}
\begin{split}
(M_T^{(6)})^{-1})_{1,6} &=
-\frac{
\left|\begin{matrix} 
 t_{2}^{10}  
\end{matrix} \right| 
d_\ge(\eps^{7+4+3+1}) + d_>(\eps^{15})} {
\left|\begin{matrix} 
t_{1}^{10}  & t_{2}^{10}  \\
t_{1}^{7}  & t_{2}^{7}  
\end{matrix} \right|
d_\ge(\eps^{4+3+1+0}) + d_>(\eps^{8})} 
, \quad
(M_T^{(6)})^{-1})_{2,6} =
\frac{
\left|\begin{matrix} 
 t_{1}^{10}  
\end{matrix} \right| 
d_\ge(\eps^{7+4+3+1}) + d_>(\eps^{15})} {
\left|\begin{matrix} 
t_{1}^{10}  & t_{2}^{10}  \\
t_{1}^{7}  & t_{2}^{7}  
\end{matrix} \right|
d_\ge(\eps^{4+3+1+0}) + d_>(\eps^{8})} 
, \\
(M_T^{(6)})^{-1})_{3,6}& =
-\frac{
\left|\begin{matrix} 
t_{1}^{10}  & t_{2}^{10}  \\
t_{1}^{7}  & t_{2}^{7}  
\end{matrix} \right|
\left|\begin{matrix} 
\eps_{4}^{4}  &\eps_{5}^{4}  & \eps_{6}^{4}  \\
\eps_{4}^{3}  &\eps_{5}^{3}  & \eps_{6}^{3}  \\
\eps_{4}^{1}  &\eps_{5}^{1}  & \eps_{6}^{1}  \\
\end{matrix} \right| + d_>(\eps^{8})} 
{\left|\begin{matrix} 
t_{1}^{10}  & t_{2}^{10}  \\
t_{1}^{7}  & t_{2}^{7}  
\end{matrix} \right|
\left|\begin{matrix} 
\eps_3^{4} & \eps_{4}^{4}  &\eps_{5}^{4}  & \eps_{6}^{4}  \\
\eps_3^{3} & \eps_{4}^{3}  &\eps_{5}^{3}  & \eps_{6}^{3}  \\
\eps_3^{1} & \eps_{4}^{1}  &\eps_{5}^{1}  & \eps_{6}^{1}  \\
1 & 1  & 1 & 1 
\end{matrix} \right| + d_>(\eps^{8})}, \quad
(M_T^{(6)})^{-1})_{4,6} =
\frac{
\left|\begin{matrix} 
t_{1}^{10}  & t_{2}^{10}  \\
t_{1}^{7}  & t_{2}^{7}  
\end{matrix} \right|
\left|\begin{matrix} 
\eps_{3}^{4}  &\eps_{5}^{4}  & \eps_{6}^{4}  \\
\eps_{3}^{3}  &\eps_{5}^{3}  & \eps_{6}^{3}  \\
\eps_{3}^{1}  &\eps_{5}^{1}  & \eps_{6}^{1}  \\
\end{matrix} \right| + d_>(\eps^{8})} 
{\left|\begin{matrix} 
t_{1}^{10}  & t_{2}^{10}  \\
t_{1}^{7}  & t_{2}^{7}  
\end{matrix} \right|
\left|\begin{matrix} 
\eps_3^{4} & \eps_{4}^{4}  &\eps_{5}^{4}  & \eps_{6}^{4}  \\
\eps_3^{3} & \eps_{4}^{3}  &\eps_{5}^{3}  & \eps_{6}^{3}  \\
\eps_3^{1} & \eps_{4}^{1}  &\eps_{5}^{1}  & \eps_{6}^{1}  \\
1 & 1  & 1 & 1 
\end{matrix} \right| + d_>(\eps^{8})}, \\
&\quad
 \cdots.
\end{split}
\end{equation*}

\begin{equation*}
\begin{split}
(M_T^{(6)})^{-1})_{1,3} &=
\frac{
\left|\begin{matrix} 
 t_{2}^{10}  
\end{matrix} \right| 
d_\ge(\eps^{7+3+1+0}) + d_>(\eps^{11})} {
\left|\begin{matrix} 
t_{1}^{10}  & t_{2}^{10}  \\
t_{1}^{7}  & t_{2}^{7}  
\end{matrix} \right|
d_\ge(\eps^{4+3+1+0}) + d_>(\eps^{8})} 
, \quad
(M_T^{(6)})^{-1})_{2,3} =
-\frac{
\left|\begin{matrix} 
 t_{1}^{10}  
\end{matrix} \right| 
d_\ge(\eps^{7+3+1+0}) + d_>(\eps^{11})} {
\left|\begin{matrix} 
t_{1}^{10}  & t_{2}^{10}  \\
t_{1}^{7}  & t_{2}^{7}  
\end{matrix} \right|
d_\ge(\eps^{4+3+1+0}) + d_>(\eps^{8})} 
, \\
(M_T^{(6)})^{-1})_{3,3}& =
\frac{
\left|\begin{matrix} 
t_{1}^{10}  & t_{2}^{10}  \\
t_{1}^{7}  & t_{2}^{7}  
\end{matrix} \right|
\left|\begin{matrix} 
\eps_{4}^{3}  &\eps_{5}^{3}  & \eps_{6}^{3}  \\
\eps_{4}^{1}  &\eps_{5}^{1}  & \eps_{6}^{1}  \\
1 & 1 & 1 
\end{matrix} \right| + d_>(\eps^{4})} 
{\left|\begin{matrix} 
t_{1}^{10}  & t_{2}^{10}  \\
t_{1}^{7}  & t_{2}^{7}  
\end{matrix} \right|
\left|\begin{matrix} 
\eps_3^{4} & \eps_{4}^{4}  &\eps_{5}^{4}  & \eps_{6}^{4}  \\
\eps_3^{3} & \eps_{4}^{3}  &\eps_{5}^{3}  & \eps_{6}^{3}  \\
\eps_3^{1} & \eps_{4}^{1}  &\eps_{5}^{1}  & \eps_{6}^{1}  \\
1 & 1  & 1 & 1 
\end{matrix} \right| + d_>(\eps^{8})}, \quad
(M_T^{(6)})^{-1})_{4,3} =
-\frac{
\left|\begin{matrix} 
t_{1}^{10}  & t_{2}^{10}  \\
t_{1}^{7}  & t_{2}^{7}  
\end{matrix} \right|
\left|\begin{matrix} 
\eps_{3}^{3}  &\eps_{5}^{3}  & \eps_{6}^{3}  \\
\eps_{3}^{1}  &\eps_{5}^{1}  & \eps_{6}^{1}  \\
1 & 1 & 1 
\end{matrix} \right| + d_>(\eps^{4})} 
{\left|\begin{matrix} 
t_{1}^{10}  & t_{2}^{10}  \\
t_{1}^{7}  & t_{2}^{7}  
\end{matrix} \right|
\left|\begin{matrix} 
\eps_3^{4} & \eps_{4}^{4}  &\eps_{5}^{4}  & \eps_{6}^{4}  \\
\eps_3^{3} & \eps_{4}^{3}  &\eps_{5}^{3}  & \eps_{6}^{3}  \\
\eps_3^{1} & \eps_{4}^{1}  &\eps_{5}^{1}  & \eps_{6}^{1}  \\
1 & 1  & 1 & 1 
\end{matrix} \right| + d_>(\eps^{8})},\\
& \quad
 \cdots.
\end{split}
\end{equation*}
}}
\end{proof}

The above proof yields the following Proposition:

\begin{proposition}\label{prop:L^k}
For the Young diagram $\Lambda$ associated with the 
$C_{r,s}$ curve $X$ of genus $g$, an integer $k$ $(0\le k < g)$,
and the characteristics of the partition of $\Lambda^{[k]}$,
$$
(a_1, a_2, \ldots, a_{n_k}; b_1, b_2, \ldots, b_{n_k}),
$$
the following holds:
\begin{enumerate}
\item There exists an integer $\ell_i$ such that
$$
	\Lambda_{\ell_i} + g - \ell_i = a_i + b_i + 1
$$
for every $i = 0, 1, \ldots, n_k$;

\item When the correspondence is denoted by 
$$
L^{[k]}(a_i, b_i) := \ell_i,
$$
an example of $I$ appearing in Lemma 
 \ref{lemma:schurLambda} 2. is given by
$$
      I =\{ L^{[k]}(a_1, b_1),
L^{[k]}(a_2, b_2),
\ldots, L^{[k]}(a_{n_k}, b_{n_k})\} ;
$$

\item $L^{[k]}(a_{n_k}, b_{n_k}) = k + 1$, and 

\item When the $C_{rs}$ curve is hyperelliptic
of genus $g$,  {\it{ i.e.}},
$(r, s) = (2, 2g + 1)$, the set of indices $I$ is equal
to
$$
\left\{
\begin{matrix} 
\{g, g - 2, \ldots, k + 2, k\} & \mbox{if } g-k \mbox{ is even}, \\
\{g - 1, g - 3, \ldots, k +3 , k +1 \} & \mbox{ otherwise}.
\end{matrix} \right.
$$ 
\end{enumerate}
\end{proposition}

\begin{proof} The proof of Lemma \ref{lemma:schurLambda} gives
1. and 2.; 3. is proved using the definition in 2.
and the equality $a_{n_k} + b_{n_k} + 1 =
\Lambda_{k + 1} + g - k - 1$. 4. is obtained by straightforward computation.
\end{proof}

Note that since $I$ in Proposition \ref{prop:L^k} 4. corresponds  to
$\natural_k$ in Theorem \ref{Onishi}, such 
$I$'s are shown  in Table 1.1.

\medskip
For use as in Lemma \ref{lemma:schurLambda}, we define a family of 
sequences which we name $\Index$.
\begin{definition} \label{def:sigmaIk}
Let $\Index$ be the family of all finite sequences made up
with  numbers between 1 and $g$ (some numbers may be repeated),
though changing the order of the elements in a sequence
would not change the values defined herewith for a given
element of $\Index$.
For an element $I_k$ of $\Index$ and $u \in \mathbb{C}^g$, define\,{\rm{:}}
$$
	\sigma_{I_k}:=\left( 
       \prod_{i \in I_k} 
         \frac{\partial}{\partial u_i} \right) \sigma,
$$
$$
        \mwdeg(I_k) := \sum_{i \in I_k} \mwdeg(u_i).
$$
\end{definition}

In view of  Proposition \ref{prop:L^k}, we construct a set of indices as
a natural extension of those in \cite{O1, O2, MO}.
\begin{definition} \label{def:naturalk}
For $k = 1, 2, \ldots, g-1$,
and the characteristics of the partition of
$\Lambda^{[k]}$,  $(a_1, \ldots, a_r; b_1, \ldots, b_r)$,
we define
$$
\natural_k :=\{
L^{[k]}(a_1,b_1), 
L^{[k]}(a_2,b_2), 
 \ldots,
L^{[k]}(a_{n_k}, b_{n_k})\},
$$ 
and 
$$
\natural_k^{(i)} := \left(\natural_k \setminus \{k + 1\}\right) \bigcup \{i\},
\quad \mbox{for} \quad i = 1, 2, \ldots, k.
$$
Further,
$\natural_g :=\emptyset$ and
$\natural_g^{(i)} := i$ for $i= 1, 2, \ldots, g$.
\end{definition}

We continue the examples of Tables 2.1, 2.2
 (with $(n_k, m_k)$ 
corresponding to $k$ in Corollary \ref{cor:6.7} 1.)
in
Table 5.1  for the case  $(r, s) = (5, 7)$ 
and in Table 5.2 for the case  $(r, s) = (7, 9)$.

{\tiny{
\begin{gather*}
\centerline{
\vbox{
	\baselineskip =10pt
	\tabskip = 1em
	\halign{&\hfil#\hfil \cr
        \multispan7 \hfil Table 5.1 a \hfil \cr
	\noalign{\smallskip}
	\noalign{\hrule height0.8pt}
	\noalign{\smallskip}
$i$ & \strut\vrule& 0 &1 & 2 & 3 & 4 & 5 & 6 & 7 & 8 & 9 & 10 & 11 & 12 \cr
\noalign{\smallskip}
\noalign{\hrule height0.3pt}
\noalign{\smallskip}
$\phi(i)$ & \strut\vrule 
&1& $x$& $y$ & $x^2$&$xy $&$y^2$ & $x^3$ & $x^2y$ 
& $xy^2$ & $x^4$ & $y^3$ & $x^3y$ & $x^2y^2$ \cr 
$N(i)$ &
 \strut\vrule & 
 0&  5 & 7 & 10 & 12 & 14 & 15 & 17 & 19 & 20 & 21 & 22 & 24  \cr 
$\Lambda_i$ &
 \strut\vrule & 
 - & 12 & 8 & 7 & 5 & 4 & 3 & 3 & 2 & 1 & 1 & 1 & 1  \cr 
$\Lambda_i+g - i $ &
 \strut\vrule & 
 - & 23 & 18 & 16 & 13 & 11 & 9 & 8 & 6 & 4 & 3 & 2 & 1  \cr 
$n_i$ &
 \strut\vrule & 
  4 & 4 & 3 & 3 & 3 & 2 & 2 & 1 & 1 & 1 & 1 & 1& -  \cr 
$N_i$ &
 \strut\vrule & 
  48 & 36 & 28 & 21 & 16 & 12 & 9 & 6 & 4 & 3 & 2 & 1 &- \cr 
\noalign{\smallskip}
\noalign{\hrule height0.3pt}
\noalign{\smallskip}
	\noalign{\hrule height0.8pt}
}
}
}
\end{gather*}
}}

{\tiny{
\begin{gather*}
\centerline{
\vbox{
	\baselineskip =10pt
	\tabskip = 1em
	\halign{&\hfil#\hfil \cr
        \multispan7 \hfil Table 5.1 b \hfil \cr
	\noalign{\smallskip}
	\noalign{\hrule height0.8pt}
	\noalign{\smallskip}
$k$ & \strut\vrule& $(a_1,...,a_{n_k};b_0,...,b_{n_k})$ & \strut\vrule & 
$(a_i+b_i+1)_{1\le i\le n_k}$ &
\strut\vrule & $\sum(a_i+b_i+1)$ & \strut\vrule & $
\natural_k$ \cr
\noalign{\smallskip}
\noalign{\hrule height0.3pt}
\noalign{\smallskip}
$0$ & \strut\vrule& 
      $(1, 4, 6, 11; 1, 4, 6, 11)$ & \strut\vrule & 
      $(3, 9, 13, 23)$ & \strut\vrule & 
      $48$ & \strut\vrule & $(10, 6, 4, 1)$ \cr
$1$ & \strut\vrule& 
      $(0, 3, 5, 10; 0, 2, 5, 7)$ & \strut\vrule & 
      $(1, 6, 11, 18)$ & \strut\vrule & 
      $36$ & \strut\vrule & $(12, 8, 5, 2)$ \cr
$2$ & \strut\vrule& 
      $(2, 4, 9; 1, 3, 6)$ & \strut\vrule & 
      $(4, 8, 16)$ & \strut\vrule & 
      $28$ & \strut\vrule & $(9, 7, 3)$ \cr
$3$ & \strut\vrule& 
      $(1, 3, 8; 0, 2, 4)$ & \strut\vrule & 
      $(2, 6, 13)$ & \strut\vrule & 
      $21$ & \strut\vrule & $(11, 8, 4)$ \cr
$4$ & \strut\vrule& 
      $(0, 2, 7; 0, 1, 3)$ & \strut\vrule & 
      $(1, 4, 11)$ & \strut\vrule & 
      $16$ & \strut\vrule & $(12, 9, 5)$ \cr
$5$ & \strut\vrule& 
      $(1, 6; 1, 2)$ & \strut\vrule & 
      $(3, 9)$ & \strut\vrule & 
      $12$ & \strut\vrule & $(10, 6)$ \cr
$6$ & \strut\vrule& 
      $(0, 5; 0, 2)$ & \strut\vrule & 
      $(1, 8)$ & \strut\vrule & 
      $9$ & \strut\vrule & $(12, 7)$ \cr
$7$ & \strut\vrule& 
      $(4; 1)$ & \strut\vrule & 
      $(6)$ & \strut\vrule & 
      $6$ & \strut\vrule & $(8)$ \cr
$8$ & \strut\vrule& 
      $(3; 0)$ & \strut\vrule & 
      $(4)$ & \strut\vrule & 
      $4$ & \strut\vrule & $(9)$ \cr
$9$ & \strut\vrule& 
      $(2; 0)$ & \strut\vrule & 
      $(3)$ & \strut\vrule & 
      $3$ & \strut\vrule & $(10)$ \cr
$10$ & \strut\vrule& 
      $(1; 0)$ & \strut\vrule & 
      $(2)$ & \strut\vrule & 
      $2$ & \strut\vrule & $(11)$ \cr
$11$ & \strut\vrule& 
      $(0; 0)$ & \strut\vrule & 
      $(1)$ & \strut\vrule & 
      $1$ & \strut\vrule & $(12)$ \cr
\noalign{\smallskip}
\noalign{\hrule height0.3pt}
\noalign{\smallskip}
\noalign{\hrule height0.8pt}
}
}
}
\end{gather*}
}}

{\tiny{
\begin{gather*}
\centerline{
\vbox{
	\baselineskip =10pt
	\tabskip = 1em
	\halign{&\hfil#\hfil \cr
        \multispan7 \hfil Table 5.2 a\hfil \cr
	\noalign{\smallskip}
	\noalign{\hrule height0.8pt}
	\noalign{\smallskip}
$i$ & \strut\vrule& 
0 &1 & 2 & 3 & 4 & 5 & 6 & 7 & 8 & 9 & 10 & 11 & 12 \cr
\noalign{\smallskip}
\noalign{\hrule height0.3pt}
\noalign{\smallskip}
$\phi(i)$ & \strut\vrule 
&1& $x$& $y$ & $x^2$&$xy$&$y^2$ & $x^3$ & $x^2y$ 
&$xy^2$&$y^3$ & $x^4$ & $x^3 y$ & $x^2y^2$ \cr
$N(i)$ &
 \strut\vrule & 
 0&  7 & 9 & 14 & 16 & 18 & 21 & 23 & 25 & 27 & 28 & 30  & 32 \cr
$\Lambda_i$ &
 \strut\vrule & 
-  & 24 & 18 & 17 & 13 & 12 & 11 & 9 & 8 & 7 & 6 & 6 & 5 \cr
$\Lambda_i+g - i $ &
 \strut\vrule & 
 - & 47 & 40 & 38 & 33 & 31 & 29 & 26 & 24 & 22 & 20 & 19 & 18  \cr 
$n_i$ &
 \strut\vrule  
   & 8 & 7 & 7 & 6 & 6 & 6 & 5 & 5 & 4 & 4 & 3& 3 &3 \cr 
$N_i$ &
 \strut\vrule & 
160 & 136 & 118 & 101 & 88 & 76 & 65 & 56 & 48 & 41 & 35 & 29 & 24 & \cr 
\noalign{\smallskip}
\noalign{\hrule height0.3pt}
\noalign{\smallskip}
	\noalign{\hrule height0.8pt}
$i$ & \strut\vrule& 
 13 & 14 & 15 & 16 & 17 & 18 & 19 & 20 & 21 & 22 & 23 & 24 \cr
\noalign{\smallskip}
\noalign{\hrule height0.3pt}
\noalign{\smallskip}
$\phi(i)$ & \strut\vrule 
& $xy^3$
&$x^5$ &$y^4$ & $x^4y$ & $x^3 y^2$ & $x^2y^3$
&$x^6$ & $xy^4$ & $x^5y$ & $y^5$ & $x^4y^2$ & $x^2y^4$ \cr
$N(i)$ &
 \strut\vrule & 
  34 & 35 & 36 & 37 & 39 & 41 & 42 & 43 & 44 & 45 & 46 & 48  \cr 
$\Lambda_i$ &
 \strut\vrule  
  & 4 & 3 & 3 & 3 & 3 & 2 & 1 & 1 & 1 & 1 & 1 & 1   \cr 
$\Lambda_i+g - i $ &
 \strut\vrule & 
 15 & 13 & 12 & 11 & 10 & 8 & 6 & 5 & 4 & 3 & 2 & 1    \cr 
$n_i$ &
 \strut\vrule & 
   3 & 3 & 2 & 2 & 1 & 1 & 1 & 1 & 1 & 1 & 1& -  \cr 
$N_i$ &
 \strut\vrule 
  & 20 & 17 & 14 & 11 & 8 & 6 & 5 & 4 & 3 & 2 & 1 &- \cr 
\noalign{\smallskip}
\noalign{\hrule height0.3pt}
\noalign{\smallskip}
	\noalign{\hrule height0.8pt}
}
}
}
\end{gather*}
}}

{\tiny{
\begin{gather*}
\centerline{
\vbox{
	\baselineskip =10pt
	\tabskip = 1em
	\halign{&\hfil#\hfil \cr
        \multispan7 \hfil Table 5.2 b \hfil \cr
	\noalign{\smallskip}
	\noalign{\hrule height0.8pt}
	\noalign{\smallskip}
$k$ & \strut\vrule& $(a_1,...,a_{n_k};b_0,...,b_{n_k})$ & \strut\vrule & 
$(a_i+b_i+1)_{1\le i\le n_k}$ &
\strut\vrule & $\sum(a_i+b_i+1)$ & \strut\vrule & $
\natural_k$ \cr
\noalign{\smallskip}
\noalign{\hrule height0.3pt}
\noalign{\smallskip}
$0$ & \strut\vrule& 
      $(0, 2, 5, 7, 9, 14, 16, 23; 0, 2, 5, 7, 9, 14, 16, 23)$& \strut\vrule
&   
      $(1, 5, 11, 15, 19, 29, 33, 47)$ & \strut\vrule & 
      $160$ & \strut\vrule & $(24, 20, 16, 12, 11, 6, 4, 1)$ \cr
$1$ & \strut\vrule& 
      $(1, 4, 6, 8, 13, 15, 22; 1, 3, 6, 8, 10, 15, 17)$& \strut\vrule & 
      $(3, 8, 13, 17, 24, 31, 40)$ & \strut\vrule & 
      $136$ & \strut\vrule & $(22, 18, 14, 12, 8, 5, 2)$ \cr
$2$ & \strut\vrule& 
      $(0, 3, 5, 7, 12, 14, 21; 0, 2, 4, 7, 9, 11, 16)$& \strut\vrule & 
      $(1, 6, 10, 15, 22, 26, 38)$ & \strut\vrule & 
      $118$ & \strut\vrule & $(24, 19, 17, 13, 9, 7, 3)$ \cr
$3$ & \strut\vrule& 
      $(2, 4, 6, 11, 13, 20; 1, 3, 5, 8, 10, 12)$& \strut\vrule & 
      $(4, 8, 12, 20, 24, 33)$ & \strut\vrule & 
      $101$ & \strut\vrule & $(21, 18, 15, 10, 8, 4)$ \cr
$4$ & \strut\vrule& 
      $(1, 3, 5, 10, 12, 19; 0, 2, 4, 6, 9, 11)$& \strut\vrule & 
      $(2, 6, 10, 17, 22, 31)$ & \strut\vrule & 
      $88$ & \strut\vrule & $(23, 19, 17, 12, 9, 5)$ \cr
$5$ & \strut\vrule& 
      $(0, 2, 4, 9, 11, 18; 0, 1, 3, 5, 7, 10)$& \strut\vrule & 
      $(1, 4, 8, 15, 19, 29)$ & \strut\vrule & 
      $76$ & \strut\vrule & $(24, 21, 18, 13, 11, 6)$ \cr
$6$ & \strut\vrule& 
      $(1, 3, 8, 10, 17; 1, 2, 4, 6, 8)$& \strut\vrule & 
      $(2, 6, 14, 17, 26)$ & \strut\vrule & 
      $65$ & \strut\vrule & $(22, 19, 14, 12, 7)$ \cr
$7$ & \strut\vrule& 
      $(0, 2, 7, 9, 16; 0, 2, 3, 5, 7)$& \strut\vrule & 
      $(1, 5, 11, 15, 24)$ & \strut\vrule & 
      $56$ & \strut\vrule & $(24, 20, 16, 15, 8)$ \cr
$8$ & \strut\vrule& 
      $(1, 6, 8, 15; 1, 3, 4, 6)$& \strut\vrule & 
      $(3, 10, 13, 22)$ & \strut\vrule & 
      $48$ & \strut\vrule & $(22, 17, 14, 9)$ \cr
$9$ & \strut\vrule& 
      $(0, 5, 7, 14; 0, 2, 4, 5)$& \strut\vrule & 
      $(1, 8, 12, 20)$ & \strut\vrule & 
      $41$ & \strut\vrule & $(24, 18, 15, 10)$ \cr
$10$ & \strut\vrule& 
      $(4, 6, 13; 1, 3, 5)$& \strut\vrule & 
      $(6, 10, 19)$ & \strut\vrule & 
      $35$ & \strut\vrule & $(19, 17, 11)$ \cr
$11$ & \strut\vrule& 
      $(3, 5, 12; 0, 2, 4)$& \strut\vrule & 
      $(17, 8, 4)$ & \strut\vrule & 
      $29$ & \strut\vrule & $(21, 18, 12)$ \cr
$12$ & \strut\vrule& 
      $(2, 4, 11; 0, 1, 3)$& \strut\vrule & 
      $(3, 6, 15)$ & \strut\vrule & 
      $24$ & \strut\vrule & $(22, 19, 13)$ \cr
$13$ & \strut\vrule& 
      $(1, 3, 10; 0, 1, 2)$& \strut\vrule & 
      $(2, 5, 13)$ & \strut\vrule & 
      $20$ & \strut\vrule & $(23, 20, 14)$ \cr
$14$ & \strut\vrule& 
      $(0, 2, 9; 0, 1, 2)$& \strut\vrule & 
      $(1, 4, 12)$ & \strut\vrule & 
      $17$ & \strut\vrule & $(24, 21, 15)$ \cr
$15$ & \strut\vrule& 
      $(1, 8; 1, 2)$& \strut\vrule & 
      $(3, 11)$ & \strut\vrule & 
      $14$ & \strut\vrule & $(22, 16)$ \cr
$16$ & \strut\vrule& 
      $(0, 7; 0, 2)$& \strut\vrule & 
      $(1, 10)$ & \strut\vrule & 
      $11$ & \strut\vrule & $(24, 17)$ \cr
$17$ & \strut\vrule& 
      $(6; 1)$ & \strut\vrule & 
      $(8)$ & \strut\vrule & 
      $8$ & \strut\vrule & $(18)$ \cr
$18$ & \strut\vrule& 
      $(5; 0)$ & \strut\vrule & 
      $(6)$ & \strut\vrule & 
      $6$ & \strut\vrule & $(19)$ \cr
$19$ & \strut\vrule& 
      $(4; 0)$ & \strut\vrule & 
      $(5)$ & \strut\vrule & 
      $5$ & \strut\vrule & $(20)$ \cr
$20$ & \strut\vrule& 
      $(3; 0)$ & \strut\vrule & 
      $(4)$ & \strut\vrule & 
      $4$ & \strut\vrule & $(21)$ \cr
$21$ & \strut\vrule& 
      $(2; 0)$ & \strut\vrule & 
      $(3)$ & \strut\vrule & 
      $3$ & \strut\vrule & $(22)$ \cr
$22$ & \strut\vrule& 
      $(1; 0)$ & \strut\vrule & 
      $(2)$ & \strut\vrule & 
      $2$ & \strut\vrule & $(23)$ \cr
$23$ & \strut\vrule& 
      $(0; 0)$ & \strut\vrule & 
      $(1)$ & \strut\vrule & 
      $1$ & \strut\vrule & $(24)$ \cr
\noalign{\smallskip}
\noalign{\hrule height0.3pt}
\noalign{\smallskip}
\noalign{\hrule height0.8pt}
}
}
}
\end{gather*}
}}

We can now state the main theorem (cf. Theorem \ref{Onishi} and Table 1.1):

\begin{theorem} \label{vanishingTh}
Let $\cI_g= \{ \emptyset \}$.
For each $k=1,2,\ldots,g$,
there exists a subfamily of $\Index$, 
$\cI_k$, of cardinality $n_k$, whose element $I_{k}$ is such that
$\mwdeg(I_{k})\ge N_k$,
and as a function over 
$\kappa^{-1}(\Theta^k \setminus (\Theta^k_1 \cup \Theta^{k-1}))$,
\begin{equation}
	\sigma_{J_{k}}=\left\{
        \begin{matrix} 
           \neq 0 & \mbox{ for } J_{k} =I_k\\ 
           = 0 & \mbox{ for } J_{k}\varsubsetneqq I_k. \\
        \end{matrix} \right.
\label{eq:Avn-nvn}
\end{equation}
Moreover, 
$\{ \natural_k, \natural_k^{(k)},
 \natural_k^{(k-1)}, \ldots, \natural_k^{(2)}, \natural_k^{(1)}\}\subset
\cI_k$.
\end{theorem}

\begin{remark} 
{\rm{
The property
$\sigma_{\natural_{k}} \neq 0$ over 
$\kappa^{-1}(\Theta^k \setminus (\Theta^k_1 \cup \Theta^{k-1}))$
is the generalization of \^Onishi's results in \cite{O1, O2, MO}.
Further we note that there exists $J_k \in \Index$ such that
$\# J_k = n_k$ but 
$\sigma_{J_k}(u) = 0$ for $u \in 
\kappa^{-1}(\Theta^k \setminus (\Theta^k_1 \cup \Theta^{k-1}))$.
}}
\end{remark} 

\begin{remark} 
{\rm{
Using Proposition \ref{prop:pperiod}, we have the following
corollary,
which shows $\sigma_{J_{k,i}}$ is also a normalized theta function
over $\kappa^{-1}(\Theta^k \setminus (\Theta^k_1 \cup \Theta^{k-1}))$,
cf. Remark \ref{rmk:SigmaTheta}:
\begin{corollary} \label{cor:pperiod} 
For $u$ $\in \kappa^{-1}(\Theta^k \setminus (\Theta^k_1 \cup \Theta^{k-1}))$,
$\ella$
($=2\omega'\ella'+2\omega''\ella''$) $\in\Pi$, and $J_{k} \in \cI_k$, 
we have
\begin{equation}
	\sigma_{J_{k}}(u + \ella) = 
	\sigma_{J_{k}}(u) \exp(L(u+\frac{1}{2}\ella, \ella)) \chi(\ella).
\label{eq:trans_natural}
\end{equation}
\end{corollary}

\begin{proof}
After we apply the differential operators
$\displaystyle{\prod_{i \in I_k} 
         \frac{\partial}{\partial u_i}}$ on 
 both sides of the equality in Proposition \ref{prop:pperiod}, 
we restrict the domain to 
$\kappa^{-1}(\Theta^k \setminus (\Theta^k_1 \cup \Theta^{k-1}))$.
Then by
 Theorem \ref{vanishingTh}, the terms containing the
 lower-order derivatives of $\sigma$ vanish and the equality follows.
\end{proof}
}}
\end{remark} 

The former part of Theorem  \ref{vanishingTh}
is the same as 
Riemann's singularity theorem in Theorem \ref{thm:RST}.
The latter part, which gives 
a specific subset of $\mathcal{I}_k$,
is new and we will show it as follows. 

\smallskip


\begin{lemma} \label{lm:5.18}
For $g-r-1 \le k \le g-1$, $\cI_{k} = \{\{1\},\{2\},\ldots,\{k+1\}\}$.
\end{lemma}

\begin{proof}
Given that $\sigma$ is even or odd, the analysis of 
$u^{[k]} \in \kappa^{-1}(\Theta^k\setminus (\Theta^k_1\cup\Theta^{k-1} ))$ 
is essentially reduced to that of
$u^{[k]} \in \kappa^{-1}(\WW^k\setminus (\WW_1^k\cup\WW^{k-1}))$. 
We consider  $u = u^{[g-1]}+v$ $
 \in \kappa^{-1}(\Theta^{g}\setminus (\Theta^g_1\cup\Theta^{g-1} ))$ where 
$u^{[g-1]}\in \kappa^{-1}(\Theta^{g-1} \setminus 
(\Theta^{g-1}_1\cup\Theta^{g-2} ))$.
By Theorem \ref{thm:RST} and Corollary \ref{cor:6.7}, $n_k=1$
and there exists $j$ such that $\sigma_j(u^{[g-1]})$ is not identically
zero. 
 From Theorem \ref{thm:JIF},
$\sigma_i (u^{[g-1]})
= (-1)^{g-i}\mu_{g-1, i-1} (u^{[g-1]}) \sigma_{g} (u^{[g-1]}) $
and thus $\sigma_j (u^{[g-1]})$ does not vanish identically for
$j = 1, \ldots, g$.
Similarly for $g - r- 1 \le k \le g-1$,
$N_{k-1} = \mwdeg(u_{k})$ for $k=g-r-1, \ldots, g-1$.
Thus
for $u = u^{[k-1]}+v \in
 \kappa^{-1}(\Theta^{k}\setminus (\Theta^{k}_1\cup\Theta^{k-1} ))$
and $u^{[k-1]} \in \kappa^{-1}(\Theta^{k-1}\setminus 
(\Theta^{k-1}_1\cup\Theta^{k-2} ))$,
we conclude that
$\sigma_i(u^{[k-1]})$ does not identically equal zero $(i = 1, \ldots, k)$ 
because
$\sigma_i (u^{[k-1]})
= (-1)^{k-i+1}\mu_{k-1, i-1} (u^{[k-1]}) \sigma_{k} (u^{[k-1]})$.
\end{proof}

\begin{lemma} \label{lm:5.19}
For $k < g$,
$\cI_k$ contains an element $I_k$ for which $\ell \in I_k$,
$\ell = 1, 2, \ldots, k +1$.
If a finite sequence  $I_k$ consists only of elements of $\{k+2, \ldots, g\}$
and $\# I_k = n_k$, then it does not belong to $\cI_k$.
\end{lemma}

\begin{proof}
The statement is obvious for $g-r \le k$. 
We thus consider $k < g-r$ and $n_k \ge 2$.
Let us assume that every $I_k \in \cI_k$ doesn't contain $k+1$.
Let $u^{[k]} = u^{[k-1]}+ v^{(k)} \in \kappa^{-1}$
$(\Theta^{k}\setminus (\Theta^{k}_1\cup\Theta^{k-1}))$.
The assumption means that for every $J_k \in \Index$ 
such that $\#J_k = n_k -1$,
$ \sigma_{\{k+1\}\cup J_k}(u^{[k-1]}) $ vanishes.
Since L'Hospital's theorem and Theorem \ref{thm:JIF} show
\begin{gather}
\begin{split}
\sigma_{\{i\}\cup J_k} (u^{[k]})
&= (-1)^{k-i+1}\mu_{k, i-1} (u^{[k]}) 
\sigma_{\{k+1\}\cup J_k} (u^{[k-1]}),
\quad\mbox{for } i \le k \\
\sigma_{\{i\}\cup J_k} (u^{[k]})
&= 0 \times \sigma_{\{k+1\}\cup J_k} (u^{[k]}),
\quad\mbox{for } i >k, \\
\end{split}
\end{gather}
every $\sigma_{\{i\}\cup J_k} (u^{[k]})$ vanishes for every
$i = 1, 2, \ldots, g$.
This  contradicts  Theorem \ref{thm:RST}.
Thus $\sigma_{\{k+1\}\cup J_k} (u^{[k-1]})$ cannot vanish identically
and the statements are proved.
\end{proof}

\begin{lemma} \label{lm:natural}
For $k < g$,
the sets $\natural_k$ and
$\natural_k^{(i)}$ for $i=1, 2, \ldots, k$ belong to  $\cI_k$.
\end{lemma}

\begin{proof}
Corollary \ref{cor:Schur} shows that 
$\# \natural_k = n_k$ and $\mwdeg \natural_k = N_k$.
L'Hospital's theorem and Theorem \ref{thm:JIF} show
that $\natural_k^{(i)}$ for $i=1, 2, \cdots, k$ 
 belongs to  $\cI_k$ if $\natural_k$ does.
Proposition \ref{prop:Nakayashiki},
the expansion (\ref{eq:v_t}), 
 and Lemma \ref{lemma:schurLambda} (2) imply that,
as a function over 
$\kappa^{-1}(\Theta^k \setminus (\Theta^k_1 \cup \Theta^{k-1})$,
\begin{equation}
	\sigma_{\natural_k}\neq 0,
\end{equation}
hence $\natural_k$ is an element of $\cI_k$.
Clearly $\mwdeg(\natural_k^{(i)}) \ge N_k$.
\end{proof}

\bigskip
As a consequence of Proposition \ref{prop:Nakayashiki},
\begin{corollary} \label{coro:Nakayashiki}
For $u^{[k]} \in \kappa^{-1}(\WW^k\setminus (\WW_1^k\cup\WW^{k-1}))$,
the expansion of $\sigma_{\natural_k}(u^{[k]})$ at the origin takes the form
$$
   \sigma_{\natural_k}(u^{[k]}) = S_{\Lambda^{(k)}}(T)|_{T_{\Lambda_i + g - i}
   = u^{[k]}_i}  
  + \sum_{|\wg(\alpha)|>|\Lambda| } c^{[k]}_\alpha\cdot (u^{[k]})^\alpha
$$ 
where $c_\alpha\in \QQ[\lambda_j]$ and $S_{\Lambda^{(a)}} (T)$ 
is the lowest-order term in the w-degree of the $u^{[k]}_i$; 
$\sigma_{\natural_k}(u^{[k]})$ is homogeneous of degree  $|\Lambda^{(a)}|$ 
 with respect to the $\lambda$-degrees.
\end{corollary}
\bigskip

\begin{remark}{\rm{
For example, by letting
$J_i= \natural_k \setminus \{k + 1\} \bigcup \{i\}$ for
$i = k+2, k+3, \ldots, g$
and $u\in \kappa^{-1}(\Theta^k \setminus \Theta^k_1)$,
\begin{equation}
	\sigma_{J_i}(u) = 0,
\end{equation}
due to Theorem \ref{thm:JIF}.

Here we note that there is an element in
$\cI_k\setminus$
$\{ \natural_k, \natural_k^{(k)},
 \natural_k^{(k-1)}, \ldots, \natural_k^{(2)}, \natural_k^{(1)}\}$.
Some  examples are reported in \cite{MO},
where there is given an element $I_k$ 
($\# I_k = n_k$, $\mwdeg I_k = N_k$)
of $\Index$ which differs from $\natural_k$ and satisfies
$\sigma_{I_k}\neq 0$
as a function over $\kappa^{-1}(\Theta^k \setminus (\Theta^k_1\cup
\Theta^{k-1}))$.
}}
\end{remark}

\bigskip

Theorem \ref{vanishingTh} follows from the Lemmas above. 

We can state a stronger version  of Theorem \ref{thm:JIF} 3.
\begin{theorem}\label{algebraic}
For $k < g$,
$(P_1, \ldots, P_k) \in \SSS^k(X\backslash \infty ) \setminus (\SSS^k_1(X)
\cap\SSS^k(X\backslash \infty ))$ and
$u = \pm w(P_1, \ldots, P_k)\in \kappa^{-1}(\Theta^k)$,
$$
\frac{ \sigma_{\natural_k^{(i)}}(u) }{\sigma_{\natural_k}(u)}
      = (-1)^{k-i+1} \mu_{k, i-1} (P_1, \ldots, P_k) .
$$
\end{theorem}

\bigskip
Note that neither denominator or numerator 
in the left-hand side  vanish.

\bigskip
We also state another  version  of 
Theorem \ref{thm:JIF} 3.
besides Theorem \ref{algebraic}, as follows.

In the proof of Lemma \ref{lemma:schurLambda}, we note that
that $h_{a_i+b_i+1}^{(g;k)}
=\cdots + \frac{1}{(a_i + b_i + 1)!}
(T_1^{(g;k)})^{a_i + b_i + 1}+\cdots$ which
appears in (\ref{eq:hab1}) for
the partition characteristics 
$(a_1, \ldots, a_{n_k}; b_1, \ldots, b_{n_k})$  
of $\Lambda^{[k]}$.
We have the following Lemma:
\begin{lemma} \label{lmm:ggggg}
By using the notations 
$(u_i^{[k]} := T^{(k)}_{\Lambda_i+g-i}$,
 $u_i^{[g]} := T^{(g)}_{\Lambda_i+g-i})$
in Lemma \ref{lemma:schurLambda}, 
for a subsequence 
$J_\ell = \left\{L^{[k]}_{(a_\ell,b_\ell)}, 
L^{[k]}_{(a_{\ell+1},b_{\ell+1})}, \ldots,
L^{[k]}_{(a_{n_k},b_{n_k})}\right\} \subset \natural_k$
$(\ell \le n_k+1)$
using the characteristics of the partition 
$(a_1, a_2, \ldots, a_{n_k}; b_1, b_2, \ldots, b_{n_k})$
of $\Lambda^{[k]}$, we have
\begin{equation}
\bs_{\Lambda^{(k)}}(u^{[k]}) = 
\varepsilon_{\Lambda,J_\ell, \natural_k}'
\left(\frac{\partial}{\partial u_g^{[g]}}\right)^{
\mwdeg(\natural_k\setminus J_\ell)} 
\left(\prod_{i \in J_\ell}\frac{\partial}{\partial u_{i}^{[g]}}\right)
\bs_{\Lambda^{(g)}}(u^{[g]})
\Bigr|_{u^{[g]}=u^{[k]}},
\label{eq:gggggg}
\end{equation}
where
$\varepsilon_{\Lambda,J_\ell, \natural_k}'$ is a certain 
non-vanishing rational number. (Note that $J_{n_k+1} = \emptyset$.)
\end{lemma}

\begin{proof}
With notation as  in the proof of 
Lemma \ref{lemma:schurLambda}, 
using the characteristics of the partition,
we introduce the Young diagrams 
 $\Lambda^{[k,i]}$, $(i=1, \ldots, n_k)$,
given by $(a_1, \ldots, a_{i}; b_1, \ldots, b_{i})$  
such that $\Lambda^{[k,n_k]} \equiv \Lambda^{[k]}$.
Then by letting $t^{(g;k)} := (t_{k+1}, \ldots, t_g)$, the determinant
of the matrix in the proof of Lemma \ref{lemma:schurLambda},  $H_j$,
equals the Schur function of $\Lambda^{[k,i]}$,
$$
s_{\Lambda^{[k,i]}}(t^{(g;k)}) = |H_{n_k - i +1}|.
$$
Since we are concerned only with the term of
$(\frac{1}{n!}T_1^{(g;k)})^n$ in  $h_n^{(g;k)} =$
$\cdots + \frac{1}{n!}(T_1^{(g;k)})^n + \cdots$,
we analyze the behavior of $T_1$. 

>From the Jacobi-Trudi determinant expression in Proposition \ref{prop:SchurC},
for a Young diagram $\Lambda$, $s_{\Lambda}$ contains the term
$$
S_{\Lambda, T_1}:=
\left|\frac{1}{(\Lambda_i + j -i)!} T_1^{\Lambda_i + j -i}
 \delta_{ (\Lambda_i + j -i)} \right|,
$$
where 
 $\delta_{i} =1$ for $i \ge 0$ and ,
 $\delta_{i} =0$ for $i < 0$.
Then the term $S_{\Lambda, T_1}$, i.e., the determinant of
the matrix
$\left(\frac{1}{(\Lambda_i + j -i)!} T_1^{\Lambda_i + j -i}\right)$,
 does not vanish as a polynomial of $T_1$ because  
 the vectors
$U_i := \left(\frac{1}{(\Lambda_i + j -i)!} T_1^{\Lambda_i + j -i}\right)$,
viewed as columns, are independent as their coordinates show.
By considering their weight,
$S_{\Lambda, T_1}$ is a monomial  $T_1^{|\Lambda|}$
with a non-vanishing rational factor.

On the other hand, 
(\ref{eq:hab1}) shows that $\tilde \bs_{\Lambda^{[k]}}(t^{(g;k)})$ has
a term $\left(u_g^{[g;k]}\right)^{\mwdeg(\natural_k\setminus J_\ell)} 
\prod_{i \in J_\ell} u_{i}^{[g;k]}$ up to a non-vanishing rational factor,
where  $u^{[g;k]} := u^{[g]} - u^{[k]}$. 
The claim follows.
\end{proof}

For $u^{[k]} \in \kappa^{-1}(\Theta^k)$ and $u^{[g]} \in \CC^g$,
we introduce the following notation as an extension of Definition
 \ref{def:sigmaIk},
$$
\sigma_{J,g^N} (u^{[g]}):=
\left(\frac{\partial}{\partial u_g^{[g]}}\right)^{N}
\left(\prod_{i \in J}\frac{\partial}{\partial u_{i}^{[g]}}\right)
\sigma(u^{[g]}),
$$
and $\sigma_{J,g^N} (u^{[k]}):=$
 $\sigma_{J,g^N} (u^{[g]}) \Bigr|_{u^{[g]}=u^{[k]}}$.

\bigskip
We have now the variant proposition:

\begin{proposition} \label{prop:gggggk}
For $k < g$,
$(P_1, \ldots, P_k) \in \SSS^k(X\backslash\infty) \setminus (\SSS^k_1(X)
\cap\SSS^k(X\backslash\infty))$,
$u = \pm w(P_1, \ldots, P_k)\in \kappa^{-1}(\Theta^k)$,
subsequences
$J_\ell = \left\{L^{[k]}_{(a_\ell,b_\ell)}, 
L^{[k]}_{(a_{\ell+1},b_{\ell+1})}, \ldots,
L^{[k]}_{(a_{n_k},b_{n_k})}\right\} \subset \natural_k$ 
$(\ell \le n_k+1)$
using the characteristics of the partition 
$(a_1, a_2, \ldots, a_{n_k}; b_1, b_2, \ldots, b_{n_k})$
of $\Lambda^{[k]}$, 
and $J_\ell^{(i)} := J_\ell\setminus \{k+1\} \cup \{i\}$ $(i=1,2,\ldots, k)$,
the following relations hold:
\begin{enumerate}
\item
For $\ell \le n_k$,
$$
\frac{
\sigma_{J_\ell^{(i)}, g^{\mwdeg(\natural_k\setminus J_\ell)}}(u)
 }{
\sigma_{J_\ell, g^{\mwdeg(\natural_k\setminus J_\ell)}}(u)
 } = (-1)^{k-i+1} \mu_{k, i-1} (P_1, \ldots, P_k),
$$
especially,
$$
\frac{ \sigma_{i+1,g^{N_k - \mwdeg(k+1)}}(u)
 }{
\sigma_{k+1,g^{N_k - \mwdeg(k+1)}}(u)
 } = (-1)^{k-i+1} \mu_{k, i-1} (P_1, \ldots, P_k). 
$$

\item
For $\ell \le n_k$,
we have
as a function over 
$\kappa^{-1}(\Theta^k \setminus (\Theta^k_1 \cup \Theta^{k-1}))$,
$$
\sigma_{J_\ell^{(i)}, g^{\mwdeg(\natural_k\setminus J_\ell)} }
 \neq 0, \quad
\sigma_{J, g^{N'} } = 0,
$$
where
$0\le N' \le \mwdeg(\natural_k\setminus J_\ell)$ and 
$J \subset J_\ell^{(i)}$ such that $\# J + N' < \# J_\ell^{(i)} + 
\mwdeg(\natural_k\setminus J_\ell)$,
and $\ella\in\Pi$,
$$
\sigma_{J_\ell^{(i)}, g^{\mwdeg(\natural_k\setminus J_\ell)}}
(u + \ella) = 
\sigma_{J_\ell^{(i)}, g^{\mwdeg(\natural_k\setminus J_\ell)}} (u) 
\exp(L(u+\frac{1}{2}\ella, \ella)) \chi(\ella).
$$

\item
For every $\ell = 1, 2, \ldots, n_k$,
$$
\sigma_{J_\ell^{(i)}, g^{\mwdeg(\natural_k\setminus J_\ell)}} (u)
= \epsilon_{k,J_\ell} \sigma_{\natural_k^{(i)}}(u), \quad
i = 1, \ldots, k+1, \quad
$$
$$
\sigma_{g^{N_k}} (u)
= \epsilon_{k} \sigma_{\natural_k}(u), \quad
$$
where $\epsilon_{k,J_\ell}$ and $\epsilon_k$
are  non-vanishing rational numbers.

\end{enumerate}
\end{proposition}

Proposition \ref{prop:gggggk} 3. in the  hyperelliptic case is given 
 in the  work \cite{EHKJLP}.

\begin{proof}
We introduce the following objects and notation.
As usual, $u^{[k]} 
 \in \kappa^{-1}(w(\SSS^k(X\backslash\infty) \setminus (\SSS^k_1(X)
\cap\SSS^k(X\backslash\infty)))$ is a $g-$vector; for
$(P_{1}, \ldots, P_g) \in \SSS^{g}(X)$,  
 $v^{(i)}:=w(P_i)$ for $(i=1, 2, \ldots, g)$, 
 $u^{[\ell']} := \sum_{i=1}^{\ell'} v^{(i)}$, and 
 $u^{[g;\ell']} := \sum_{i=\ell'+1}^g v^{(i)}$,
 $ (\ell'=1, \ldots, g)$. 
Further we introduce the non-negative integer
$\hat N_{k, J_\ell} :=\mwdeg(\natural_k\setminus J_\ell)$.

For a sequence $J$ consisting of
$\{k+1, k+2,\ldots, g\}$,
Lemma \ref{lemma:d_uinfty},
which is essentially the same as (\ref{eq:pug=pugk}), gives
\begin{equation}
\begin{split}
\left(\frac{\partial}{\partial u_g^{[g]}}\right)^{N}
\left(\prod_{i \in J}\frac{\partial}{\partial u_{i}^{[g]}}\right)
\sigma(u^{[g]})\Bigr|_{u^{[g]}=u^{[k]}}
&=
\left(\frac{\partial}{\partial u_g^{[g;k]}}\right)^{N}
\left(\prod_{i \in J}\frac{\partial}{\partial u_{i}^{[g;k]}}\right)
\sigma(u^{[g]})\Bigr|_{u^{[g;k]}=0}\\
&\qquad + \mbox{lower order differentials of }\sigma(u^{[k]}).
\label{eq:diff_remainder}
\end{split}
\end{equation}
>From the proofs of Lemma
\ref{lemma:schurLambda} and Lemma \ref{lmm:ggggg},
 at $u^{[k]} \in \Theta^{[k]}$, $\sigma(u^{[g]})$ behaves like
\begin{equation}
\begin{split}
 \sigma(u^{[g]})&=
\left(
\prod_{i = 1}^{n_k} 
h^{(g;k)}_{a_i+ b_i+1}
\right)_{T^{(g;k)}_{\Lambda_i+g-i}=u^{[g;k]}} 
\left(
S_{\Lambda^{(k)}}(T^{(k)})|_{T^{(k)}_{\Lambda_i+g-i}=u^{[k]}}+ 
\sum_{|\alpha|>\Lambda^{(k)}} a_\alpha\cdot (u^{[k]})^\alpha
\right)\\
&\times \left(1 + d_{>}(v^{(k+1)}_{g}, \ldots,v^{(g)}_{g})\right) 
+ \xi_{k}(u^{[g]}),
\end{split}
\label{eq:sJS}
\end{equation}
where $\xi_{k}(u^{[g]})$ 
represents the terms which does not contain
$\left( \prod_{i = 1}^{n_k} h^{(g;k)}_{a_i+ b_i+1} \right)$,
and $a_\alpha \in \QQ[\lambda_i]$.
Then we obviously have
\begin{equation}
\sigma_{J_\ell,g^{\hat N_{k, J_\ell}}}(u^{[k]})
= \epsilon'
S_{\Lambda^{(k)}}(T^{(k)})|_{T^{(k)}_{\Lambda_i+g-i}=u^{[k]}}+ 
\sum_{|\alpha|>\Lambda^{(k)}} a_\alpha\cdot (u^{[k]})^\alpha,
\label{eq:sJNakayashiki}
\end{equation}
where the weight of the first term is $|\Lambda^{(k)}|$.

\smallskip
Since the $J_\ell=\natural_k$ ($k = 1, \ldots, g-2$)
case is the same as Theorem  \ref{vanishingTh},
we consider only the $J_\ell\neq \natural_k$ 
($k = 1, \ldots, g-2$) case. 

\smallskip
We assume that  $\ell \le n_k$.
Since $k \ge g- r$ is also obvious, we consider only 
$k < g- r$. 
Due to Riemann's singularity theorem (Theorem \ref{thm:RST}), 
for $J_\ell\neq \natural_k$,
$\sigma_{J_\ell} (u^{[k]}) = 0$ for $k< g- r$.
On the other hand, since the relation (\ref{eq:sJNakayashiki})
is not identically zero,
there may exist a positive integer 
$N_{k, J_\ell} \le \hat N_{k, J_\ell}$ and
a subsequence $J_\ell' \subset J_\ell$
such that we have, as a function over 
$\kappa^{-1}(\Theta^k \setminus (\Theta^k_1 \cup \Theta^{k-1}))$,
\begin{equation}
\sigma_{J_\ell', g^{ N_{k,J_\ell} }} \neq 0,
 \quad 
\sigma_{J', g^{N'}} =0, 
\label{eq:ssggggA}
\end{equation}
where
$0 \le N' \le  N_{k,J_\ell}$,
$J' \subset J_\ell'$ such that $\# J' + N' < 
\# J_\ell' + N_{k,J_\ell}$.
We show that
such $N_{k, J_\ell}$ and $J_\ell'$
are identical to $\hat N_{k, J_\ell}$ and
$J_\ell$ as follows.

 Assume that $\sigma_{J_\ell', g^{N_{k,J_\ell}}}  \neq 0$ for
some $N_{k, J_\ell} \le \hat N_{k,J_\ell}$,
$J_\ell' \subset J_\ell$ such that 
$ \# J_\ell' + N_{k,J_\ell} \le \# J_\ell + \hat N_{k,J_\ell}$.
>From (\ref{eq:diff_remainder}), we have
\begin{equation*}
\begin{split}
\sigma(u^{[g]}) &= 
\tilde \epsilon
\left(u_g^{[g;k]}\right)^{N_{k,J_\ell}}
\left(\prod_{i \in J_\ell'} u_{i}^{[g;k]}\right)
\sigma_{J_\ell',g^{N_{k,J_\ell}}} (u^{[k]})
\left(1 
+ d_{>}(v^{(k+1)}_{g}, \ldots,v^{(g)}_{g})\right) \\
 &+ \mbox{remainder},
\label{eq:sigmaJell}
\end{split}
\end{equation*}
where $\tilde \epsilon$ is a rational number.
Thus $\sigma(u^{[k]} + t v^{(k+1)})$ is given by 
\begin{equation*}
\begin{split}
\tilde \epsilon&
\left(t v_g^{(k+1)}\right)^{N_{k,J_\ell}}
\left(\prod_{i \in J_\ell'} t v_{i}^{(k+1)}\right)
\sigma_{J_\ell',g^{N_{k,J_\ell}}} (u^{[k]})
\left(1 
+ d_{>}(t v^{(k+1)}_{g}, 0,\ldots,0)\right) \\
& + \mbox{remainder}.
\label{eq:sigmaJellk}
\end{split}
\end{equation*}
>From Lemma \ref{lm:degu},
$\displaystyle{
\left(\frac{\partial^{N_{k,J_\ell}+\mwdeg(J_\ell')}
}{{\partial v^{(k+1)}_{g}}^{N_{k,J_\ell}+\mwdeg(J_\ell')}}
\right)
}$
$\sigma(u^{[k]} + t v^{(k+1)})\Bigr|_{v^{(k+1)}=0} $ is not identically zero.
If we assume that  $N_{k, J_\ell} + \mwdeg(J_\ell') < N_{k}$,
it contradicts Corollary \ref{cor:Fay} 1.;
since the vanishing order $N_k$ of $\sigma$
in Corollary \ref{cor:Fay}
 agrees with the weight of
the inflection at $\Theta^k$ \cite{BV},
we conclude that $N_{k, J_\ell}$ and $J_\ell'$  must be 
$\hat N_{k, J_\ell}$ and $J_\ell$ respectively 
and (\ref{eq:ssggggA}) must hold.

>From Proposition \ref{prop:pperiod},
we have
the translation formula for $\ella\in\Pi$,
\begin{equation}
\sigma_{J_\ell, g^{\hat N_{k,J_\ella}}}(u^{[k]} + \ella) = 
\sigma_{J_\ell, g^{\hat N_{k,J_\ella}}}(u^{[k]}) 
\exp(L(u^{[k]}+\frac{1}{2}\ella, \ella)) \chi(\ella),
\label{eq:trans_ggg}
\end{equation}
as in Corollary \ref{cor:pperiod}.

\smallskip
In the $J_\ell \neq \emptyset$ case, or $\ell \le n_k$
for $i =1, \cdots, k$, we have
$$
\left(\frac{\partial}{\partial u_g^{[g]}}\right)^{\hat N_{k,J_\ell}} 
\left(
\prod_{j=J_\ell \setminus \{k+1\}}
\frac{\partial}{\partial u_j^{[g]}}\right) 
\left((\sigma_{k+1}(u^{[g]})\cdot \mu_{k,i-1}(u^{[g]})\right) 
  \Bigr|_{u^{[g]}=u^{[k]}}
$$ $$
=\sigma_{J_\ell,g^{\hat N_k, J_\ell}}(u^{[g]})\cdot \mu_{k,i-1}(u^{[g]})) 
  \Bigr|_{u^{[g]}=u^{[k]}},
$$
because $\mu_{k,i-1}$ does not vanish for 
$(P_1, \ldots, P_k) \in \SSS^k(X\backslash \infty ) \setminus (\SSS^k_1(X)
\cap\SSS^k(X\backslash \infty ))$
and does not diverge due to the assumption, whereas
we have the vanishing property (\ref{eq:ssggggA})  when $u^{[g;k]} =0$.
As a consequence,  for every $i =  1, \ldots, k+1$,
$$
\sigma_{J_\ell^{(i)},g^{\hat N_k, J_\ell}}(u^{[k]}) 
=(-1)^{k-1+1}
\sigma_{J_\ell,g^{\hat N_k, J_\ell}}(u^{[k]})\cdot \mu_{k,i-1}(u^{[k]}).
$$
As a function over 
$\kappa^{-1}(\Theta^k \setminus (\Theta^k_1 \cup \Theta^{k-1}))$,
we have
\begin{equation*}
\sigma_{J_\ell^{(i)}, g^{ \hat N_{k,J_\ell} }} \neq 0,
 \quad 
\sigma_{J, g^{N'}} =0, 
\end{equation*}
where
$0 \le N' \le \hat N_{k,J_\ell}$,
$J \subset J_\ell^{(i)}$ such that $\# J + N' < \# J_\ell +\hat N_{k,J_\ell}$,
and for $\ell_a \in \Pi$,
$$
\sigma_{J_\ell^{(i)}, g^{\hat N_{k,J_\ell}}}(u^{[k]} + \ella) = 
\sigma_{J_\ell^{(i)}, g^{\hat N_{k,J_\ell}}}(u^{[k]}) 
\exp(L(u+\frac{1}{2}\ella, \ella)) \chi(\ella).
$$
>From Theorem \ref{algebraic}
for every $i=1,\ldots,k$,
$\displaystyle{
\frac
{\sigma_{J_\ell, g^{\hat N_{k,J_\ell}}}(u^{[k]})} 
{\sigma_{\natural_k}(u^{[k]})}
}$ 
$=\displaystyle{
\frac
{\sigma_{J_\ell^{(i)}, g^{\hat N_{k,J_\ell}}}(u^{[k]})} 
{\sigma_{\natural_k^{(i)}}(u^{[k]})}
}$ 
is a meromorphic function over $\Theta^{k}$; we denote it by
$q_\ell(u^{[k]})$. 

\smallskip
Let the divisor of the meromorphic function
$q_\ell(u^{[k]})$ as a function of $w(P_k)$
be $\sum_j p_j^+ - \sum_j p_j^-$, where 
$\sum_j p_j^+$ and  $\sum_j p_j^-$ are effective divisors. 
We have
$\sigma_{J_\ell^{(i)}, g^{\hat N_{k,J_\ell}}}(u^{[k]}) $
$= q_\ell(u^{[k]})\sigma_{\natural_k^{(i)}}(u^{[k]}) $
for every $i = 1, 2, \cdots, k+1$,
and 
every $\sigma_{J_\ell^{(i)}, g^{\hat N_{k,J_\ell}}}(u^{[k]}) $
is an entire function over $\kappa^{-1}\Theta^{(k)}$ as
 a function of $v^{(k)}=w(P_k)$ for $u^{[k]}= v^{(k)} + u^{[k-1]}$.
Hence $\{p_j^-\}$ must be the preimage under the
Abel map of a subset of the common divisors of 
$\sigma_{J_\ell^{(i)}, g^{\hat N_{k,J_\ell}}}(u^{[k]})$ 
for every $i = 1, 2, \ldots, k+1$, so that
there exists $N_q \ge 0$ and  $p_j^- = \infty$ ($j=1, \ldots, N_q$).
Hence $q_\ell$ is expanded as
$$
q_\ell(u^{[k]}) =
q_\ell(v_g^{(k)}, u^{[k-1]}) 
= (v_g^{(k)})^{-N_q}\left(q_{\ell,0} +
    q_{\ell,1} v_g^{(k)} + q_{\ell,2} (v_g^{(k)})^2 + \cdots \right),
$$
where every coefficient
$q_{\ell,i}$ is a function of $u^{[k-1]} \in \Theta^{k-1}$.

\smallskip

Corollary \ref{coro:Nakayashiki} gives
the expansion of $\sigma_{\natural_k}$ and
$\sigma_{J_\ell,g^{\hat N_{k, J_\ell}}}$ is given by
(\ref{eq:sJNakayashiki}), whereas
we have $\sigma_{J_\ell,g^{\hat N_{k, J_\ell}}}(u^{[k]})
= q_\ell(u^{[k]}) \sigma_{\natural_k}(u^{[k]})$. 
Hence $q_\ell$ is a constant function and $N_q =0$.
In other words, there is  a non-vanishing rational number
 $\epsilon_{k,J_\ell}$ such that 
\begin{equation}
\sigma_{J_\ell^{(i)},g^{\hat N_{k, J_\ell}}}(u^{[k]})
= \epsilon_{k,J_\ell} \sigma_{\natural_k^{(i)}}(u^{[k]}),
\quad i = 1, 2, \ldots, k+1.
\label{eq:sJes}
\end{equation}

\smallskip
Now we consider the $J_{n_k+1} \equiv \emptyset$ case.
Corollary \ref{cor:Fay} 2. means that
\begin{equation}
\sigma_{g^{N_{k} }} \neq 0,
 \quad 
\sigma_{g^{N'}} =0, \quad N' \le N_{k}.
\end{equation}
Due to (\ref{eq:sJNakayashiki}), 
we have $\sigma_{g^{N_{k}}}(u^{[k]})
= \epsilon_{k} \sigma_{\natural_k}(u^{[k]})$
for a suitable $\epsilon_{k}$.
\end{proof}

\vspace{1cm}

\noindent
{Shigeki  \textsc{Matsutani}}\\
{8-21-1 Higashi-Linkan,\\
Minami-ku, Sagamihara 252-0311,\\
JAPAN\\
E-mail {rxb01142@nifty.com}}

\smallskip
\noindent
{Emma \textsc{Previato}}\\
{Department of Mathematics and Statistics\\
Boston University \\
Boston, MA 02215-2411,\\
U.S.A.\\
E-mail {ep@bu.edu}}

\label{finishpage}

\end{document}